\documentclass[12pt]{article}
\input epsf.tex
\usepackage{graphicx}
\usepackage{amsmath,amsthm,amsfonts,amscd,amssymb,comment,eucal,latexsym,mathrsfs}
\usepackage{stmaryrd}
\usepackage[all]{xy}

\usepackage{epsfig}

\usepackage[all]{xy}
\xyoption{poly}
\usepackage{fancyhdr}
\usepackage{wrapfig}
\usepackage{epsfig}

\theoremstyle{plain}
\newtheorem{thm}{Theorem}[section]
\newtheorem{prop}[thm]{Proposition}
\newtheorem{lem}[thm]{Lemma}
\newtheorem{cor}[thm]{Corollary}

\theoremstyle{definition}
\newtheorem{defn}{Definition}
\theoremstyle{remark}
\newtheorem{remark}{Remark}
\newtheorem{example}{Example}

\advance \topmargin by -\headheight
\advance \topmargin by -\headsep
\evensidemargin \oddsidemargin
  \def\C{{\mathbb{C}}}  \def\E{{\mathbb{E}}} \def\F{{\mathbb{F}}}  \def\H{{\mathbb{H}}}      \def\N{{\mathbb{N}}}    \def\R{{\mathbb{R}}}        \def\Z{{\mathbb{Z}}}

 \def\cB{{\mathcal{B}}} \def\cC{{\mathcal{C}}}   \def\cF{{\mathcal{F}}} \def\cG{{\mathcal{G}}} \def\cH{{\mathcal{H}}}   \def\cK{{\mathcal{K}}} \def\cL{{\mathcal{L}}} \def\cM{{\mathcal{M}}} \def\cN{{\mathcal{N}}} \def\cO{{\mathcal{O}}} \def\cP{{\mathcal{P}}}  \def\cR{{\mathcal{R}}} \def\cS{{\mathcal{S}}} \def\cT{{\mathcal{T}}}      

       \def\bcH{\overline{ {\mathcal{H}}}}

           \def\hmu{{\widehat{\mu}}}

  \def\sC{{\mathscr{C}}}     \def\sH{{\mathscr{H}}}                  \def\sZ{{\mathscr{Z}}}

\newcommand\Aut{\operatorname{Aut}}
\newcommand\APER{{\operatorname{APER}}}

\newcommand\diam{\operatorname{diam}}
\newcommand\dom{\operatorname{dom}}

\newcommand\ERG{\operatorname{ERG}}

\newcommand\Fix{{\operatorname{Fix}}}

\newcommand\Hull{\operatorname{Hull}}

\newcommand\Isom{\operatorname{Isom}}

\newcommand\LOXO{\operatorname{LOXO}}

\newcommand\PARA{\operatorname{PARA}}
\newcommand\Prob{\operatorname{Prob}}
\newcommand\Proj{\operatorname{Proj}}

\newcommand\Res{\operatorname{Res}}

\newcommand\rng{{\operatorname{rng}}}

\newcommand\supp{\operatorname{supp}}

\newcommand{\resto}{\upharpoonright}
\def\cc{{\curvearrowright}}
\newcommand{\lb}{\llbracket}
\newcommand{\rb}{\rrbracket}

\newcommand{\grph}{\operatorname{graph}}

\newcommand{\bsH}{\overline{\sH}}

\begin{document}
\title{Equivalence relations that act on bundles of hyperbolic spaces}
\author{Lewis Bowen\footnote{supported in part by NSF grant DMS-1500389, NSF CAREER Award DMS-0954606} \\ University of Texas at Austin}
%\author{University of Hawaii}
%\author[Lewis Bowen]{Lewis Bowen$\dagger$}
%\address{Department of Mathematics\\
%University of Hawai'i--Manoa\\
%} %one \address command per author
%\email{lpbowen@math.hawaii.edu}
%%\thanks{$\dagger$ Supported in part by NSF grants DMS-??.}
\maketitle

\begin{abstract}
Consider a measured equivalence relation acting on a bundle of hyperbolic metric spaces by isometries. We prove that every aperiodic hyperfinite subequivalence relation is contained in a {\em unique} maximal hyperfinite subequivalence relation. We classify elements of the full group according to their action on fields on boundary measures (extending earlier results of Kaimanovich), study the existence and residuality of different types of elements and obtain an analogue of Tits' alternative. 
\end{abstract}


\noindent
\tableofcontents

\section{Introduction}

The broad goal of this research is to generalize the theory of groups acting on hyperbolic spaces to measured equivalence relations. To explain this properly we introduce some notation, leaving details to later sections.

Let $(X,\mu)$ denote a standard probability space and $\cR \subset X \times X$ a discrete Borel equivalence relation. We require that $\mu$ is $\cR$-invariant which means that if $\phi:X \to X$ is any Borel isomorphism with graph contained in $\cR$, then $\phi_*\mu=\mu$. We let $[\cR]$, the {\bf full group}, denote the set of all such Borel isomorphisms up to equivalence (isomorphisms are equivalent if they agree $\mu$-almost everywhere). We usually  require that $\cR$ is ergodic which means that if $Y \subset X$ is any Borel set equal to a union of $\cR$-classes then $\mu(Y) \in \{0,1\}$. The triple $(X,\mu,\cR)$ is an ergodic {\bf discrete probability-measure-preserving  (pmp) equivalence relation}.

\begin{defn}[metric and Borel bundles]\label{defn:main0}
A {\bf Borel bundle over $X$} is a standard Borel space $B$ with a Borel surjection $\pi:B \to X$ called the {\bf bundle projection}. We let $B*B = \{ (y,z):~\pi(y)=\pi(z)\}$ denote the fiber product with its Borel structure inherited from the inclusion $B*B \subset B\times B$. For each $x\in X$, the {\bf fiber} over $x$ is the subset $\pi^{-1}(x)$. It is denoted by $B_x:=\pi^{-1}(x)$. A {\bf metric bundle over $X$} consists of a Borel bundle $\pi:B \to X$ with a Borel map $d:B*B \to [0,\infty)$ such that for each $x\in X$, $B_x$ equipped with the restriction $d\resto B_x \times B_x$ is a metric space. A {\bf section} is a map $\sigma:X \to B$ such that $\sigma(x) \in B_x$ for all $x$. A metric bundle is {\bf separable} if there exists a countable set $\{\sigma_i\}_{i\in \N}$ of Borel sections $\sigma_i:X \to B$ such that for every $x\in X$, the set $\{\sigma_i(x)\}_{i\in \N}$ is dense in $B_x$. If each $(B_x, d\resto B_x \times B_x)$ is a geodesic Gromov hyperbolic space (the definition of which is in \S \ref{sec:hyperbolic}
 below) then we say $B$ is a {\bf bundle of hyperbolic spaces}. 
\end{defn}

\begin{defn}\label{defn:main}
Let $\pi:\sH \to X$ denote a bundle of hyperbolic spaces. An {\bf action of the equivalence relation $\cR$ on $\sH$ by isometries} consists of a family $\{\alpha(x,y):~(x,y) \in \cR\}$ of isometries $\alpha(x,y): \sH_y \to \sH_x$ satisfying:
\begin{itemize}
%\item (isometry) the restriction of $\alpha(x,y)$ to $\sH_y$ is an isometry onto $\sH_x$,
\item (cocycle condition) $\alpha(x,y)\alpha(y,z)=\alpha(x,z)$ for all $x\cR y \cR z$,
\item (Borel condition) $\{ (p,q) \in \sH \times \sH:~ \alpha(\pi(p),\pi(q))(q) = p\}$ is Borel.
\end{itemize}
The tuple $(\sH,d,\pi,\alpha)$ is an {\bf isometric action of $\cR$ on a bundle of hyperbolic spaces}. 
\end{defn}

%Let $\sH$ be a standard Borel space with a Borel surjection $\pi:\sH \to X$. We let $\sH*\sH = \{(p,q):~\pi(p)=\pi(q)\}$ denote the fiber product over $\pi$ and let $\sH_x =\pi^{-1}(x)$ denote the fiber over $x$. We assume there is a Borel map $d:\sH*\sH \to [0,\infty)$ such that for any $x\in X$, the restriction of $d$ to $\sH_x \times \sH_x$ is a distance function and $(\sH_x,d_x)$ is a complete geodesic Gromov hyperbolic metric space. We also assume that for every $(x,y) \in \cR$ there is an isometry $\alpha(x,y):\sH_y \to \sH_x$ satisfying
%\begin{itemize}
%\item (cocycle condition) $\alpha(x,y)\alpha(y,z)=\alpha(x,z)$;
%\item (Borel condition) $\{ (p,q) \in \sH \times \sH:~ \alpha(\pi(p),\pi(q))(q) = p\}$ is Borel.
%\end{itemize}
%We say that $(\sH,\pi,\alpha)$ is {\bf an isometric action of $\cR$ on a bundle of hyperbolic spaces}. 

\begin{example}\label{ex:graphing}
A {\bf graphing} is a Borel subset $\cG \subset \cR$ such that $\cR$ is the smallest equivalence relation containing $\cG$ and $\cG$ is symmetric (which means $(x,y) \in \cG \Rightarrow (y,x) \in \cG$). For each $x\in X$, we let $\cG_x$ denote the {\bf graph at $x$}. It has vertex set $[x]_\cR$ (the equivalence class of $x$) and edge set $E_x=\{ \{y,z\} :~ (y,z) \in \cG$ and $y,z \in [x]_\cR\}$. By abuse of notation we also consider $\cG_x$ to be a metric graph by assigning each edge length 1. We say $\cG$ is {\bf hyperbolic} if each $\cG_x$ is a Gromov hyperbolic metric space. In this case, let $\sH=\sqcup_x \cG_x$ be the disjoint union of the metric graphs. We consider $\sH$ to be a bundle with projection map that takes $\cG_x$ to $x$. It is a bundle of hyperbolic spaces where $d:\sH*\sH \to [0,\infty)$ is defined by setting $d(y,z)$ equal to the length of the shortest path from $y$ to $z$ in $\cG_x$ (if $\pi(y)=\pi(z)=x$). We define the action by setting $\alpha(x,y):\cG_y \to \cG_x$ equal to the natural identification. Equivalence relations with hyperbolic graphings were studied by Kaimanovich in \cite{kaimanovich-boundary-amenability}.
\end{example}

\begin{remark}
We prefer to work with isometric actions on bundles of hyperbolic spaces instead of hyperbolic graphings for the following reason: the class of equivalence relations that admit such actions is closed under taking subequivalence relations (because we can always restrict the action to the subequivalence relation).  By contrast, it is unknown whether the existence of a hyperbolic graphing is closed under taking subequivalence relations.
\end{remark}

\begin{defn}[Bundle of Gromov completions]
Let $(\sH,d,\pi,\alpha)$ be as in Definition \ref{defn:main}. For each $x\in X$, let $\partial \sH_x$ denote the Gromov boundary of $\sH_x$ and $\bsH_x = \sH_x \cup \partial \sH_x$ denote the Gromov completion. We also let 
$$\bsH= \sqcup_x \bsH_x,\quad \partial \sH = \sqcup_x \partial \sH_x \subset \bsH$$
 be the disjoint unions. We extend the projection map $\pi$ to $\bsH$ so that $\pi$ maps $\bsH_x$ to $x$ (for $x\in X$). In this manner, we consider $\bsH$ and $\partial \sH$ to be bundles over $X$. In \S \ref{sec:hyperbolic-bundle} we show that $\bsH$ and $\partial \sH$ are naturally endowed with Borel structures so that the inclusions $\sH\to \bsH, \partial \sH \to \bsH$ are Borel and the projection map $\pi:\bsH \to X$ is Borel. We also extend the action $\alpha$ as follows. Because $\alpha(x,y):\sH_y \to \sH_x$ is an isometry there is a unique extension, which we also denote by $\alpha(x,y):\bsH_y \to \bsH_x$ that is a homeomorphism. It satisfies the cocycle condition and the Borel condition of Definition \ref{defn:main}. These statements are proven in \S \ref{sec:hyperbolic-bundle}  below.
\end{defn}

\begin{defn}[Sections, orbits and limit sets]\label{defn:section}
Now let $(\sH,d,\pi,\alpha)$ be as in Definition \ref{defn:main}.  Recall that a section is a map $\sigma: X \to \sH$ such that $\sigma(x) \in \sH_x$ for a.e. $x$. This induces a map
$$\sigma_x:[x]_\cR \to \sH_x, \quad \sigma_x(y) := \alpha(x,y)\sigma(y).$$
The image of $\sigma_x$ is the {\bf orbit of $x$}, denoted $\cO_x^\sigma$. In order to ensure non-triviality, we require that $\sigma$ is {\bf metrically proper} which means: for every metric ball $\cB \subset \sH_x$, $\cO_x^\sigma \cap \cB$ is finite. Let $\cL^\sigma_x(\cR):= \overline{\cO_x^\sigma} \cap \partial \sH_x$ be the {\bf limit set} (where the closure of $\cO_x^\sigma$ is taken in $\bsH_x$). 

Many arguments rely on a map  $d_\sigma:\cR  \to [0,\infty)$ defined by
$$d_\sigma(x,y) = d(\sigma_z(x),\sigma_z(y))$$
for any $z\in [x]_\cR=[y]_\cR$. The cocycle property of the action implies that this definition of $d_\sigma$ is independent of the choice of $z$. This gives a pseudo-metric on each equivalence class $[x]_{\cR}$. It is a metric if $\sigma_x:[x]_\cR \to \sH_x$ is injective. In this case $\sigma_x$ is an isometric embedding of $[x]_\cR$ into $\sH_x$. 
\end{defn}

Our main results hold under the following hypotheses:
\begin{defn}[Main Assumption]
We say the {\bf Main Assumption} is satisfied if $(X,\mu)$ is a standard non-atomic probability space, $\cR \subset X\times X$ is a discrete Borel equivalence relation, $\mu$ is $\cR$-invariant and ergodic, $\pi:\sH \to X$ is a separable bundle of hyperbolic spaces with isometric action $\alpha$, $\sigma:X\to \sH$ is a metrically proper section and  for a.e. $x$, the closure $\overline{\cO^\sigma_x}$ of the orbit of $x$ is compact in $\bsH_x$. This compactness assumption is automatically satisfied if each fiber $\sH_x$ is locally compact. 
\end{defn} 
Our first result is that the limit sets are essentially independent of $\sigma$:
\begin{thm}\label{thm:limit-set}
If the Main Assumption is satisfied and $\sigma,\eta:X \to \sH$ are metrically proper sections whose orbit closures $\overline{\cO^\sigma_x}, \overline{\cO^\eta_x}$ are compact (for a.e. $x$) then $\cL^\sigma_x = \cL^\eta_x$ for a.e. $x$.
\end{thm}
Because of this theorem, we may write $\cL_x:=\cL^\sigma_x$.

%We now make the assumption that $\cL_x$ is compact for a.e. $x$. This is automatically true if each fiber $\sH_x$ is locally compact. 

\subsection{Treeable subequivalence relations}

Our main results concern the structure of subequivalence relations of $\cR$. More precisely, they concern two specific types of subequivalence relations: treeable and hyperfinite. Let us recall that an equivalence relation $\cR$ is {\bf treeable} if it admits a graphing $\cG$ (as in Example \ref{ex:graphing}) in which each local graph $\cG_x$ is a tree. Such equivalence relations are analogous to free groups in group theory and have been  studied intensively. For example treeings play a central role in \cite{gaboriau-cost} (treeings realize the cost). A group is {\bf treeable} if it admits an essentially free action whose orbit-equivalence relation is treeable. See \cite{gaboriau-2005} for many examples of treeable groups. If $\cR$ is ergodic, treeable and non-hyperfinite then there exists an essentially free ergodic action $\F_2 \cc X$ of the rank 2 free group such that each orbit of the action is contained in an $\cR$-class \cite[Proposition 14]{gaboriau-lyons}. Since trees are hyperbolic metric spaces, any treeable equivalence relation with finite cost satisfies the Main Assumption (with respect to the bundle defined in Example \ref{ex:graphing}). 

 Let us also remark on the von Neumann-Day problem in group theory and its analog in the theory of equivalence relations. This problem asked whether every non-amenable group necessarily contains a  non-amenable free group. It was disproven by Ol'shankii \cite{olshankii-book}. However, the analogous problem for equivalence relation, ``does every non-hyperfinite pmp equivalence relation contain a non-hyperfinite treeable subequivalence relation? '' remains open. A strong partial answer due to Gaboriau-Lyons \cite{gaboriau-lyons} states that the orbit equivalence relation of any Bernoulli shift action (with large enough base entropy) of a non-amenable group has this property. Our first main result is a positive answer to this question under hyperbolicity assumptions:

\begin{thm}\label{thm:tits}
If $\cR$ satisfies the Main Assumption and is non-hyperfinite then $\cR$ contains an ergodic non-hyperfinite treeable subequivalence relation. 
\end{thm}
Because the hypotheses on $\cR$ are inherited by subequivalence relations, the result above also holds for all subequivalence relations of $\cR$. Thus we may think of this as an analog of Tits' alternative for hyperbolic groups.

\subsection{Hyperfinite subequivalence relations}

Let us now recall that an equivalence relation $\cR$ is {\bf hyperfinite} if there exist Borel subequivalence relations $\cR_1 \le \cR_2 \le \cdots $ such that 
\begin{itemize}
\item for each $n$, all $\cR_n$-classes are finite;
\item $\cR=\cup_n \cR_n$.
\end{itemize}
Hyperfiniteness is analogous to amenability in group theory. In fact, it is equivalent to amenability in the theory of equivalence relations \cite{OW80, CFW81}. It is well-known that, in a hyperbolic group, any infinite amenable subgroup is contained in a unique maximal amenable subgroup (which must, in fact, be virtually cyclic). Moreover, normalizers of infinite amenable subgroups are necessarily amenable. Our next result shows that this phenomenon extends to equivalence relations:

\begin{thm}\label{thm:maximal}
If $\cR$ satisfies the Main Assumption and $\cS \le \cR$ is an aperiodic hyperfinite subequivalence relation then $\cS$ is contained in a {\em unique} maximal hyperfinite subequivalence relation. Moreover, the subequivalence relation generated by $\cS$ and the normalizer of $[\cS]$ in $[\cR]$ is hyperfinite.
%$N_\cR(\cS)=\{f\in [\cR]:~x\cS y \Rightarrow f(x)\cS f(y)\}$ is the normalizer of $\cS$ in $\cR$, then the subequivalence relation generated by $\cS$ and $N_\cR(\cS)$ is hyperfinite.
\end{thm}
\begin{remark}
It is easy to prove that every hyperfinite subequivalence relation is contained in some maximal hyperfinite subequivalence relation. However, uniqueness is not true in general. For example  let $G$ be any countable group with infinite amenable groups $H_1,H_2,H_3$ such that $\langle H_1,H_2\rangle$, $\langle H_2,H_3\rangle$ are amenable but $\langle H_1,H_2,H_3\rangle$ is non-amenable (for example this property is satisfied by $G=SL(3,\Z)$ with respect to its elementary subgroups). If $\cR$ denotes the orbit equivalence relation of an essentially free pmp $G \cc (X,\mu)$ and $\cS_{ij} \le \cR$ is the subequivalence relation generated by the $\langle H_i,H_j\rangle$-action then $\cS_{12}, \cS_{23}$ and $\cS_2=\cS_{12}\cap \cS_{23}$ are aperiodic and hyperfinite. But since the subequivalence relation generated by $\cS_{12}$ and $\cS_{23}$ is non-hyperfinite, there is more than one maximal hyperfinite subequivalence relation containing $\cS_2$. 
\end{remark}

\begin{remark}
The von Neumann algebra analog of Theorem \ref{thm:maximal} is open even in the case of free group factors (see the end of \cite{peterson-thom-group-cocycles} where this is stated as a conjecture). There are several constructions of maximal amenable subalgebras of von Neumann algebras with hyperbolic flavor \cite{CFRW-2010, houdayer-2014, shen-2006}. Recently (and independently of this research) R. Boutonnet and A. Carderi \cite{boutonnet-carderi} have shown that if $H<G$ is a maximal amenable subgroup of a word hyperbolic group and $G \cc (X,\mu)$ any essentially free pmp action then the orbit equivalence relation of the $H$-action is a maximal hyperfinite subequivalence relation of the orbit-equivalence relation the $G$-action. 
\end{remark}

%The normalizer of $\cS$ is the largest subequivalence relation $\cS' \le \cR$ such that $\cS \le \cS'$ and $\cS$ is normal in $\cS'$ in the sense of \cite{FSZ89}.

\subsubsection{Parabolic and loxodromic subequivalence relations}
We now turn towards a more detailed picture of the hyperfinite subequivalence relations of $\cR$ analogous to the elliptic/parabolic/loxodromic classification of isometries of real hyperbolic space. This classification arises from considering the action of $\cR$ on fields of boundary measures. To be precise, a {\bf field of boundary measures} is an assignment $x\mapsto \nu_x$ where $\nu_x$ is a Borel probability measure on $\partial \sH_x$ and the assignment satisfies a certain Borel condition (see \S \ref{sec:fields}). We do not distinguish between fields that agree almost everywhere. The space $\Prob(\partial\sH \to X)$ of all Borel fields of boundary measures admits the structure of a compact convex subspace of a Banach space (Lemma \ref{lem:fields}). Moreover, the full group $[\cR]$ acts jointly continuously on $\Prob(\partial\sH \to X)$ by $(f\nu)_{fx} = \alpha(fx,x)_*\nu_x$. The subspace $\Prob(\cL \to X) \subset \Prob(\partial \sH \to X)$ of fields with $\nu_x$ supported on the limit set $\cL_x$ is a minimal set for this action (Corollary \ref{cor:minimal}) whenever $\cR$ is non-hyperfinite. By contrast, if $\cR$ is hyperfinite then $[\cR]$ is extremely amenable \cite{giordano-pestov-extreme-amenability} and therefore every minimal action is trivial.

Given a subequivalence relation $\cS \le \cR$, let $\Fix(\cS)$ be the set of $\eta\in \Prob(\partial\sH \to X)$ that are fixed by $\cS$ (so $\alpha(x,y)_*\eta_y=\eta_x$ for $(x,y)\in \cS$). It is essentially a result of Kaimanovich \cite{kaimanovich-boundary-amenability} that $\Fix(\cS)$ is nonempty if and only if $\cS$ is hyperfinite. Moreover, if $\cS$ is aperiodic then for any $\eta \in \Fix(\cS)$ the support of $\eta_x$ has cardinality at most 2. So we say $\cS$ is 
\begin{itemize}
\item {\bf parabolic} if there is a unique $\eta \in \Fix(\cS)$ and for a.e. $x$, the support of $\eta_x$ has cardinality 1;
\item {\bf loxodromic} if there exists $\eta \in \Fix(\cS)$ such that for a.e. $x$, the support of $\eta_x$ has cardinality 2;
\item {\bf mixed} if there is a nontrivial disjoint measurable partition $X=Y\sqcup Z$ such that $\cS \resto Y$ is parabolic and $\cS \resto Z$ is loxodromic where $\resto$ denotes ``restricted to''. Moreover this partition is unique up to null sets.
\end{itemize}
It follows from Kaimanovich's results that if $\cS\le \cR$ is hyperfinite and aperiodic then it is either parabolic, loxodromic or mixed. Moreover, if $\cS$ is ergodic then it must be either parabolic or loxodromic. Similarly, if $f \in [\cR]$ then we say that $f$ is parabolic/loxodromic/mixed if the subequivalence relation generated by $f$ is parabolic/loxodromic/mixed. Let $\APER$, $\PARA$,$\LOXO \subset [\cR]$ denote the subsets of aperiodic, parabolic and loxodromic elements. As shown in \cite[I.2]{Kechris-global-aspects}, $\APER$ is a $G_\delta$ subset of $[\cR]$; in particular, it is a Polish space. We prove
\begin{thm}\label{thm:classification}
If the Main Assumption is satisfied and $\cR$ is non-hyperfinite then both $\PARA$ and $\LOXO$ are nonempty. Moreover, $\PARA$ is a dense $G_\delta$ subset of $\APER$. 
On the other hand, if $\cR$ is hyperfinite then either $\APER=\PARA$ or $\APER=\LOXO$. 
\end{thm}
This is a surprising result; by contrast consider the isometry group $\Isom(\H^n)$ of real hyperbolic $n$-space. It is easy to show that the set of parabolic elements has positive codimension. Moreover, if $\Gamma$ is a countable Gromov hyperbolic group then $\Gamma$ does not have any parabolic elements at all.

\subsection{Other results}

\subsubsection{Svarc-Milner Lemma}
The Svarc-Milner Lemma is a fundamental result in geometric group theory. It states that if $\Gamma$ is a finitely generated group acting isometrically and properly discontinuously on a metric space $X$ and $x\in X$ then the map $\gamma \in \Gamma \mapsto \gamma x \in X$ is a quasi-isometric embedding with respect to any word metric on $\Gamma$. In particular, if $X$ is a Gromov hyperbolic space and $\Gamma$ acts isometrically and properly discontinuously with compact quotient then $\Gamma$ itself is Gromov hyperbolic. In Theorem \ref{thm:svarc-milner} we obtain the following analog: if the Main Assumption is satisfied and the section $\sigma:X\to \sH$ is $r$-cobounded (meaning: the open radius $r$ neighborhood of the orbit $O^\sigma_x$ in $\sH_x$ is all of $\sH_x$), then $\cR$ admits a hyperbolic graphing.

\subsubsection{Minimality}
We prove that if the Main Assumption is satisfied and $\cR$ is not hyperfinite then the action of $\cR$ on the bundle of limit sets $\cL$ is minimal in two different senses. First, suppose $K \subset \cL$ is a Borel subset such that for a.e. $x\in X$, $K_x = K \cap \cL_x$ is closed in $\cL_x$. Suppose also that $K$ is invariant in the sense that $\alpha(x,y)K_y=K_x$ for $(x,y) \in \cR$. Then either $K_x$ is empty for a.e. $x$ or $K_x=\cL_x$ for a.e. $x$. This is Theorem \ref{thm:minimal}. It plays a key role in proving the existence of loxodromic elements of $[\cR]$. Second, as mentioned above, the full group $[\cR]$ acts minimally on $\Prob(\cL \to X)$ the space of fields of boundary measures. This is Corollary \ref{cor:minimal}. 

\subsubsection{Limit sets}
If $G$ is a rank 1 simple Lie group then $g \in G$ is either elliptic, parabolic or loxodromic depending on whether it has 0, 1 or 2 fixed points on the boundary. These fixed points form the limit set of the subgroup $\langle g\rangle$. This observation leads one to guess that if $\cS \le \cR$ is ergodic and hyperfinite then it should be parabolic or loxodromic depending on whether for the limit set $\cL_x(\cS)$ has one or two elements for a.e. $x$. We prove in Lemma \ref{lem:loxodromic-case} below that indeed, if the Main Assumption is satisfied and if $\cS$ is loxodromic then $|\cL_x(\cS)|=2$ for a.e. $x$. The question remains open if $\cS$ is parabolic. By contrast, if the measure $\mu$ on $X$ is only required to be $\cR$-quasi-invariant instead of $\cR$-invariant then there are counterexamples (this was observed earlier by Kaimanovich \cite{kaimanovich-boundary-amenability}).

\subsubsection{Maximal hyperfinite subequivalence relations}
The result mentioned above is used to prove: if the Main Assumption is satisfied and $\cR$ is parabolic (in particular it is hyperfinite) then all of its aperiodic subequivalence relations are also parabolic. In turn this result is used to prove Theorem \ref{thm:maximal} that every aperiodic hyperfinite subequivalence relation $\cS \le \cR$ is contained in a {\em unique} maximal hyperfinite subequivalence relation $\cM$. Indeed, we obtain an explicit description of $\cM$ as the stabilizer for a canonical $\cS$-invariant field of boundary measures. 

\subsubsection{Rank 1}
Let us say that a measured equivalence relation $\cR$ has {\bf rank 1} if every aperiodic hyperfinite subequivalence relation $\cS\le \cR$ is contained in a {\em unique} maximal hyperfinite subequivalence relation. This definition is motivated by the theory of semisimple Lie groups: a semisimple Lie group $G$ of noncompact type has real rank 1 if and only if every closed noncompact unimodular amenable subgroup is contained in a unique maximal unimodular amenable subgroup.

In Theorem \ref{thm:norm} we prove that if $\cR$ has rank 1 and $\cS\le \cR$ is an aperiodic hyperfinite subequivalence relation then the subequivalence relation generated by $\cS$ and the normalizer of $[\cS]$ in $[\cR]$ is hyperfinite. (In fact, we prove a more general result using quasi-normalizers).

Also we provide examples of equivalence relations that are not rank 1: let $G=SL(n,\Z)$ ($n\ge 3$) or $G=H_1\times H_2$ where $H_1$ contains an infinite amenable subgroup and $H_2$ is non-amenable. Let $\cR$ be the orbit-equivalence of an essentially free ergodic pmp $G$-action. In Corollaries \ref{cor:SLnZ}, \ref{cor:product} we show that $\cR$ does not have rank 1. It follows that $\cR$ does not admit a hyperbolic graphing (as in Example \ref{ex:graphing}). The latter fact is a well-known result of Adams  \cite{adams-indecomp} that has been partially generalized and reproven in \cite{hjorth-nontreeability, pemantle-peres-2000, gaboriau-cost} (for example, \cite{gaboriau-cost} shows non-amenable groups with a cost 1 action are not treeable).

%oxodromic elements is a dense $G_\delta$ subset of $\Isom(\H^n)$ whereas

%Now suppose there is a standard Borel space $\sH$ and a Borel surjection $\pi:\sH \to X$ such that for each $x$, the fiber $\pi^{-1}(x) =:\sH_x$ is endowed with the structure of a complete geodesic Gromov hyperbolic space. More precisely, we let $\sH *\sH$ denote the fiber product $\sH * \sH = \{ (y,z) \in \sH\times \sH:~\pi(y)=\pi(z)\}$ and we require the existence of a Borel map $d:\sH*\sH \to [0,\infty)$ such that the restriction of $d$ to $\sH_x \times \sH_x$ is a distance function on $\sH_x$  that makes $\sH_x$ into a complete geodesic Gromov hyperbolic space. 

%Now suppose there are standard Borel spaces $\bsH=\sH \sqcup \partial \sH$ and a Borel surjection $\pi:\bsH \to X$ such that for each $x$, the fiber $\pi^{-1}(x) \cap \sH=:\sH_x$ is endowed with the structure of a complete geodesic Gromov hyperbolic space with Gromov boundary $\partial \sH_x=\pi^{-1}(x)\cap \partial \sH$. More precisely, we let $\sH *\sH$ be the fiber product $\sH * \sH = \{ (y,z) \in \sH\times \sH:~\pi(y)=\pi(z)\}$. Then we require the existence of a Borel map $d:\sH*\sH \to [0,\infty)$ such that the restriction of $d$ to $\sH_x \times \sH_x$ is a distance function on $\sH_x$  that makes $\sH_x$ into a complete geodesic Gromov hyperbolic space. 

\subsection{Related literature}

The paper owes a large debt to Kaimanovich's paper \cite{kaimanovich-boundary-amenability}. The latter studies amenable equivalence relations that admit hyperbolic graphings (as in Example \ref{ex:graphing}). However, the measure is only assumed to be quasi-invariant under the relation rather than invariant. Kaimanovich proves that any such equivalence relation necessarily has an invariant field of boundary measures $\nu$ such that the support of $\nu_x$ has cardinality at most 2. The proof strategy has roots in earlier work of Adams \cite{adams-indecomp} which shows that an orbit-equivalence relation of a pmp essentially free action of a non-elementary hyperbolic group cannot decompose as a nontrivial direct product of measured equivalence relations. Corollary \ref{cor:product} below partially generalizes this result. There is also a related result of Adams showing that any cocycle of a higher rank semisimple Lie group into the isometry group of a hyperbolic space must be cohomologous to a cocycle taking values in a compact subgroup \cite{adams-reduction}.

A recent paper \cite{anderegg-henry} of Anderegg and Henry studies isometric actions of equivalence relations on bundles of CAT(0) spaces. Unfortunately, I only learned about their paper when this paper was nearly finished. In addition to their results on isometric CAT(0) actions, it develops the general theory of fields of metric spaces over a Borel space much more thoroughly than is done in this paper. 

As mentioned above, this paper proves that if $\cR$ satisfies the Main Assumption and $\cS\le \cR$ is hyperfinite and aperiodic then the subequivalence relation $\cS'$ generated by $\cS$ and the normalizer of $[\cS]$ in $[\cR]$ is also hyperfinite. There is a well-known analogous result in the theory of von Neumann algebras due to Ozawa \cite{ozawa-solid}: If $\Gamma$ is an i.c.c. (infinite conjugacy classes) Gromov hyperbolic group, then its von Neumann algebra, denoted $L\Gamma$ is solid, i.e., $A'\cap L\Gamma$ is amenable for every diffuse von Neumann subalgebra $A \subset L\Gamma$. This result was strengthened in \cite{chifan-sinclair-I}.

\subsection{Organization}
Standard definitions and results on Gromov hyperbolic spaces and measured equivalence relations are relegated to Appendices \ref{sec:hyperbolic} and \ref{sec:mer} for easy reference. 

Theorem \ref{thm:maximal} (that every aperiodic hyperfinite subequivalence relation is contained in a {\em unique} maximal hyperfinite subequivalence relation) is proven in \S \ref{sec:maximal}. The proof uses all of the material in \S \ref{sec:hyperbolic-bundle} - \S \ref{sec:maximal} except for the Svarc-Milner Lemma (Theorem \ref{thm:svarc-milner}), the Minimality Theorem \ref{thm:minimal} and Theorem \ref{thm:jcontinuous1} on the topology of $\Prob(\cL \to X)$ (the space of fields of boundary measures). 

The proof of the Tit's Alternative (Theorem \ref{thm:tits}) is contained in \S \ref{sec:tits}. The proof relies on the existence of loxodromic elements (Theorem \ref{thm:exists}) which relies on the Minimality Theorem \ref{thm:minimal}. 

The existence of parabolic elements (Lemma \ref{lem:exists}) relies on the Tit's Alternative and a result of \cite{conley-brooks} on the existence of one-ended forests. The abundance of parabolic elements (Theorem \ref{thm:parabolic}) uses the topological structure of $\Prob(\cL \to X)$, proofs of which are in Appendix \ref{sec:fields}. The final result, that the full group $[\cR]$ acts minimally on $\Prob(\cL \to X)$ uses the existence of parabolic elements and everything in Appendix \ref{sec:fields}.

%The main body of the paper begins with \S \ref{sec:hyperbolic-bundle} which defines a Borel structure on the bundle of Gromov completions, proves a Svarc-Milner Lemma (Theorem \ref{thm:svarc}) and proves that the bundle of limit sets does not depend on the choice of metrically proper section (Theorem ). The latter result is used to obtain the First Minimality Theorem: that the equivalence relation $\cR$ acts minimally on the bundle $\cL$ of limit sets. The next section \S \ref{sec:fields1} introduces fields of boundary measures. The space of all such fields, denoted by $\Prob(\cL \to X)$, is affinely homeomorphic to a compact convex subspace of a dual Banach space (with respect to the weak* topology). Moreover, the full group $[\cR]$ acts jointly continuously on this space. These facts are stated in Theorem ** and proven (in greater generality) in Appendix \ref{sec:fields}. Next we show that the limit sets of a loxodromic equivalence relation contain exactly two points and use this to prove that every aperiodic subequivalence relation of a parabolic subequivalence relation is itself parabolic (Theorem ). A similar result is true for loxodromic equivalence relations, although this is much easier to prove. These facts are then used in \S ** to prove Theorem * that every aperiodic hyperfinite subequivalence relation is contained in a {\em unique} maximal hyperfinite subequivalence relation. 

{\bf Acknowledgements}. I am especially grateful to Sukhpreet Singh for many conversations that got this project started. Thanks also to Robin Tucker-Drob for  helpful discussions and to Damien Gaboriau for remarks and corrections.

\section{Bundles of hyperbolic spaces}\label{sec:hyperbolic-bundle}

Let $(X,\mu,\cR)$ be a discrete ergodic pmp equivalence relation, $(\sH,d,\pi)$ a separable bundle of hyperbolic spaces over $X$ and $\alpha$ an isometric action of $\cR$ on $(\sH,d,\pi)$.

\subsection{The boundary extension}

In this section we define a Borel structure on $\bsH$, the bundle of Gromov completions, and extend the action $\alpha$ to this bundle. Recall that each fiber $\sH_x$ may be regarded as a Gromov hyperbolic space with respect to the metric $d_x$ which denotes the restriction of $d$ to $\sH_x\times \sH_x$. Thus it has a Gromov boundary, denoted $\partial \sH_x$ and a Gromov completion, denoted $\bsH_x:=\sH_x \cup \partial \sH_x$. We let $\bsH = \sqcup_x \bsH_x$, $\partial \sH = \sqcup_x \partial \sH_x \subset \bsH$ denote the disjoint unions. We extend the projection map $\pi$ to $\bsH$ in the obvious way: $\pi(\xi) = x$ if $\xi \in \bsH_x$.

%Let $\Gromov: \bsH * \bsH \to [0,\infty]$ be the map
%$$\Gromov(\xi, \eta) = (\xi | \eta)_x$$
%where $x = \pi(\xi) = \pi(\eta)$ and $(\cdot | \cdot)_x$ denotes the Gromov product (as defined in \S --). 

Recall from Appendix \ref{sec:hyperbolic} that the {\bf Gromov product} of $p,q \in \sH_x$ with respect to $r \in \sH_x$ is defined by
$$(p|q)_r = (1/2)( d(p,r) + d(q,r) - d(p,q)).$$
It extends to the boundary by taking limits inferior (see Appendix \ref{sec:hyperbolic}). Now we choose the Borel structure on $\bsH$ to be the smallest one such that for any Borel sections $\sigma,\eta:X \to \sH$ and Borel set $B \subset [0,\infty)$ the subset
$$\Omega(\sigma,\eta,B):=\left\{ \xi \in \bsH:~ \pi(\xi)=x \Rightarrow ( \xi | \sigma(x))_{\eta(x)} \in B \right\}$$
is Borel. 

%The next two lemmas are exercises.

\begin{lem}
A section $\xi:X \to \bsH$ is Borel if and only if for every pair of Borel sections $\sigma,\eta:X \to \sH$ the map $F:X \to [0,\infty)$ defined by
$$F(x) =  (\xi(x)|\sigma(x))_{\eta(x)}$$
is Borel.
\end{lem}

\begin{proof}
Let $\sigma,\eta:X \to \sH$ be Borel sections and $B \subset [0,\infty)$ be Borel. For any section $\xi:X \to \bsH$,
$$F^{-1}(B) = \xi^{-1}(\Omega(\sigma,\eta,B)).$$
If $\xi$ is Borel then $F^{-1}(B)$ is Borel which implies (because $B$ is arbitrary) that $F$ is Borel. On the other hand, if $F$ is Borel then $\xi^{-1}(\Omega(\sigma,\eta,B))$ is Borel which implies that $\xi$ is Borel (because the sets $\Omega(\sigma,\eta,B)$ generate the Borel sigma-algebra on $\bsH$).
\end{proof}

\begin{lem}
If $\xi,\sigma:X \to \bsH$ and $\eta:X \to \sH$ are Borel sections then the map
$$x \mapsto  (\xi(x)|\sigma(x))_{\eta(x)}$$
is Borel.
\end{lem}

\begin{proof}
If $\sigma$ maps into $\sH$ then this follows from the previous lemma. So without loss of generality, we may assume $\sigma$ maps into $\partial \sH$. Because $\sH \to X$ is a separable bundle, there exists a sequence $\{\sigma_n\}_{n=1}^\infty$ of Borel sections such that for every $x\in X$, $\{\sigma_n(x)\}_{n\in \N}$ is dense in $\sH_x$. 

%Let $\eta_n:X \to \sH$ be the section defined by: $\eta_n(x)=\sigma_m(x)$ where $m$ is the smallest number such that 
%$$(\sigma_m(x) |\sigma(x))_{\eta(x)} \ge n.$$
%Then $\eta_n$ is Borel and 

For $n \in \N$ and $x\in X$, let
$$K_x(n)=\{p \in \sH_x:~ (p |\sigma(x))_{\eta(x)} \ge n\}.$$
Then $K(n) = \cup_x K_x(n)$ is Borel. Intuitively, $K_x(n)$ is a small neighborhood of $\sigma(x)$ when $n$ is large. Let $f_n:X \to \R$ denote the function
$$f_n(x)=\inf (\xi(x)| \sigma_m(x) )_{\eta(x)} )  $$
where the infimum is over all $m$ such that $\sigma_m(x) \in K_x(n)$.  Because $K(n)$ is Borel, $f_n$ is also Borel.  Because
$$(\xi(x)|\sigma(x))_{\eta(x)} = \liminf_{n\to\infty}  f_n(x)$$
the lemma follows.
%For $m\in \N$, let $K(n,m)=\{p\in K(n):~ (\xi(x)| p )_{\eta(x)}  \le f_n(x)+1/m\}$. Then $K(n,m)$ is Borel and $\pi:K(n,m) \to X$ is surjective. So there exist Borel sections $\sigma_{n,m}:X \to K(n,m)$. The previous lemma implies $x \mapsto  (\xi(x)|\sigma_{n,m}(x))_{\eta(x)}$ is Borel. Because
%$$(\xi(x)|\sigma(x))_{\eta(x)} = \lim_{n,m\to\infty}  (\xi(x)|\sigma_{n,m}(x))_{\eta(x)}$$
%the lemma follows.
\end{proof}

\begin{lem} 
There exists a constant $\delta>0$ such that $\sH_x$ is $\delta$-hyperbolic for a.e. $x$.
\end{lem}
\begin{proof}
Because $\sH$ is separable there exists a sequence $\{\sigma_i\}_{i\in\N}$ of Borel sections $\sigma_i:X \to \sH$ such that for every $x$, $\{\sigma_i(x)\}_{i\in\N}$ is dense in $\sH_x$. Let 
$$\delta_x:=\sup_{i,j,k,l} \min\{ (\sigma_i(x)|\sigma_j(x))_{\sigma_l(x)}, (\sigma_j(x)|\sigma_k(x))_{\sigma_l(x)} \} - (\sigma_i(x)|\sigma_k(x))_{\sigma_l(x)}.$$
This formula shows $x \mapsto \delta_x$ is Borel. Because $\{\sigma_i(x)\}_{i\in\N}$ is dense in $\sH_x$, it follows that $\delta_x$ is a hyperbolicity constant for $\sH_x$. Because $x\mapsto \delta_x$ is $\cR$-invariant (meaning $\delta_x=\delta_y$ whenever $x\cR y$), $\mu$-ergodicity of $\cR$ implies the lemma.
\end{proof}

Now let $\epsilon>0$ be such that $\epsilon \delta \le 1/5$ where $\delta>0$ is a hyperbolicity constant for the fibers $\sH_x$ (for a.e. $x$). Given a Borel section $\sigma:X \to \sH$, we define the metric $\rho_\sigma:\bsH*\bsH \to [0,\infty)$ by 
\begin{displaymath}
\rho_\sigma(\xi,\eta) = \left\{\begin{array}{cc}
\inf  \sum_{i=1}^{n-1}  \exp(-\epsilon (\xi_i|\xi_{i+1})_{\sigma(x)}) & \xi\ne \eta \\
0 & \xi =\eta
\end{array}\right.
\end{displaymath}
where the infimum is over all sequences $\xi_1,\ldots, \xi_n \in \bsH_x$ with $\xi_1=\xi, \xi_n = \eta$. By Lemma \ref{lem:meta}, the restriction of $\rho_\sigma$ to any fiber $\bsH_x$ is a metric. The previous lemmas imply that $\rho_\sigma$ is Borel.

We extend the action $\alpha$ to $\bsH$ by continuity. More precisely, because each $\alpha(x,y):\sH_y \to \sH_x$ is an isometry there is a unique extension (which we also denote by $\alpha(x,y)$) from $\bsH_y \to \bsH_x$ that is a homeomorphism. The uniqueness of the extension implies that the cocycle equation $\alpha(x,y)\alpha(y,z)=\alpha(x,z)$ is still satisfied. 
\begin{lem}
$$\{ (p,q) \in \bsH \times \bsH:~ \alpha(\pi(q),\pi(p))p = q\}$$
is a Borel subset of $\bsH \times \bsH$. 
\end{lem}

\begin{proof}
Let $\sigma:X \to \sH$ be a Borel section and define the metric $\rho_\sigma$ as above. Let $D$ denote the set of all $(p,q) \in \sH \times \sH$ such that $\alpha(\pi(q),\pi(p))p = q$. This set is Borel by hypothesis on $\alpha$. For $n>0$, let $D_n$ denote a certain $1/n$-neighborhood of $D$ in the square of the Gromov completions:
$$D_n:=\{(p,q) \in \bsH \times \bsH:~ \exists (p',q') \in D ~ \rho_\sigma(p,p') <1/n, ~\rho_\sigma(q,q') < 1/n\}.$$
Because $\rho_\sigma$ is Borel, $D_n$ is Borel. The lemma now follows from:
$$\bigcap_n D_n = \{ (p,q) \in \bsH \times \bsH:~ \alpha(\pi(q),\pi(p))p = q\}.$$
\end{proof}

 \subsection{Svarc-Milner}

Here we prove an analogue of the Svarc-Milner Lemma, giving conditions under which a hyperbolic graphing exists. 

\begin{thm}\label{thm:svarc-milner}
Suppose the Main Assumption is satisfied and there is a metrically proper Borel section $\sigma:X \to \sH$ and a real number $r>0$ such that $\sigma$ is  $r$-cobounded: the open radius $r$ neighborhood of the orbit $\cO^\sigma_x=\{\alpha(x,y)\sigma(y):~y\in [x]_\cR\}$ in $\sH_x$ is all of $\sH_x$. Then $\cR $ admits a hyperbolic graphing. % If moreover $\sigma$ is uniformly proper then $\cR $ admits a bounded degree hyperbolic graphing.
\end{thm}

\begin{remark}
We do not actually need all of the Main Assumption. We do not use separability or compactness of the orbit closure.
\end{remark}

\begin{proof}
%For $x\in X$, define $\sigma_x:[x]_\cR \to \sH_x$ by $\sigma_x(y) = \alpha(x,y)\sigma(y).$
%Define $d_\sigma:\cR  \to [0,\infty)$ by
%$$d_\sigma(x,y) = d(\sigma_z(x),\sigma_z(y))$$
%for any $z\in [x]_\cR=[y]_\cR$. The cocycle property of the action implies that this definition of $d_\sigma$ is independent of the choice of $z$. This gives a pseudo-metric on each equivalence class $[x]_{\cR}$. It is a metric if $\sigma_x:[x]_\cR \to \sH_x$ is injective.

Recall the definition of $d_\sigma$ from Definition \ref{defn:section}. Let $\cG =\{ (x,y) \in \cR :~d_\sigma(x,y)\le 3r\}$. We claim that $\cG$ is a hyperbolic graphing. To see that $\cG$ is a graphing, let $(x,y) \in \cR $. Let $\gamma$ be a geodesic in $\sH_x$ from $\sigma(x)$ to $\alpha(x,y)\sigma(y)$. So $\gamma$ is a continuous map from the interval $[0,d_\sigma(x,y)]$ to $\sH_x$ with $\gamma(0)=\sigma(x)$ and $\gamma(d_\sigma(x,y))=\alpha(x,y)\sigma(y)$. Because $\sigma$ is $r$-cobounded, for every number $t \in [0,d_\sigma(x,y)]$ there exists $y_t \in [x]_\cR $ such that 
$$d(\gamma(t),\sigma_x(y_t)) \le r.$$
By the triangle inequality, if $t,s \in [0,d_\sigma(x,y)]$ then
$$d_\sigma(y_t, y_s) \le 2r+|t-s|.$$
Therefore, if $|t-s|\le r$ then $(y_t,y_s)\in \cG$. In particular, $x, y_r, y_{2r},\ldots, y_{nr}, y$ is a path in $\cG$ from $x$ to $y$ where $n =\lfloor d_\sigma(x,y)/r\rfloor$. This proves $\cG$ is a graphing. Moreover if $d_\cG(x,y)$ is the length of the shortest path in $\cG$ from $x$ to $y$ then
$$d_\cG(x,y) \le 2+  d_\sigma(x, y)/r.$$

On the other hand, there exists a shortest path $x=x_1,\ldots,x_m=y$ in $\cG$ from $x$ to $y$. By definition $d_\sigma(x_i,x_{i+1})\le 3r$. Therefore $d_\sigma(x,y)\le 3rm = 3rd_\cG(x,y)$. So we have proven
$$(1/3r)d_\sigma(x,y) \le d_\cG(x,y) \le 2+  d_\sigma(x,y)/r.$$
This proves that $\sigma_x:[x]_\cR \to \sH_x$ is a quasi-isometry with respect to $d_\cG$ and $d \resto \sH_x$. Because $\sigma$ is metrically proper, $\cG$ is locally finite. By \cite[Theorem 3.18]{vaisala-hyperbolic}, $([x]_\cR,d_\cG)$ is a hyperbolic metric space.
\end{proof}

 \subsection{Orbits \& limit sets}
Let $\sigma:X \to \sH$ be a Borel section and define $\sigma_x, \cO_x^\sigma, \cL_x^\sigma$ as in the introduction. The main result of this section is:
%\begin{notation}
%For convenience, if $\sigma:X \to \sH$ is a section we let $\sigma_x:[x] \to \sH_x$ denote the map
%$$\sigma_x(y)=\alpha(x,y)\sigma(y).$$
%\end{notation}

%In order to define limit sets, let $\sigma:X \to B$ be a Borel section. For $x \in X$, let $O_x = O^\sigma_x= \{ \alpha(x,y)\sigma(y):~ y\cR x\} \subset B_x$. This is the {\bf orbit of $x$} with respect to $\sigma$. Let $L^\sigma = \{ (x,\xi):~x\in X, \xi \in \partial \overline{O_x}\}$ where the closure is in $\overline{B_x}$. This is the {\bf limit set} of $\sigma$. Observe that $O_x = \alpha(x,y)O_y$. Therefore, $L^\sigma$ is invariant: $L^\sigma_x = \alpha(x,y)L^\sigma_y$.

%Does $L_\sigma$ depend on $\sigma$?

%\begin{thm}\label{thm:limit-set}
%For any metrically proper sections $\sigma,\eta$, $\cL^{\sigma}_x=\cL^{\eta}_x$ for a.e. $x$.
%\end{thm}

\noindent {\bf Theorem \ref{thm:limit-set}}
{\it If the Main Assumption is satisfied and $\sigma,\eta:X \to \sH$ are metrically proper sections whose orbit closures $\overline{\cO^\sigma_x}, \overline{\cO^\eta_x}$ are compact (for a.e. $x$) then $\cL^\sigma_x = \cL^\eta_x$ for a.e. $x$.}

The idea of the proof is this: suppose $K \subset \bsH$ is such that $K_x:=K \cap  \bsH_x$ is closed and quasi-convex and $K$ is invariant in the sense that $\alpha(x,y)K_y=K_x$. Suppose as well that $K_x \cap \sH_x$ is nonempty for a.e. $x$. We consider the projection map from $\bsH_x$ to $K_x$. In case $K_x$ does not contain the limit set $\cL^\sigma_x$, this map is infinite-to-1 (a.e. $x$). However, this contradicts the Mass-Transport Principle. Applying this to the case when $K$ equals the convex hull of $\cO^\sigma$, we obtain that $\cL^\eta_x \subset \Hull(\cO^\sigma_x) \cap \partial \sH_x$ for a.e. $x$. Since $\Hull(\cO^\sigma_x) \cap \partial \sH_x = \cL^\sigma_x$, this proves one inclusion. The other follows by symmetry.

%By switching the roles of $\eta$ and $\sigma$, we see that either Theorem \ref{thm:limit-set} holds or $|\cL^\sigma_x|\le 1$ for a.e. $x$ or $|\cL^\eta_x|\le 1$ for a.e. $x$. The latter possibilities are controlled using a Busemann function argument:

Now the details. First we review convex hulls and quasi-convexity.

\begin{defn}
Let $(\cH,d)$ be a complete $\delta$-hyperbolic metric space and $\bcH$ its Gromov completion. A subset $Y \subset \bcH$ is {\bf $C$-quasi-convex} if for every $x,y \in Y$, every geodesic $[x,y]$ is contained in the radius-$C$-neighborhood of $Y$. If $Y$ is $C$-quasi-convex for some $C>0$ then $Y$ is called {\bf quasi-convex}. If $Y \subset \bcH$ then the {\bf closed convex hull of $Y$}, denoted $\Hull(Y)$, is the closure of the union of all geodesic segments $[\xi,\eta]$ with $\xi,\eta \in F$. 
\end{defn}

%\begin{lem}
%For every $\lambda, c$ there exists a constant $C=C(\lambda,c,\delta)$ such that every $(\lambda,c)$-quasi-geodesic in $X$ is $C$-quasi-convex.
%\end{lem}

\begin{lem}\label{lem:greenberg}
Let $(\cH,d)$ be a complete $\delta$-hyperbolic metric space and $K \subset \bcH$ be compact. Then $\Hull(K)$ is quasi-convex and closed. Moreover, $\Hull(K) \cap \partial \cH = K \cap \partial \cH$.
\end{lem}

\begin{proof}
The special case in which $K \subset \partial \cH$ is \cite[Lemmas 3.2, 3.6]{kapovich-short-1996}. The general case is nearly identical.
\end{proof}

\begin{lem}\label{lem:projection}
Let $(\cH,d)$ be a complete $\delta$-hyperbolic metric space and let $K \subset \bcH$ be closed and $C$-quasi-convex. For every $x\in \cH$, let $\Proj_K(x)$ be the set of all elements $y \in K$ realizing the minimum distance from $x$ to $K$ (so $d(x,y)=d(x,K)$). For $\xi \in \partial \cH \setminus K$ let $\Proj_K(\xi)$ be the set of all $k \in K$ such that there exists a sequence $\{x_n\} \subset K$ with $\lim_n x_n = \xi$ and $k_n \in \Proj_K(x_n)$ with $\lim_n k_n = k$. 

%Similarly, if $\xi \in \partial \cH \setminus \partial H$ then let 
%$$\Proj_H(\xi)= \{y \in H:~ h_\xi(y) = \min_{z \in H} h_\xi(z)\}.$$
Then there exists a constant $Q=Q(C,\cH,d)$ such that for every $x\in \bcH \setminus K$,  the diameter of $\Proj_K(x)$ is at most $Q$. Moreover, $\Proj_K(x)$ is nonempty and if $\{x_n\} \subset \cH$ is any sequence with $\lim_n x_n = \xi \in \partial \cH \setminus K$ then 
$$\limsup_n d(\Proj_K(x_n), \Proj_K(\xi)) \le Q$$
where $d(\cdot,\cdot)$ is the Hausdorff metric.
%Moreover if $\{x_n\} \subset \bX$ is a sequence with $\lim_n x_n = x_\infty \in \bX \setminus \partial H$ then $\limsup_{n \to \infty} d(\Proj_H(x_n), \Proj_H(x_\infty)) < 100\delta$.
\end{lem}
%[[** this lemma is used in the proof that $\cL^\sigma=\cL^\eta$ **]]

\begin{proof}
If $x\in \cH$ then the fact that $\diam(\Proj_K(x))$ is uniformly bounded is \cite[Chapter III.$\Gamma$ Proposition 3.11]{bridson-haefliger-book}. The case $x \in \partial \cH \setminus \partial K$ follows by taking limits. The rest of the lemma follows from standard arguments left to the reader.
\end{proof}

The next lemma is used several times throughout the paper.
\begin{lem}\label{limit-set1}
Suppose $\sigma$ is metrically proper. Suppose $K \subset \overline{\sH}$ is Borel and
\begin{itemize}
\item $K_x:=K \cap \bsH_x$ is closed and quasi-convex for a.e. $x$,
\item $K_x \cap \sH_x$ is nonempty for a.e. $x$,
\item $\alpha(y,x)K_x = K_y$ for all $x\cR y$.
% and\item $K_x \subset \cL^\sigma_x$ for every $x$.
\end{itemize}
Then  $\cL^\sigma_x \subset K_x$ for a.e. $x$. %there is a conull set $Y \subset X$ such that $\cL^{Y,\sigma}_x \subset K_x$ for a.e. $x$.
\end{lem}

\begin{proof}
%If $x\cR y$ then $\alpha(y,x)K_x = K_y$ which implies $|K_x| = |K_y|$. Since $\cR$ is ergodic, $|K_x|$ is essentially constant in $x$. We may assume without loss of generality that $|K_x|>1$ for a.e. $x$. Let $H(x) \subset \overline{\sH}_x$ be the closed convex hull of $K_x$. Note $H(x)$ is not contained in $\partial \sH_x$ because $|K_x|>1$ for a.e. $x$. By Lemma \ref{lem:hull} $H(r,x)$ is quasi-convex.

%Let $H(r,x)$ denote the closed radius $r$ neighborhood of $H(x)$. By Lemmas \ref{lem:hull} and \ref{lem:hull2} $H(r,x)$ is quasi-convex. Let $X_r=\{x\in X:~\sigma(x) \in H(r,x)\}$. Then $\cup_r X_r=X$. So there exists an $r_0>0$ such that $\mu(X_{r_0})>0$. %Let $C_0>0$ be a quasi-convexity constant for $K_x$ (for a.e. $x$).

%\noindent {\bf Claim 1}. For a.e. $x\in X$, $K_x\cap O^\sigma_x \ne \emptyset$.

%\begin{proof}[Proof of Claim 1]
%Let $S_0 = \{x\in X: K_x\cap O^\sigma_x \ne \emptyset\}$. Observe that $X_{r_0} \subset S_0$ since if $x\in X_{r_0}$ then $\sigma(x) \in K_x\cap O^\sigma_x$. So $\mu(S_0)>0$. Because $\cR $ is ergodic, it suffices to show that $S_0$ is saturated. Indeed, if $x \in S_0$ then there exists $z$ with $z\cR x$ and $\alpha(x,z)\sigma(z) \in K_x$. So if $y\cR x$ is arbitrary then
%$$\alpha(y,z)\sigma(z) \in \alpha(y,x) K_x = H(r_0,y)$$
%which implies $y\in S_0$ as required.
%\end{proof}

%After removing a set of measure zero, we will now assume that $K_x\cap O^\sigma_x \ne \emptyset$ for every $x\in X$. 

As in Lemma \ref{lem:projection}, given $b \in \sH_x$, let 
$$\Proj_{K_x}(b)=\Big\{c \in K_x:~d(c,b) = \inf \{d(c',b):~ c' \in K_x\} \Big\}.$$
We also extend $\Proj_{K_x}$ to $\partial \sH_x$ by defining $\Proj_{K_x}(\xi)$ equal to the set of all $h\in K_x$ such that there exists $\{q_n\} \subset \sH_x$ with $q_n \to \xi$ and $h_n \in \Proj_{K_x}(q_n)$ with $h_n \to h$ as $n\to\infty$. 

By Lemma \ref{lem:projection} there is a constant $Q>0$ such that $\diam( \Proj_{K_x}(b)) \le Q$ and $\Proj_{K_x}(b)$ is nonempty for $b \in \bsH_x \setminus K_x$. 

Let $W$ be the set of all $x\in X$ such that there exists $\xi \in \cL^\sigma_x \setminus K_x$. To obtain a contradiction, suppose $\mu(W)>0$. For $r>0$, consider the subset 
$$V_r:=\{x\in X:~ \exists \xi \in \cL^\sigma_x \setminus K_x \textrm{ such that } d(\sigma(x), \Proj_{K_x}(\xi)) < r\}.$$
 Because $\cup_r V_r = W$, there exists some $r_1>0$ such that $\mu(V_{r_1})>0$. 

For $S \subset \sH_x$ and $r>0$ let $\cN_{r}(S)$ denote the closed radius $r$-neighborhood of $S$. Define $F:\cR  \to \R$ by 
$$F(x,y) = |\cO^\sigma_y \cap \cN_{r_1 + Q}(\Proj_{K_y}(\sigma_y(x)))|^{-1}$$
if $\sigma(y)\in \cN_{r_1 + Q}(\Proj_{K_y}(\sigma_y(x)))$ and $0$ otherwise. By the Mass-Transport Principle (Lemma \ref{lem:mtp}), 
\begin{eqnarray}\label{eqn:mtp2}
1 = \int \sum_{y \in [x]_\cR} F(x,y)~d\mu(x) = \int \sum_{x \in [y]_\cR} F(x,y)~d\mu(y).
\end{eqnarray}
Here we have used that $\sigma$ is metrically proper and $\Proj_{K_x}(\sigma(x))$ has bounded diameter to conclude that $|\cO^\sigma_y \cap \cN_{r_1 + Q}(\Proj_{K_y}(\sigma_y(x)))|$ is finite.

Let $y \in V_{r_1}$. By definition, there exists $\xi \in \cL^\sigma_y \setminus K_y$ such that $d(\sigma(y), \Proj_{K_y}(\xi)) < r_1$. Since $\xi$ is a limit point, there is a sequence $\{x_i\} \subset [y]_\cR$ such that $\sigma_y(x_i) \to \xi$ as $i\to\infty$. 

By Lemma \ref{lem:projection},
$$\limsup_i d(\Proj_{K_y}(\sigma_y(x_i)), \Proj_{K_y}(\xi)) \le Q.$$
Therefore, if $i$ is sufficiently large then $F(x_i,y) > 0$. By definition we have
$$F(x_i,y) = |\cO^\sigma_{y} \cap\cN_{r_1+Q}(\Proj_{K_y}(\sigma_y(x_i)))|^{-1} \ge |\cO^\sigma_{y} \cap \cN_{r_1+3Q}(\Proj_{K_y}(\xi))|^{-1} >0.$$
Therefore
$$\sum_{x \in [y]_\cR} F(x,y) \ge \sum_{x \in [y]_\cR} |\cO^\sigma_{y} \cap \cN_{r_1+3Q}(\Proj_{K_y}(\xi))|^{-1} = +\infty.$$
By (\ref{eqn:mtp2}), we must have that $\mu(V_{r_1})=0$ contradicting our earlier hypothesis. Thus $\mu(W)=0$. This implies the lemma.

\end{proof}

\begin{proof}[Proof of Theorem \ref{thm:limit-set}]
By Lemmas \ref{lem:greenberg} and \ref{limit-set1}, for a.e. $x$,
$$\cL^\sigma_x \subset \Hull(\cO^\eta_x) \cap \partial \sH_x = \cL^\eta_x.$$ 
By symmetry, this proves the Theorem.
\end{proof}

From now on, we will write $\cL$ to mean $\cL^\sigma$. 

%To prove Theorem \ref{thm:minimal2} it will be useful to have the following lemma:

%\begin{lem}\label{lem:limit-set3}
 %Suppose the Main Assumption is satisfied. Let $Y \subset X$ be a set with positive measure. For $x\in X$, let $\cO^Y_x = \{ \alpha(x,y)\sigma(y):~y\in Y\}$ and let $\cL^Y_x = \overline{\cO^Y_x} \cap \partial \sH_x$ be the corresponding limit set. Then $\cL^Y_x  = \cL_x$ for a.e. $x$.
% \end{lem}
 
 %\begin{proof}
% Note that $\cO^Y$ is invariant in the sense that $\alpha(x,y)\cO^Y_y = \cO^Y_x$ for a.e. $x\cR y$. By Lemmas \ref{lem:greenberg} and \ref{limit-set1}, for a.e. $x$,
%$$\cL_x \subset \Hull(\cO^Y_x) \cap \partial \sH_x = \cL^Y_x.$$ 
% Since the opposite inequality is immediate, this proves the lemma.
 % \end{proof}

\subsection{Minimality}

%\begin{defn}[Minimal actions]
%Let $(X,\mu,\cR)$ be a discrete measured equivalence relation, $\pi:\sH \to X$ be a bundle of metric spaces and $\{\alpha(x,y)\}$ is an action of $\cR $ on $\sH$. The action is {\bf minimal} if $\cR $ is ergodic and there does not exist a Borel subset $C \subset \sH$ such that
%\begin{itemize}
%\item $C_x:=C \cap \sH_x$ is closed in $\sH_x$ for a.e. $x$,
%\item $C$ is $\alpha$-invariant in the sense that $\alpha(x,y)C_y=C_x$ for a.e. $(x,y)\in \cR$,
%\item $C$ is {\bf proper}:  $C \cap \partial \sH_x \notin \{\emptyset, \sH_x\}$ for a.e. $x$.
%\end{itemize}
%\end{defn}

%Suppose $\sigma:X \to \sH$ is a metrically proper section. Let $\cL_x = \cL^\sigma_x$ and $\cL=\cup_x \cL_x \subset \partial \sH$. By Theorem \ref{thm:limit-set} $\cL$ is well-defined up to measure zero sets.

\begin{thm}\label{thm:minimal}
Suppose the Main Assumption is satisfied and $\cR$ is non-hyperfinite. Then the action of $\cR$ on the limit set $\cL$ is minimal in the following sense: if $C \subset \partial \cL$ is Borel and 
\begin{itemize}
\item $C_x:=C \cap \cL_x^\sigma$ is closed for a.e. $x$,
\item $C$ is $\alpha$-invariant in the sense that $\alpha(x,y)C_y = C_x$ for a.e. $(x,y)\in \cR$,
\item $C_x \ne \emptyset$ for a.e. $x$
%\item $C$ is {\bf proper}:  $C \cap \cL_x \notin \{\emptyset, \cL_x\}$ for a.e. $x$.
\end{itemize}
then $\cL_x \subset C_x$ for a.e. $x$.
\end{thm}

\begin{prop}\label{prop:limit-set3}
Suppose $\sigma:X \to \sH$ is a metrically proper section and $\xi:X \to \partial \sH$ is an $\cR$-invariant section (so $\alpha(y,x)\xi(x)=\xi(y)$ a.e.). Then $\cR$ is hyperfinite.
\end{prop}

\begin{proof}
The proof is essentially the same as the proof of \cite[Theorem 3.4]{kaimanovich-boundary-amenability} (which uses Reiter's criterion to prove $\cR$ is amenable). The main difference is that Kaimanovich considers only the case when the bundle $\sH$ comes from a hyperbolic graphing as in Example \ref{ex:graphing}. The general case is not significantly different. 
\end{proof}

\begin{proof}[Proof of Theorem \ref{thm:minimal}]
By Proposition \ref{prop:limit-set3} and the $\mu$-ergodicity of $\cR$, it follows that $|C_x|>1$ for a.e. $x$. Therefore, $\Hull(C_x) \cap \sH_x$ is nonempty for a.e. $x$. So Lemma  \ref{limit-set1} implies $\cL_x \subset C_x$ for a.e. $x$ which implies the theorem.
\end{proof}

\section{Fields of boundary measures}\label{sec:fields1}

Let us assume the Main Assumption is satisfied. As in the previous section, we let $\cL\to X$ denote the bundle of limit sets. Given a function $F:\cL \to \C$ and $x\in X$, we let $F_x:\cL_x \to \C$ denote its restriction to the fiber $\cL_x$. We say two functions $F, G$ on $\cL$ are {\bf equivalent} if for a.e. $x\in X$ $F_x  = G_x$. For each $x\in X$, let $C(\cL_x)$ denote the Banach space of continuous functions on $\cL_x$ with the sup norm. Suppose $F:\cL \to \C$ is a Borel function such that $F_x \in C(\cL_x)$ for a.e. $x\in X$. Then we define its norm by
$$\|F\|:= \| x \mapsto \|F_x\| \|_{L^\infty(X,\mu)}.$$
Let $C(\cL \to X)$ denote the set of all equivalence classes of Borel functions $F:\cL \to \C$ such that for a.e. $x\in X$, $F_x \in C(\cL_x)$ and $\|F\|<\infty$. This is a Banach space with the above norm.

A {\bf field of boundary measures} is a collection $\nu=\{\nu_x:~x\in X\}$ such that
\begin{itemize}
\item for every $x\in X$, $\nu_x$ is a Borel probability measure on the fiber $\cL_x$ and 
\item for every $F \in C(\cL\to X)$, the map $x \mapsto  \nu_x(F_x) = \int  F_x~d\nu_x$ is $\mu$-measurable.
\end{itemize}
Two fields $\nu,\eta$ are {\bf equivalent} if $\nu_x=\eta_x$ for a.e. $x$. By abusing notation, we will not distinguish between equivalent fields. 

Let $\Prob(\cL\to X)$ denote the set of all (equivalence classes of) fields of boundary measures. Given $F \in C(\cL \to X)$ and an open set $O \subset \C$, let $\Omega(F,O)$ be the set of all $\nu \in \Prob(\cL \to X)$ such that $\int \nu_x(F_x)~d\mu(x) \in O$. We give $\Prob(\cL \to X)$ the topology generated by sets of the form $\Omega(F,O)$. 

\begin{thm}\label{thm:jcontinuous1}
If the Main Assumption is satisfied then $\Prob(\cL \to X)$ is affinely homeomorphic to a compact convex metrizable subspace of a Banach space. Moreover the full group $[\cR]$ acts jointly continuously on $\Prob(\cL \to X)$ by
$(\phi \nu)_{\phi x} = \alpha(\phi x,x)_*\nu_{x}$. 
\end{thm}

\begin{proof}
This follows from Lemma \ref{lem:fields} and Theorem \ref{thm:jcontinuous}.
\end{proof}

\begin{defn}[Fixed point set]
If $\cS \le \cR$ is a subequivalence relation then we let $\Fix(\cS)$ denote the set of all $\eta \in \Prob(\cL \to X)$ such that for a.e. $(x,y) \in \cS$, $\alpha(y,x)_*\eta_x=\eta_y$. Also if $f \in [\cR]$ then we let $\Fix(f)$ denote the set of all $\eta \in \Prob(\cL \to X)$ such that for a.e. $x$, $\alpha(f(x),x)_*\eta_x = \eta_{f(x)}$. Observe that $\Fix(\cS)$ and $\Fix(f)$ are closed convex subsets of $\Prob(\cL \to X)$.
\end{defn}

\begin{thm}\label{thm:K}
Suppose the Main Assumption is satisfied. Let $\cS\le \cR$ be a subequivalence relation. Then $\Fix(\cS)$ is nonempty if and only if $\cS$ is hyperfinite. Moreover, if $\cS$ is aperiodic and $\nu \in \Fix(\cS)$ then for a.e. $x$, the support of $\nu_x$ has at most 2 points.
\end{thm}

\begin{remark}
The proof of Theorem \ref{thm:K} above is essentially the same as \cite[Theorem 2.20]{kaimanovich-boundary-amenability}. The main difference between the two results are in the hypotheses. In Kaimanovich's case the bundle $\sH$ comes from a hyperbolic graphing (as in Example \ref{ex:graphing}), the equivalence relation $\cR$ need not be measure-preserving and $\nu$ is assumed to be $\cR$-invariant instead of merely $\cS$-invariant. This strategy has roots in  earlier work of Adams \cite{adams-indecomp}.

Here is a brief sketch: if $\cS$ is hyperfinite then the fixed point property of amenable equivalence relations (see \cite{zimmer-1984}) immediately gives that $\Fix(\cS)$ is nonempty. If $\Fix(\cS)$ is nonempty then one can prove $\cS$ is hyperfinite from Reiter's condition. Now suppose $\nu \in \Fix(\cS)$. There is a map that associates to any Borel probability measure on $\partial \sH_x$ whose support contains more than 2 elements, a bounded subset of $\sH_x$ called its barycenter.  Let $r_x$ be the infimum of the distances between $\sigma_x(y)$ and this barycenter. Because $\sigma$ is metrically proper, the subset of $[x]_\cR$ that realizes this infimum is a finite set called the $\sigma$-barycenter. Note we are using the metrical properness of $\sigma$ as a surrogate for the local finiteness that plays a similar role in Adams' arguments.

The map that associates to $x$ the $\sigma$-barycenter of $\nu_x$ is Borel and $\cR$-invariant. Therefore, if the support of $\nu_x$ contains more than 2 elements for a set of $x$'s of positive measure then by Lemma \ref{lem:smooth}, $\cS$ is finite on a set of positive measure. In particular, it cannot be aperiodic.

%**Should bother to write up a proof? We could just say it was proven by Kaimanovich.
\end{remark}

Recall that a hyperfinite subequivalence relation $\cS \le \cR$ is {\bf parabolic} if there is a unique element $\nu \in \Fix(\cS)$ and $\nu_x$ is supported on a single point of $\cL_x$ for a.e. $x$. We say $\cS$ is {\bf loxodromic} if there exists $\nu \in \Fix(\cS)$ such that the support of $\nu_x$ contains two points of $\cL_x$ for a.e. $x$. We say $\cS$ is {\bf mixed} if there is a nontrivial measurable partition $X=Y\sqcup Z$ such that $\cS \resto Y$ is parabolic and $\cS\resto Z$ is loxodromic.  Moreover this partition is unique up to null sets. Uniqueness implies it is $\cS$-invariant (modulo a measure zero set).

\begin{lem}\label{lem:mixed}
Suppose the Main Assumption is satisfied and $\cS \le \cR$ is aperiodic and hyperfinite. Then $\cS$ is either parabolic, loxodromic or mixed. 
\end{lem}

\begin{proof}
If $\nu \in \Fix(\cS)$ then let $D(\nu)$ be the set of all $x\in X$ such that $\nu_x$ is a Dirac measure on $\cL_x$ (that is, the $\nu_x$ is supported on a single point of $\cL_x$). Let $\beta = \inf \{ \mu(D(\nu)):\nu \in \Fix(\cS)\}$. For every natural number $n$, let $\nu_n \in \Fix(\cS)$ be a field of boundary measures such that $\mu(D(\nu_n))\le \beta+1/n$. Finally consider $\sum_{n=1}^\infty 2^{-n}\nu_n :=\nu_\infty$. Because $\Fix(\cS)$ is convex and closed, $\nu_\infty \in \Fix(\cS)$ and $\mu(D(\nu_\infty))=\beta$ by construction. 

If $\cS$ is neither loxodromic nor parabolic then $0<\beta <1$. Observe that $\cS$ restricted to $D(\nu_\infty)$ is parabolic and $\cS$ restricted to the complement $X\setminus D(\nu_\infty)$ is loxodromic. The uniqueness of the set $D(\nu_\infty)$ is immediate since if $\cS \resto Y_i$ is parabolic (for $i=1,2$) then $\cS \resto (Y_1\cup Y_2)$ is also parabolic. 
\end{proof}

%\begin{lem}\label{lem:local-global3}
%Let $\cH$ be a complete geodesic $\delta$-hyperbolic metric space and let $s>0$. Then there exists $r>0$ such that the following holds. Let $\xi\in \partial \cH$ and $x_1,x_2,\ldots \in \cH$ satisfy:
%\begin{itemize}
%\item for each $i$ there is a geodesic $[x_i,\xi]$ and a point $y_i \in [x_i,\xi]$ such that $d(x_i,y_i)\ge r$, $d(y_i,x_{i+1})\le s$.
%\end{itemize}
%Then the concatenation of geodesic  segments $[x_1, x_2], [x_2,x_3], \ldots$ is a $C$-quasi-geodesic with $\xi = \lim_i x_i$. Moreover, the constant $C>0$ depends only on $\delta,r,s$.
%\end{lem}

%\begin{remark}
%It seems likely that if $\cR$ is parabolic (and satisfies all of the other hypotheses of the previous lemma) then $\cL_x(\cR)$ contains exactly one point almost surely. 
%\end{remark}

\subsection{Limit sets of hyperfinite equivalence relations}

%For this subsection, let $\cR$ be hyperfinite and aperiodic. As in the introduction, we say that $\cR$ is {\bf parabolic} if there is a unique element $\nu \in \Fix(\cR)$ and $\nu_x$ is supported on a single point for a.e. $x$. We say $\cR$ is {\bf loxodromic} if there exists $\nu \in \Fix(\cR)$ such that the support of $\nu_x$ contains two points for a.e. $x$. If it neither parabolic nor loxodromic, then $\cR$ is {\bf mixed}. If $\cR$ is ergodic then it must be either parabolic or loxodromic. 

It may seem obvious that if $\cR$ is parabolic then the limit set $\cL_x(\cR)$ should contain exactly one point and if $\cR$ is loxodromic then $\cL_x(\cR)$ contains exactly two points. We prove the latter below (assuming the Main Assumption); the former remains open. This result is used in the proof that every aperiodic hyperfinite subequivalence relation is contained in a unique maximal hyperfinite subequivalence relation.

 Before getting to the proof we observe that both statements are definitely false if we do not require $\cR$ to be probability-measure-preserving. To see this, let $G$ denote a non-elementary word hyperbolic group and $S \subset G$ a finite symmetric generating set for $G$. If $G \cc (X,\mu)$ is any measure-class-preserving action of $G$ and $\cR \subset X\times X$ is the resulting orbit-equivalence relation then there is a canonical graphing $\cG$ of $X$: $\cG=\{ (x,sx):~s\in S, x\in X\}$. This graphing determines a canonical bundle of hyperbolic metric spaces over $X$ as in Example \ref{ex:graphing} in which each fiber is isometric to the Cayley graph determined by $(G,S)$. Now let $\Gamma=\Gamma(G,S)$ denote this Cayley graph, $\partial \Gamma$ its Gromov boundary and $\nu$ a probability measure on $\partial \Gamma$ whose measure class is preserved under the canonical $G$-action. In \cite[Examples 2.2.4, 2.2.5]{kaimanovich-boundary-amenability} it is shown that the orbit-equivalence relation of the action $G \cc (\partial \Gamma,\nu)$ is parabolic and the orbit-equivalence relation of $G \cc (\partial \Gamma \times \partial \Gamma, \nu \times \nu)$ is loxodromic with respect to the canonical hyperbolic bundle described above. (Neither equivalence relation is probability-measure-preserving and therefore neither satisfies the Main Assumption). However, in both cases each fiber is isometric to the Cayley graph $\Gamma$ on which $G$ acts vertex-transitively. It follows that in both cases, the limit set $\cL_x$ can be identified with $\partial \Gamma$. It is therefore infinite whenever $G$ is non-elementary.

\begin{lem}\label{lem:loxodromic-case}
Suppose the Main Assumption holds, $\cR$ is loxodromic and $\nu \in \Fix(\cR)$ is an invariant field of boundary measures such that the support of $\nu_x$ has cardinality 2 (for a.e. $x$) then $\cL_x$ is the support of $\nu_x$ for a.e. $x$. In particular, $|\cL_x|=2$ for a.e. $x$.
\end{lem}

\begin{proof}
%Because $\cR $ is loxodromic, there exists an invariant field of boundary measures $\nu$ such that the cardinality of the support of $\nu_x$ is two for a.e. $x$.

 %For $x\in X$, let $\gamma_x \subset \sH_x$ be a geodesic with endpoints in the support of $\nu_x$. We may choose $\gamma_x$ to depend measurably on $x$. 

 For $r>0$, let $X_r$ denote the set of all $x\in X$ such that there exists a geodesic $\gamma_x \subset \sH_x$ with endpoints in the support of $\nu_x$ such that $d(\sigma(x),\gamma_x) \le r$. Because $\cup_{r>0} X_r = X$, there exists $r>0$ such that $\mu(X_r)>0$. 

Let $\cO^{r}_x:= \{ \sigma_x(y):~y \in X_r \cap [x]_\cR\}$ and $\cL^{r}_x := \overline{\cO^{r}_x} \cap \partial \sH_x$ be the associated limit set. On the one hand, because each $y \in X_r$ has image $\sigma_x(y)$ $r$-close to a geodesic $\gamma_x$, it follows that $\cL^{r}_x$ is contained in the support of $\nu_x$. On the other hand, Lemma \ref{limit-set1} implies
$$\cL_x \subset \Hull(\overline{\cO^r_x}) \cap \partial \sH_x = \cL^r_x \subset \rm{support}(\nu_x).$$
Since the support of $\nu_x$ is contained in $\cL_x$ by definition of $\Fix(\cR)$, the lemma follows.
\end{proof}

\subsection{Parabolic equivalence relations}

The goal of this section is to prove when $\cR$ is parabolic then every aperiodic subequivalence relation of $\cR$ also parabolic. More precisely:
\begin{thm}\label{thm:non-nested}
Suppose the Main Assumption is satisfied and $\cR$ is hyperfinite and parabolic. Then every aperiodic subequivalence relation $\cS\le\cR$ is also parabolic.
\end{thm}

Here is a short proof sketch: to obtain a contradiction suppose there exists a loxodromic $\cS \le \cR$. Because $\cR$ is parabolic there exists a unique $\cR$-invariant section $\xi:X \to \cL$. We show there is a function $\Phi:X \to X$ with $x\cR \Phi(x)$ (for a.e. $x$), such that $\sigma_x(\Phi(x))$ is ``closer'' to the boundary point $\xi(x)$ than $\sigma(x)$ and the three points $\sigma(x),\sigma_x(\Phi(x)),\xi(x)$ all lie close to a geodesic.  Moreover the broken geodesic path 
$$[\sigma(x),\sigma_x(\Phi(x))] \cup [\sigma_x(\Phi^2(x)), \sigma_x(\Phi^3(x))] \cup \cdots$$
forms a quasi-geodesic limiting on $\xi(x)$. The Mass Transport Principle implies that $\Phi$ restricted to a certain positive measure subset of $X$ is a bijection (mod $\mu$). But this contradicts the geometric fact that $\Phi$ is contracting.

Now for the details. Recall that if $Y \subset X$ then $\cR \resto Y:=\cR \cap (Y\times Y)$ is an equivalence relation on $Y$ called the {\bf compression} (or {\bf restriction}) of $\cR$ to $Y$ (more details of this construction are discussed in \S \ref{sec:compression}). It will be convenient to have the following Lemma:
\begin{lem}\label{lem:compression}
Suppose the Main Assumption is satisfied and $\cR$ is hyperfinite. Let $Y \subset X$ have positive $\mu$-measure. Then $\cR \resto Y$ is parabolic if and only if $\cR$ is parabolic. Similarly, $\cR \resto Y$ is loxodromic if and only if $\cR$ is loxodromic.
\end{lem}

\begin{proof}
Because $\cR$ and $\cR \resto Y$ are ergodic, they cannot be mixed. So it suffices to show that $\cR$ is loxodromic if and only if $\cR \resto Y$ is loxodromic.  If $\cR$ is loxodromic then there exists a $\cR$-invariant field of boundary measure $\nu$ such that the support of $\nu_x$ has cardinality two for a.e. $x$. Then $\nu$ restricted to $Y$ is $\cR\resto Y$-invariant. This proves $\cR\resto Y$ is loxodromic. 

On the other hand, suppose $\cR\resto Y$ is loxodromic. Let $\nu=\{\nu_x\}_{x\in Y}$ be a $\cR\resto Y$-invariant field of boundary measures such that the support of $\nu_x$ has cardinality two for a.e. $x\in Y$. For any $x\in X$, define $\nu_x$ by $\nu_x=\alpha(x,y)_*\nu_y$ for $y\in Y \cap [x]_\cR$. This is well-defined because $\nu$ is $\cR\resto Y$-invariant and $\mu$ is $\cR$-ergodic (so for a.e. $x\in X$, there exists $y\in Y \cap [x]_\cR$). Clearly the extended $\nu$ is $\cR$-invariant. This proves $\cR$ is loxodromic.

%Now suppose $\cR$ is parabolic. Then $\cR\resto Y$ is hyperfinite, aperiodic and ergodic. It cannot be loxodromic since that would imply $\cR$ is loxodromic. So $\cR\resto Y$ is parabolic. On the other hand, suppose $\cR\resto Y$ is parabolic. Then $\cR$ is hyperfinite, aperiodic and ergodic.  It cannot be loxodromic since that would imply $\cR\resto Y$ is loxodromic. So $\cR$ is parabolic. 

\end{proof}

Quasi-geodesics are defined in \S \ref{sec:qi}. We will use the following lemma to show that certain broken geodesic paths are quasi-geodesics:
\begin{lem}\label{lem:local-global}
Let $(\cH,d_\cH)$ be a geodesic $\delta$-hyperbolic space. Then for every $s>0$ there exists an $r>0$ such that if $\{x_i\}_{i\in \N} \subset \cH$ is any sequence with 
$$d_\cH(x_i,x_{i+1}) \ge r, \quad (x_i|x_{i+2})_{x_{i+1}} \le s \quad \forall i \in \N$$
then the piecewise geodesic $[x_1,x_2] \cup [x_2,x_3] \cup \cdots $ obtained by concatenating successive geodesic segments together is a quasi-geodesic in $\cH$. In particular, $\lim_{i\to\infty} x_i \in \partial \cH$ exists. 
\end{lem}

\begin{proof}
This follows from standard arguments. For example, see \cite[Chapter III.H, Theorem 1.13]{bridson-haefliger-book}.
\end{proof}

%It will be convenient to have the following pseudo-metric.
%\begin{defn}[Leafwise pseudo-metric]
%We define a leafwise pseudo-metric $d_\sigma:\cR  \to [0,\infty)$ by 
%$$d_\sigma(x,y) = d(\sigma(x),\alpha(x,y)\sigma(y)) = d(\alpha(y,x)\sigma(x),\sigma(y)).$$
%It is a metric if $\alpha(x,y)\sigma(y) \ne \sigma(x)$ for any $x\cR y$ with $x\ne y$ (i.e. if $\sigma_x:[x]_\cR \to \sH_x$ is an injection). In this case $\sigma_x$ is an isometric embedding of $[x]_\cR$ into $\sH_x$. 
%\end{defn}

\begin{proof}[Proof of Theorem \ref{thm:non-nested}]
Because $\cR$ is parabolic there exists a unique $\cR$-invariant section $\xi:X \to \cL$. To obtain a contradiction, suppose $\cS\le \cR $ is loxodromic on a subset of positive measure. After passing to this subset if necessary, we may assume that $\cS$ is loxodromic. Let $\nu \in \Fix(\cS)$ be such that the support of $\nu_x$ contains 2 points for a.e. $x$. %Because $\cS$ is loxodromic, there exists a $\cS$-invariant field of boundary measures $\nu$ such that the cardinality of the support of $\nu_x$ is two for a.e. $x$. The support of $\nu_x$ necessarily contains $\xi(x)$ since $\cS$ also fixes the field of measures $(1/2)[\nu + \delta_{\xi(x)}]$ and the support of any invariant field of boundary measures has cardinality $\le 2$ (Theorem \ref{thm:K}). 

\noindent {\bf Claim 1}. For a.e. $x\in X$, $\cup_{y\cR x} \alpha(x,y)\cL_y(\cS)$ is an infinite subset of $\cL_x(\cR)$.

\begin{proof}[Proof of Claim 1]
Suppose not. By ergodicity there is a natural number $k$ such that $|\cup_{y\cR x} \alpha(x,y)\cL_y(\cS)|=k$ for a.e. $x$. Let $\eta_x$ be the probability measure uniformly supported on $\cup_{y\cR x} \alpha(x,y)\cL_y(\cS)$. Then $\eta$ is $\cR$-invariant. By Theorem \ref{thm:K} this implies $k\le 2$. By Lemma \ref{lem:loxodromic-case}, $|\cL_y(\cS)|=2$ for a.e. $y$. So $k=2$. However, this implies $\cR $ is loxodromic, a contradiction. 
\end{proof}

%For $x\in X$, let $\gamma_x \subset \sH_x$ be a geodesic with endpoints in the support of $\nu_x$. We may choose $\gamma_x$ to depend measurably on $x$. For $s>0$, let $X_s$ be the set of all $x\in X$ such that $d(\sigma(x),\gamma_x) \le s$. Then $\cup_{s>0} X_s = X$; so there exists $s>0$ such that $\mu(X_s)>0$. Wlog assume $X_s=X$ **explain**.

For $s>0$, let $X_s$ be the set of all $x\in X$ such that there exists a geodesic $\gamma_x \subset \sH_x$ whose endpoints are in the support of $\nu_x$ satisfying $d(\sigma(x),\gamma_x)\le s$. Then $\cup_{s>0} X_s = X$; so there exists $s>0$ such that $\mu(X_s)>0$. By Lemma \ref{lem:compression}, we may assume without loss of generality that $X=X_s$.

Let $r>0$ be large enough so that if $z_1,z_2,\ldots$ is any sequence of points in the $\delta$-hyperbolic space $\sH_x$ satisfying (a) $d(z_i,z_{i+1})\ge r$ for all $i$ and (b) $(z_i|z_{i+2})_{z_{i+1}} \le 4s+22\delta$ for all $i$ then the path obtained by concatenating geodesic segments $[z_i,z_{i+1}]$  is a quasi-geodesic. Such a number $r>0$ exists by Lemma \ref{lem:local-global}. Moreover it does not depend on $x$ because $\mu$ is $\cR$-ergodic.

Let $\beta:\cR  \to \R$ be the Busemann function associated to $\xi$ defined by
$$\beta(y,z) = \sup_{\{p_i\}} \limsup_{i\to\infty} d(\sigma(y),p_i) - d(\sigma_y(z),p_i)$$
where the sup is over all sequences $\{p_i\}_{i\in \N} \subset \sH_y$ such that $\lim_{i\to\infty} p_i= \xi(x)$.

\noindent {\bf Claim 2}. For a.e. $x\in X$ there exists $y\in X$ satisfying $y\cR x$ and 
$$\max(d_\sigma(x,y)-4s-20\delta,r) \le \beta(x,y)\le d_\sigma(x,y).$$

\begin{proof}[Proof of Claim 2]
In general $|\beta(x,y)|\le d_\sigma(x,y)$ holds because of the triangle inequality. So the last inequality is immediate.

We claim that $\xi(x) \in \cL_x(\cS)$ for a.e. $x$. By Lemma \ref{lem:loxodromic-case} $\cL_x(\cS)$ is the support of $\nu_x$. Define a field of boundary measures $\lambda$ by
$$\lambda_x= \frac{\nu_x + \delta_{\xi(x)}}{2}.$$
If $\xi(x)$ is not in the support of $\nu_x$ then the support of $\lambda$ has three elements in contradiction to Theorem \ref{thm:K}. So $\xi(x) \in \cL_x(\cS)$.

By choice of $s$, there exists a geodesic $\gamma_x$ with endpoints in the support of $\nu_x$ and $d(\sigma(x),\gamma_x)\le s$. Because $\xi(x) \in \cL_x(\cS)$ there exists  $y\in X$ with $y\cS x$ and $r\le \beta(x,y)$. Let $x',y' \in \gamma_x$ be points closest to $\sigma(x),\sigma_x(y)$, respectively. So $d(\sigma(x),x')\le s$. By choice of $s$ there exists a geodesic $\gamma'_x$ with endpoints in the support of $\nu_x$ such that $d(\sigma_x(y),\gamma'_x)\le s$. By Lemma \ref{lem:thin}, $\gamma_x$ and $\gamma'_x$ are $2\delta$-close. So $d(\sigma_x(y),y')\le s+2\delta$. 

By Lemma \ref{lem:geodesic}, since $x',y'$ lie on the same geodesic with an endpoint in $\xi(x)$, $\beta(x',y') \ge d(x',y') - 2\delta$. By the quasi-cocycle inequality (Lemma \ref{lem:buse}) and the triangle inequality,
$$|\beta(x,y)-\beta(x',y')|  \le |\beta(x,x')| + |\beta(y,y')| + 8\delta \le d(\sigma(x),x') + d(y',\sigma_x(y)) \le 2s  + 10\delta.$$
So we have
$$\beta(x,y) \ge \beta(x',y') -2s -10\delta \ge d(x',y') - 2s -16 \delta \ge d_\sigma(x,y) - 4s - 20\delta.$$
\end{proof}

We would like to define a map $\Phi:X \to X$ by: $\sigma_x(\Phi(x))$ is the closest point in the orbit $\cO^\sigma_x$ to $\sigma(x)$ that lies on the geodesic  ray $[\sigma(x),\xi(x)]$. But of course, the orbit might not intersect any such geodesic ray in more than one point. So we use the Busemann function associated to $\xi(x)$ to define what it means for a point to be ``closer'' to $\xi(x)$. Fix an injective Borel map $\phi:X \to [0,1]$ that we will use to break `ties'. 

Now define $\Phi:X \to X$ by $\Phi(x)=y$ where $y \in [x]_\cR$ satisfies
\begin{itemize}
\item $\max(d_\sigma(x,y)-4s-20\delta,r) \le \beta(x,y)$,
\item $d_\sigma(x,y) =\inf\{d_\sigma(x,z) :~ z\in [x]_\cR, \max(d_\sigma(x,z)-4s-20\delta,r) \le \beta(x,z) \} $
\item if there is more than one element $y\in [x]$ satisfying the above conditions then we choose $y$ to be the unique one minimizing $\phi(\cdot)$.
\end{itemize}
By Claim 2 and metrical properness of $\sigma$, $\Phi$ is well-defined. 

To simplify notation, if $x,y,z$ are any points in the same $\cR$-class then we let
$$(x|y)_z^\sigma:=(\sigma(x)|\sigma_x(y))_{\sigma_x(z)}.$$

\noindent {\bf Claim 3}. For a.e. $x\in X$ the path obtained by concatening geodesic segments $[\sigma_x(\Phi^n(x)),\sigma_x(\Phi^{n+1}(x))]$ ($n\in \Z$)  is a quasi-geodesic and $\lim_{n\to\infty} \sigma_x(\Phi^n(x)) = \xi(x).$ 

\begin{proof}[Proof of Claim 3]
By the quasi-cocycle identity (Lemma \ref{lem:buse}),
\begin{eqnarray*}
d_\sigma(\Phi^{n}x, \Phi^{n+2}x) &\ge& \beta(\Phi^{n}x, \Phi^{n+2}x) \ge \beta(\Phi^{n}x, \Phi^{n+1}x) + \beta(\Phi^{n+1}x, \Phi^{n+2}x) - 4\delta\\
&\ge& d_\sigma(\Phi^{n}x, \Phi^{n+1}x) + d_\sigma(\Phi^{n+1}x, \Phi^{n+2}x) - 8s - 44\delta.
\end{eqnarray*}
Thus
$$(\Phi^n(x)|\Phi^{n+2}(x))_{\Phi^{n+1}(x)}^\sigma \le 4s + 22\delta.$$
Since $d_\sigma(\Phi^n(x),\Phi^{n+1}(x))\ge \beta(\Phi^n(x),\Phi^{n+1}(x))\ge r$, the Claim now follows from the choice of r.
\end{proof}

For $x\in X$, let $G_x$ be the directed graph with vertex set $[x]_\cR$ and directed edges $(x,\Phi(x))$. Let $G'_x$ be the induced subgraph of $G_x$ containing all vertices that lie in a bi-infinite directed path of $G_x$. In other words, $G'_x$ is the maximal subgraph of $G_x$ that does not have any vertices of degree 1. 

\noindent {\bf Claim 4}. For a.e. $x\in X$, $G'_x$ is nonempty. In fact, for every limit point $\eta \in \cup_{y\cR x} \alpha(x,y)\cL_y(\cS)$ with $\eta \ne \xi(x)$, there exists a bi-infinite directed path in $G'_x$ with endpoints $\{\xi(x),\eta\}$.

\begin{proof}[Proof of Claim 4]
Let $y \in [x]_\cR$. By Lemma \ref{lem:loxodromic-case} there is a point $\eta \ne \xi(x)$ such that $\alpha(x,y)\cL_y(\cS) =\{\eta,\xi(x)\}$ and elements $y_n \in [y]_\cS$ such that $\lim_{n\to\infty} \sigma_x(y_n) = \eta$. Let $H_n$ be the subgraph of $G_x$ induced by the trajectory $\{\Phi^m(y_n)\}_{m\in \N}$. Observe that if $\gamma_x \subset \sH_x$ is any geodesic from $\eta$ to $\xi(x)$ then the distance between $\sigma_x(\Phi^m(y_n))$ and $\gamma_x$ is bounded by a constant independent of $n,m$. This is because the path obtained by concatening geodesic segments $[\sigma_x(\Phi^n(x)),\sigma_x(\Phi^{n+1}(x))]$ ($n\in \Z$) is a $(\lambda,C)$-quasi-geodesic limiting on $\xi(x)$ (for some $\lambda,C>0$ independent of $n$) and $d(\sigma_x(y_n),\gamma_x)\le s+2\delta$. It follows that the subgraph $H_n$ has a subsequential limit $H_\infty$ in the space of all subgraphs of $G_x$ (endowed with the topology of pointwise convergence on compact subsets) and $H_\infty$ lies in the $(s+2\delta)$-neighborhood of $\gamma_x$. Since $y_n \to \eta$ as $n\to\infty$, this implies that $H_\infty$ is the required bi-infinite directed path.
\end{proof}

Let $Z$ denote the set of all $x \in X$ such that $x$ is a vertex of $G'_x$. Define  $F:\cR \to [0,1]$ by $F(x,y)=1$ if $x\in Z$ and $y=\Phi(x)$. Let $F(x,y)=0$ otherwise. By the Mass-Transport Principle (Lemma \ref{lem:mtp}),
$$\mu(Z)\le \int_Z \sum_x F(x,y)~d\mu(y) \le \int \sum_x F(x,y)~d\mu(y) = \int \sum_y F(x,y)~d\mu(x) = \mu(Z).$$
This implies $\sum_x F(x,y)=1$ for a.e. $y\in Z$. So every vertex of $G'_x$ is adjacent to exactly one incoming edge and one outgoing edge. In particular, every vertex of $G'_x$ has degree 2. By Claims 1 and 4, this implies that $G'_x$ has infinitely many connected components. 

%**we should be a little careful about saying that $\{\Phi^n(y)\}_{n\in \N}$ is a quasi-geodesic since we do not mean this in the sense of the usual definition. We don't mean the map $n \mapsto \sigma_x(\Phi^n(y))$ is a quasi-isometric embedding of the natural numbers since we do not bound the distance between successive elements**

\noindent {\bf Claim 5}. There is a constant $C>0$ such that for any $y,z \in [x]_\cR$ there exists $N \in \N$ such that for every $n \ge N$ there exists $m\in \N$ such that $d_\sigma(\Phi^n y, \Phi^m z) < C$.

\begin{proof}[Proof of Claim 5]
Let $\gamma_y$ be a geodesic from $\sigma_x(y)$ to $\xi(x)$. Because the path obtained by concatening geodesic segments $[\sigma_x(\Phi^n(y)),\sigma_x(\Phi^{n+1}(y))]$ ($n\in \Z$) is a $(\lambda,C)$-quasi-geodesic limiting on $\xi(x)$  (for some $\lambda, C>0$ independent of $y$), there is a constant $C'$ such that $\{\sigma_x(\Phi^n(y))\}_{n\in \N}$ lies in the $C'$-neighborhood of $\gamma_y$. Similarly, if $\gamma_z$ is a geodesic from $\sigma_x(z)$ to $\xi(x)$ then  $\{\sigma_x(\Phi^n(z))\}_{n\in \N}$ lies in the $C'$-neighborhood of $\gamma_z$. The two geodesics $\gamma_y$ and $\gamma_z$ are asymptotic. By Lemma \ref{lem:thin} there is an $N$ and constant $\delta'>0$ such that for all $n>N$, $\sigma_x(\Phi^n(y))$ lies in the $(C'+2\delta')$-neighborhood of $\gamma_z$. The statement now follows from the definition of $\Phi$.
\end{proof}

For $n>0$, let $Y_n$ be the set of all $x\in X$ such that there exist at most $n$ points $y \in [x]_\cR$ with $d_\sigma(x,y)\le C$ where $C>0$ is as in Claim 5. Because $\sigma$ is metrically proper, $X= \cup_{n=1}^\infty Y_n$. So there exists an $n$ such that $Y_n$ has positive measure. By replacing $\cR$ with $\cR \resto Y_n$ and invoking Lemma \ref{lem:compression} we may assume without loss of generality that $X=Y_n$. 

We have now arrived at a contradiction. To be precise, let $x\in X$ be a typical element. By Claim 1, there exist distinct elements $\eta_1,\ldots, \eta_{n+1} \in \cup_{y\cR x} \alpha(x,y)\cL_y(\cS)$ such that $\eta_i \ne \xi(x)$ for all $i$. By Claim 4, for each $i$ there a directed path $\gamma_i$ in $G'_x$ with endpoints $\{\eta_i,\xi(x)\}$. By the paragraph before Claim 5, the paths $\gamma_1,\ldots, \gamma_{n+1}$ are pairwise disjoint. By Claim 5, there exists $y \in [x]_\cR$ and $y_i \in \gamma_i$ such that $d_\sigma(y,y_i) \le C$ for all $i$. This contradicts the choice of $n$.

\end{proof}

\section{Maximal hyperfinite subequivalence relations}\label{sec:maximal}

\begin{defn}
A hyperfinite subequivalence $\cS\le\cR $ is {\bf maximal hyperfinite} if there does not exist a hyperfinite subequivalence relation $\cS' \le \cR$ with $\cS\le\cS'$ and $\cS' \setminus \cS$ non-null (with respect to the usual measure $\hmu$ on $\cR$, see \S \ref{sec:mer}).
\end{defn}

We say $\cR$ has {\bf rank 1} if every aperiodic hyperfinite subequivalence relation is contained in a {\em unique} maximal hyperfinite subequivalence relation. We will show that if the Main Assumption is satisfied then $\cR$ is rank 1 and characterize rank 1 in terms of quasi-normalizers of hyperfinite subequivalence relations. We also use this to prove that certain equivalence relations {\em do not} satisfy the Main Assumption. 

\begin{lem}\label{lem:maxxx}
If $\cS\le\cR $ is any hyperfinite subequivalence relation then there exists a maximal hyperfinite subequivalence relation $\cS'\le\cR $ with $\cS\le\cS'$.
\end{lem}

\begin{proof}
We say that two subequivalence relations $\cS_1,\cS_2$ are {\bf $\hmu$-equivalent} if $\hmu(\cS_1\vartriangle\cS_2)=0$ where $\hmu$ is the usual measure on $\cR$ (see \S \ref{sec:mer}). In general, we will not distinguish between a subequivalence relation and its $\hmu$-class. Let $\sZ$ denote the collection of all (equivalence classes of) Borel hyperfinite subequivalence relations $\cS'$ of $\cR$ with $\cS\le \cS'\le \cR$. The set $\sZ$ is partially ordered by inclusion mod $\hmu$. By Zorn's Lemma there exists a maximal chain $\sC \subset \sZ$. 

Let $\mu'$ be a probability measure on $\cR$ that is equivalent to $\hmu$ (in the sense that two measures have the same measure zero sets). Let $\beta = \sup \{\mu'(\cT):~\cT \in \sC\}$. For each integer $n\ge 1$, let $\cS_n$ be an element of the chain $\sC$ satisfying $\mu'(\cS_n) \ge \beta - 1/n$. Since $\sC$ is a chain, we must have that $\cS_1\le \cS_2 \le \ldots$ (with inclusions taken mod $\hmu$). Since each $\cS_i$ is hyperfinite, the union $\cS_\infty:=\cup_{i=1}^\infty \cS_i$ is also hyperfinite. Moreover, $\mu'(\cS_\infty)=\beta$. Since $\sC$ is a maximal chain, we must have that $\cS_\infty \in \sC$. Again since $\sC$ is a maximal chain and $\mu'(\cS_\infty)=\beta$, $\cS_\infty$ must be a maximal hyperfinite subequivalence relation.
\end{proof}

\begin{defn}
A pmp equivalence relation $\cR $ has {\bf rank 1} if every pair of distinct maximal hyperfinite subequivalence relations intersects in a finite subequivalence relation. Equivalently, $\cR $ has rank 1 if every aperiodic hyperfinite subequivalence relation is contained in a unique maximal hyperfinite subequivalence relation (where uniqueness is up to measure zero). We say $\cR $ has {\bf higher rank} if it does not have rank 1. This definition is motivated by the theory of semisimple Lie groups: a semisimple Lie group $G$ has real rank 1 if and only if every closed noncompact unimodular amenable subgroup is contained in a unique maximal unimodular amenable subgroup.
\end{defn}

\begin{thm}\label{thm:max}
Hyperbolic equivalence relations have rank 1.
%If $\cR $ is hyperbolic and $\cS\le\cR $ is aperiodic and hyperfinite then there exists a unique maximal hyperfinite subequivalence relation $\cS'\le\cR $ with $\cS\le\cS'$. Uniqueness is up to measure zero.
\end{thm}

%To prove this, we need to introduce

%we introduce the notion of a canonical invariant field of boundary measures.
\begin{defn}\label{defn:canonical-field}
Every aperiodic hyperfinite subequivalence relation $\cS \le \cR$ has a unique {\bf canonical invariant field of boundary measures} defined as follows. By Lemma \ref{lem:mixed}, there is a unique (up to null sets) partition $X=Y\sqcup Z$ such that $\cS\resto Y$ is parabolic and $\cS\resto Z$ is loxodromic. Let $\nu \in \Fix(\cS\resto Y)$ denote the unique element and let $\eta \in \Fix(\cS\resto Z)$ denote the unique element satisfying: $\eta_z$ has two atoms of equal mass $1/2$ for a.e. $z\in Z$. Finally, define $\omega \in \Fix(\cS)$ by $\omega_y=\nu_y$ for $y\in Y$ and $\omega_z=\eta_z$ for $z\in Z$. 
\end{defn}

\begin{lem}\label{cor:non-nested}
Suppose the Main Assumption is satisfied. Let $\cS \le \cT \le \cR$ be aperiodic hyperfinite subequivalence relations. Then the canonical $\cS$-invariant field of boundary measures is the same as the canonical $\cT$-invariant field of boundary measures.
\end{lem}

\begin{proof}
Let $X=Y\sqcup Z$ and $\nu,\eta,\omega$ be as in Definition \ref{defn:canonical-field}. Note that an inclusion of subequivalence relations always implies the reverse inclusion of corresponding fixed points in the space of fields of boundary measures. Therefore, $\Fix(\cT \resto Y) \subset \Fix(\cS \resto Y) = \{\nu\}$. Since $\Fix(\cT \resto Y)$ is nonempty (Theorem \ref{thm:K}), this implies $\Fix(\cT \resto Y) = \{\nu\}$. In particular, $\cT \resto Y$ is parabolic. 

Since $\cS\resto Z$ is loxodromic and $\cS \resto Z \le \cT \resto Z$, Theorem \ref{thm:non-nested} implies $\cT\resto Z$ is also loxodromic. To be precise, Theorem \ref{thm:non-nested} apriori only applies to ergodic equivalence relations. However, by decomposing a nonergodic measure into its ergodic components, we see that Theorem \ref{thm:non-nested} extends to nonergodic equivalence relations as well. So there exists a measure $\kappa \in \Fix(\cT \resto Z)$ such that for a.e. $z\in Z$, $\kappa_z$ has two atoms of equal mass $1/2$. Because $\cS \le \cT$, we have $\kappa \in \Fix(\cS\resto Z)$. We claim that $\kappa=\eta$. Indeed, if this is not true then $(1/2)(\kappa+\eta)_z$ has support containing more than 3 elements (for all $z$ is a set with positive measure), contradicting Theorem \ref{thm:K}. So $\kappa=\eta$ which implies the Corollary.
\end{proof}

\begin{proof}[Proof of Theorem \ref{thm:max}]
Let $\cR $ be a hyperbolic equivalence relation and $\cS\le\cR $ an aperiodic hyperfinite subequivalence relation. Let $\omega \in \Fix(\cS)$ be the canonical $\cS$-invariant field of boundary measures (Definition \ref{defn:canonical-field}).  Let $\cM$ be the set of all $(x,y) \in \cR$ such that $\alpha(x,y)_*\omega_y=\omega_x$. In other words, $\cM$ is the stabilizer of $\omega$ in $\cR$. Clearly, $\cS \le \cM$. It is easily checked that $\cM$ is a subequivalence relation. It must be hyperfinite by Theorem \ref{thm:K}.

 In other to show that it is maximal, let $\cK$ be a measurable hyperfinite subequivalence relation with $\cM \le \cK \le \cR$. By Lemma \ref{cor:non-nested}, $\omega \in \Fix(\cK)$. The definition of $\cM$ now implies $\cK=\cM$. 
 \end{proof}

%Next we obtain some basic properties and consequences of rank 1 equivalence relations.

%\begin{lem}
%Rank 1 is closed under subequivalence relations (i.e. if $\cS\le\cR $ and $\cR $ has rank 1 then $\cS$ has rank 1).
%\end{lem}

%\begin{proof}
%Let $\cR $ have rank 1 and let $\cS\le\cR$. Let $\cH\le \cS$ be an aperiodic hyperfinite subequivalence relation. Let $\cM\le \cR$ be the unique maximal hyperfinite subequivalence relation with $\cH \le \cM$. We claim that $\cS \cap \cM$ is the unique maximal hyperfinite subequivalence relation of $\cS$ containing $\cH$. So suppose $\cT$ is a hyperfinite subequivalence relation with $\cH \le \cT \le \cS$.  Since $\cT$ is hyperfinite, it must be contained in $\cM$. So $\cT \le \cS \cap \cM$. This proves $\cS \cap \cM$ is the unique maximal hyperfinite subequivalence relation of $\cS$ containing $\cH$.
%\end{proof}

%\begin{lem}
%Suppose $\cR $ has two commuting subequivalence relations $\cS_1, \cS_2$ and $\cS_2$ is non-hyperfinite. Then $\cR $ has higher rank.
%\end{lem}

%\begin{proof}
%To obtain a contradiction, let us suppose that $\cR $ has rank 1.

%\cSor $i=1,2$ let $\cS'_1\le\cS_1$ be a maximal hyperfinite subequivalence relation of $\cS_1$. Let $M$ be the unique maximal hyperfinite subequivalence relation of $\cR $ containing $\cS'_1$. If $G\le\cS_2$ is  aperiodic and hyperfinite then, because $\cS_1$ and $\cS_2$ commute, $\cS'_1 \vee G$ is hyperfinite. So $\cS'_1 \vee G \le M$. Since this is true for every $G\le\cS_2$, we must have that $\cS_2\leM$. But $\cS_2$ is non-hyperfinite.
%\end{proof}

%The next lemma generalizes the previous one.

\begin{defn}
Let $\cS\le \cR$ be a subequivalence relation. For any $\phi \in \Aut(X,\mu)$ we let 
$$\phi_*\cS =\{ (\phi(x),\phi(y)):~ (x,y)\in \cS\}.$$
If $\phi \in [\cR]$ then this is a subequivalence relation of $\cR$. The {\bf quasi-normalizer} of $\cS$ in $\cR$ is the subgroup $N^q_\cR(\cS)\le [\cR]$ generated by all $\phi \in [\cR]$ such that $\phi_*\cS \cap \cS$ is aperiodic. The {\bf normalizer} of $\cS$ in $\cR$ is the subgroup $N_\cR(\cS) \le [\cR]$ generated by all $\phi \in [\cR]$ such that $\phi_*\cS=\cS$. So if $\cS$ is aperiodic then $N_\cR(\cS) \le N^q_\cR(\cS)$.
%The subequivalence relation generated by $\cS$ and $N^q_\cR(\cS)$ is called the {\bf quasi-normal closure of $\cS$}. To be precise, the quasi-normal closure of $\cS$ is the smallest subequivalence relation containing both $\cS$ and $\{(x,\phi x):~\phi \in N^q_\cR(\cS), x\in X\}$. 
\end{defn} 

\begin{defn}
Let $\Aut(X,\mu)$ denote the group of measure-preserving Borel isomorphisms of $(X,\mu)$. Let $\cS \le \cR$ be a subequivalence relation and $G\subset \Aut(X,\mu)$ a subset. We let $\langle \cS,G\rangle$ denote the smallest equivalence relation on $X$ containing both $\cS$ and $\{(x,g x):~g \in G, x\in X\}$. 
\end{defn}

%Our main Theorem is:
\begin{thm}\label{thm:norm}
Let $(X,\mu)$ be a standard probability space and $\cR \subset X\times X$ a discrete Borel equivalence relation such that $\mu$ is $\cR$-invariant and $\cR$-ergodic. If $\cR$ has rank 1 then for every aperiodic hyperfinite subequivalence relation $\cS\le \cR$, $\langle \cS, N^q_\cR(\cS)\rangle$ is hyperfinite.
\end{thm}
It is an open problem whether the converse holds. The next result is essentially the same as \cite[Lemmas 2.9(ix) and 2.15(ix)]{MR1900547}. We give a proof for the reader's convenience.
\begin{lem}
Let $(X,\mu)$ denote a standard probability space. Let $\cR \subset X\times X$ be a hyperfinite Borel equivalence relation on $X$ and assume $\mu$ is $\cR$-invariant. Suppose $\phi \in \Aut(X,\mu)$ normalizes $\cR$ (this means that $ x\cR y \Rightarrow (\phi x)\cR (\phi y)$). Then $\langle \cR,\phi\rangle$ is hyperfinite.
\end{lem}

\begin{proof}
Given a discrete Borel equivalence relation $\cS$ on $(X,\mu)$, we let  $[\cS]$ acts on $L^\infty(\cS)$ by 
$$\theta f(x,y) := f(\theta^{-1}x,y), \quad \forall \theta \in [\cS], f \in L^\infty(\cS).$$
We let $L^\infty(\cS)^*$ denote the Banach dual of $L^\infty(\cS)$. We also let $[\cS]$ act on $L^\infty(\cS)^*$ by
$$(\theta \Phi)(f) = \Phi(\theta^{-1}f), \quad \forall \theta \in [\cS], f \in L^\infty(\cS), \Phi \in L^\infty(\cS)^*.$$

Following \cite[Definition 5]{CFW81} a countable pmp equivalence relation $\cS$ on $(X,\mu)$ is called {\bf amenable} if there exists a state $\Lambda:L^\infty(\cS) \to \C$ such that $\Lambda$ is $[\cS]$-invariant and  $\Lambda(g)=\int g~d\mu$ for all $g\in L^\infty(X)$ where we have embedded $L^\infty(X)$ into $L^\infty(\cR)$ via $g(x,y)=g(x)$ for $g \in L^\infty(X)$. By \cite[Theorem 10]{CFW81} a countable pmp equivalence relation $\cS$ is amenable if and only if it is hyperfinite. 

Since $\cR$ is hyperfinite, there exists a state $\Lambda:L^\infty(\cR) \to \C$ satisfying the above requirements.  By viewing $\cR \subset \langle \cR,\phi\rangle$ we obtain the restriction map $\Res: L^\infty(\langle \cR,\phi\rangle) \to L^\infty(\cR)$.  Now define $\Psi_n \in L^\infty(\langle \cR,\phi\rangle)^*$ by
$$\Psi_n(f) = \frac{1}{n}\sum_{i=0}^n \Lambda (\Res(  \phi^{-i} f )).$$
%where $f\circ \rm{id}\times \phi^i \in L^\infty(\langle \cR,\phi\rangle)$ is the map $f\circ \rm{id}\times \phi^i (x,y)=f(x,\phi^iy)$. 

Since each $\Psi_n$ is contained in the unit ball of $L^\infty(\langle \cR,\phi\rangle)^*$, the Banach-Alaoglu Theorem implies the existence of a weak* limit point, denoted $\Psi$, of $\{\Psi_n\}$. We claim that $\Psi$ is a state verifying the properties of amenability.

First we observe that if $\theta \in [\cR]$ then
$$\Psi_n(\theta f) = \frac{1}{n}\sum_{i=0}^n \Lambda (\Res( \phi^{-i} \theta f) )).$$
 Since $\Lambda$ is $[\cR]$-invariant and $\phi^{-i} \theta \phi^i \in [\cR]$,
 $$\Lambda(\Res \phi^{-i} \theta \phi^i \phi^{-i} f) = \Lambda(\Res \phi^{-i} f).$$
 Hence  $\Psi_n(\theta f) = \Psi_n(f)$ and so $\Psi(\theta f) = \Psi(f)$ for any $\theta \in [\cR], f\in L^\infty(\langle \cR,\phi\rangle)$. 

By construction, we have that $\Psi(  \phi f) = \Psi(f)$ for any $f\in \langle \cR,\phi\rangle$. So $\Psi$ is both $\phi$-invariant and $[\cR]$-invariant. Since $\phi$ and $[\cR]$ generate $[\langle \cR,\phi\rangle]$, this proves $\Psi$ is $[\langle \cR,\phi\rangle]$-invariant.

Next, let $g\in L^\infty(X) \subset L^\infty(\langle \cR,\phi\rangle)$. Note that $\Res(\phi^{-i} g)(x,y)=g(\phi^i x) = \phi^{-i}g(x)$. So
$$\Psi_n(g) = \frac{1}{n}\sum_{i=0}^n \Lambda (\Res( \phi^{-i} g) )) = \frac{1}{n}\sum_{i=0}^n \Lambda (\phi^{-i}g ).$$
Since $\Lambda(\phi^{-i}g) = \int \phi^{-i}g~d\mu = \int g~d\mu$ (since $\phi$ is measure-preserving), we have $\Psi_n(g)=\int g~d\mu$ for all $n$. So $\Psi(g)=\int g~d\mu$. Since $g$ is arbitrary, this proves that $\Psi$ is an invariant mean. Thus $\langle \cR,\phi\rangle$ is amenable.
\end{proof}

\begin{proof}[Proof of Theorem \ref{thm:norm}]
Suppose $\cR$ has rank 1 and $\cS\le \cR$ is aperiodic and hyperfinite. Then $\cS$ is contained in a unique maximal hyperfinite subequivalence relation $\cM\le \cR$. Let $\phi \in [\cR]$ quasi-normalize $\cS$. Since $\cS \cap \phi_*\cS$ is aperiodic it must be that $\cM \cap \phi_*\cM$ is aperiodic too. Because $\cM$ is the unique maximal hyperfinite subequivalence relation containing $\cM \cap \phi_*\cM$ and $\phi_*\cM$ is also a maximal hyperfinite subequivalence relation, it follows that $\cM=\phi_*\cM$. Thus $N^q_\cR(\cS)$ normalizes $\cM$. The previous lemma now implies the theorem.

\end{proof}

We now show that various kinds of equivalence relations are higher rank and in particular, cannot satisfy the Main Assumption (which implies that they do not admit hyperbolic graphings as in Example \ref{ex:graphing} and cannot be treeable).

\begin{cor}\label{cor:SLnZ}
Suppose $\cR $ is the orbit-equivalence relation for an ergodic pmp essentially free action of $SL(n,\Z)$ for some $n\ge 3$. Then $\cR$ has higher rank.
\end{cor}

\begin{proof}
For $i,j \in \{1,\ldots, n\}$ with $i\ne j$ let $E_{ij} \in SL(n,\Z)$ denote the matrix with 1's on the diagonal, a 1 in the $(i,j)$-th entry and $0$'s everywhere else. Let $\cR_{ij} \le\cR$ denote the orbit subequivalence relation generated by $E_{ij}$: $\cR_{ij} = \{(x,E_{ij}^mx):~x\in X, m \in \Z\}$. Each $\cR_{ij}$ is aperiodic and hyperfinite. 

To obtain a contradiction, suppose $\cR$ has rank 1 and let $\cM\le \cR$ denote the unique maximal hyperfinite subequivalence relation containing $\cR_{12}$. Because $E_{12}$ commutes with $E_{1k}$ ($2\le k \le n$), it follows the subequivalence relation $\cR_{12} \vee \cR_{1k}$ is also hyperfinite and therefore $\cR_{12}\vee \cR_{1k} \le \cM$. Since $E_{1k}$ commutes with $E_{jk}$ (for $j\ne k$) it follows that $\cR_{1k} \vee \cR_{jk}$ is hyperfinite and therefore $\cR_{jk} \le \cM$ for $j\ne k$. Since $\{E_{jk}\}_{j\ne k}$ generates $SL(n,\Z)$, it follows that $\cM=\cR$. This contradicts the fact that, since $SL(n,\Z)$ is non-amenable, $\cR$ is non-hyperfinite.
\end{proof}

%Given subequivalence relations $\cS,\cT \le \cR$ we let $\cS\circ \cT=\{(x,z):~\exists y, ~(x,y) \in \cS, (y,z)\in \cT\}$. This is not a subequivalence relation in general. We say $\cS,\cT$ {\bf commute} if $\cS\circ \cT = \cT \circ \cS$. 

\begin{cor}\label{cor:product}
Let $A,B$ be countably infinite groups, $G=A\times B$, $G \cc (X,\mu)$ an essentially free ergodic pmp action and $\cR \subset X\times X$ the orbit-equivalence relation. If $B$ is non-amenable and $A$ contains an infinite amenable subgroup $A' \le A$ then $\cR$ has higher rank.
\end{cor}

\begin{proof}
Let $\cS,\cR_B$ denote the orbit subequivalence relations generated by $A',B$ respectively. If we view $B$ as a subgroup of $[\cR]$ then $B$ is contained in the normalizer of $\cS$. Thus the subequivalence relation generated by $B$ and $\cS$ contains $\cR_B$ and is therefore non-hyperfinite (since $B$ is non-amenable). By Theorem \ref{thm:norm}, this implies $\cR$ has higher rank.

%By a general result \cite[Theorem 3.5]{Kechris-global-aspects}, there exists an aperiodic hyperfinite subequivalence relation $\cS \le \cR_A$. If we view $B$ as a subgroup of $[\cR]$ then $B$ is contained in the normalizer of $\cS$. Thus the subequivalence relation generated by $B$ and $\cS$ contains $\cR_B$ and is therefore non-hyperfinite (since $B$ is non-amenable). By Theorem \ref{thm:norm}, this implies $\cR$ has higher rank.

%Because $\cS_i$ is aperiodic, there exists an aperiodic hyperfinite subequivalence relation $\cT_i\le \cS_i$ ($i=1,2$). Observe that $[\cS_2]$ normalizes $\cT_1$. Indeed, if $\phi  \in [\cS_2]$ and $(x,y) \in \cT_1$ then $(x,\phi x) \in \cS_2$ so $(y,\phi x) \in \cT_1 \circ \cS_2$ so there exists $w$ such that $(y,w) \in \cS_2$ and $(w,\phi x)\in \cS_1$. 

\end{proof}
It is an open problem whether the hypothesis that $A$ contain an infinite amenable subgroup can be removed.

\section{Loxodromic elements}

The purpose of this section is to prove:

\begin{thm}\label{thm:exists}
Suppose the Main Assumption is satisfied and $\cR$ is non-hyperfinite. Then there exists an ergodic loxodromic element $\phi \in [\cR]$. Moreover, we can choose $\phi$ such that for a.e. $x\in X$ the path in $\sH_x$ obtained by concatenating geodesic segments $[\sigma_x(\phi^n(x)),\sigma_x(\phi^{n+1}(x))]$ is a quasi-geodesic.

% such that for a.e. $x\in X$, $\{\sigma_x(\phi^n(x))\}_{n\in \Z}$ is a quasi-geodesic.
\end{thm}

For the rest of this section, we assume the hypotheses of Theorem \ref{thm:exists} above. 

\subsection{Proof sketch}\label{sec:sketch}

To prove Theorem \ref{thm:exists} we will build the loxodromic element $\phi \in [\cR]$ as an increasing limit of partial transformations. To be precise, a {\bf partial transformation} of $\cR$ is a measure-space isomorphism 
$$\phi:\dom(\phi) \to \rng(\phi)$$
where $\dom(\phi),\rng(\phi) \subset X$ and the graph of $\phi$ is contained in $\cR$. We let $\lb\cR\rb$ denote the set of all such partial transformations. As per our usual convention, we identify two partial transformations that agree almost everywhere.

\begin{defn}
Let $F \subset \cR$ be measurable. A {\bf directed matching} in $F$ is a partial transformation $\phi \in \lb\cR\rb$ such that the graph of $\phi$ is contained in $F$ and $\dom(\phi) \cap \rng(\phi)$ has measure zero. It is {\bf perfect} if $X = \dom(\phi) \cup \rng(\phi)$ (modulo measure zero sets). 
\end{defn}

For example, if $\phi_1,\phi_2$ are two perfect directed matchings such that $\rng(\phi_1)=\dom(\phi_2)$ and $\rng(\phi_2)=\dom(\phi_1)$ then the composition $\phi_2\circ\phi_1 \in [\cR]$. 

%The proof of Theorem \ref{thm:exists} uses Lemma \ref{lem:local-global} which, roughly speaking, states that a path

By Lemma  \ref{lem:local-global} there are constants $r,s>0$ such that to prove the existence of a not-necessarily-ergodic loxodromic element $\phi \in [\cR]$ it suffices to construct two directed perfect matchings $\phi_1,\phi_2$ such that 
\begin{itemize}
\item $\rng(\phi_1) =\dom(\phi_2)$, $\rng(\phi_2)=\dom(\phi_1)$,
\item $d_\sigma(x,\phi_i(x) ) \ge r$ for a.e. $x$ and $i=1,2$,
\item $(x|\phi_2\circ \phi_1x)^\sigma_{\phi_1x} \le s$ for a.e. $x$ 
\end{itemize}
where
$$(x|z)_y^\sigma = \frac{1}{2} \Big( d_\sigma(x,y) + d_\sigma(y,z) - d_\sigma(x,z) \Big).$$
(Recall that $d_\sigma(x,y) = d(\sigma(x),\alpha(x,y)\sigma(y))$.) Then $\phi:=\phi_2\circ \phi_1$ is loxodromic because the path obtained by concatenating geodesic segments $[\sigma_x(\phi^i(x)),\sigma_x(\phi^{i+1}x)]$ is a quasi-geodesic. 

To obtain such matchings, we will use the hypotheses of ``extreme expansitivity'' (explained next). This hypothesis is stronger than the ones studied in \cite{lyons-nazarov, csoka-lippner} and the proofs are simpler although less constructive. The construction of an {\em ergodic} loxodromic element   is a bit more involved since it requires that we control the averages $\frac{1}{n}\sum_{i=1}^n f(\phi^i(x))$ for test functions $f\in L^2(X,\mu)$. However, the general principle is the same.

\subsection{Matchings}

\begin{defn}
A measurable set $F \subset \cR$ is {\bf extremely expansive} for every pair of non-null sets $A, B \subset X$,  $F \cap (A\times B)$ is non-null (with respect to the usual measure $\hmu$ on $\cR$, see \S \ref{sec:mer}).
\end{defn}

The usefulness of this condition is that it implies the existence of perfect matchings:
\begin{prop}\label{prop:expansive}
If $F \subset \cR$ is extremely expansive then for every pair of non-null disjoint sets $A, B \subset X$ with $\mu(A)=\mu(B)$ there exist a directed matching $\phi \in \lb \cR\rb$ with $\grph(\phi) \subset F \cap (A\times B)$ and $\dom(\phi)=A, \rng(\phi)=B$. In particular, there exists a perfect directed matching.
\end{prop}

To prove this, we need a short lemma first.

\begin{lem}\label{lem:partial-matching}
If $F \subset \cR$ is extremely expansive and $A,B \subset X$ are disjoint sets each with positive measure then there exists a directed matching $\phi \in \lb \cR \rb$ with $\mu(\dom(\phi))>0$, $\grph(\phi) \subset F \cap (A\times B)$. 
\end{lem}

\begin{proof}
By \cite[Theorem 1]{feldman-moore-1}, there exists a countable subset $\cF \subset [\cR ]$ such that for a.e. $(x,y)\in \cR $ there exists $f\in \cF$ with $f(x)=y$. Because $F$ is extremely expansive, there exists $f\in \cF$ such that
$$\rm{graph}(f) \cap F \cap (A\times B) $$
is non-null. So we define $\phi$ so that its graph equals $\rm{graph}(f) \cap F \cap (A\times B)$. 
\end{proof}

\begin{proof}[Proof of Proposition \ref{prop:expansive}] 

%Let $A,B \subset X$ be disjoint measurable sets with $X = A \cup B$ and $\mu(A)=\mu(B)=1/2$.

Let $\cC$ be the set of all $\phi \in \lb \cR \rb$ such that  $\grph(\phi) \subset F \cap (A\times B)$.  There is a natural partial order on $\cC$ given by $\phi \le \psi$ if $\dom(\phi) \subset \dom(\psi)$ $\mod \mu$ and $\psi(x)=\phi(x)$ for a.e. $x\in \dom(\phi)$.

By a standard measure exhaustion argument (as in Lemma \ref{lem:maxxx}), there exists a maximal element $\phi \in \cC$. If $\dom(\phi)$ is not co-null in $A$ then the previous lemma implies the existence of a measure-space isomorphism
$$\psi:(A\setminus \dom(\phi)) \to (B \setminus \rng(\phi))$$
with graph contained in $F$. Setting $\Phi=\phi \sqcup \psi$ yields an element of $\cC$ that is greater than $\phi$ in the partial ordering, contradicting the maximality of $\phi$. So $\dom(\phi)=A$ up to measure-zero. Since $\phi$ is measure-preserving, $\rng(\phi)=B$ up to measure zero. In particular, if $A \cup B=X$ then $\phi$ is perfect.
\end{proof}

\subsection{Matchings in $\cR $ via hyperbolic geometry}
%We assume the hypotheses of Theorem \ref{thm:exists}.
 %in order to use Lemma \ref{lem:local-global}, we first need to choose a number $s>0$. So 
 
 For any $t>0$ let $X_t$ be the set of all $x\in X$ such that there exist distinct elements $\xi,\eta \in \partial \sH_x$ with
$$(\xi|\eta)_{\sigma(x)} < t.$$
Because $X= \cup_{t>0} X_t$, there exists $s>\delta$ such that $\mu(X_{s-\delta})>0$. Later, we will use this value of $s$ in Lemma \ref{lem:local-global}. The next proposition is the key tool for proving the existence of directed matchings satisfying geometric constraints.

\begin{prop}\label{prop:key}
Let $\psi:X_{s-\delta} \to X_{s-\delta}$ be a Borel map with graph in $\cR$ (we do not require that $\psi$ is invertible). Also let $r>0$. Let $F$ be the set of all $(x,y) \in \cR \resto X_{s-\delta}$ such that 
$$d_\sigma(x,y)>r,\quad (y|\psi x)^\sigma_{x} < s,~ \rm{and}~(x|\psi y)^\sigma_y < s.$$
Then $F$ is extremely expansive in $Y \times Y$ where $Y=X_{s-\delta}$.
\end{prop}

%\begin{remark}
%This proposition is the only place in the proof of Theorem \ref{thm:exists} where we use our hypothesis that the closed orbits $\overline{\cO_x}$ are compact.
%\end{remark}

\begin{proof}
Let $A,B \subset Y$ be non-null sets. It suffices to show that $F \cap (A \times B)$ is non-null.

Given $(x,y) \in \cR  \resto Y$ and $s>0$, let $\rm{\rm{Shadow}}_s(x,y) \subset \bsH_x$ be the set of all $p\in \bsH_x$ such that 
$$(\sigma_x(y)|p)_{\sigma(x)} \ge s.$$
This is a closed subset. %For $x\in Y$, let 
%$$K_x = \bigcap_{y \in A \cap [x]_\cR} \alpha(x,y)\rm{Shadow}_s(y,\psi y) \subset \bsH_x.$$
%and set $\partial K_x = K_x \cap \partial \sH_x$. %Observe that $x\mapsto \partial K_x$ is an invariant section of closed subsets (for $\cR  \resto Y$). Therefore, by Theorem \ref{thm:minimal} either $\partial K_x=\cL_x$ for a.e. $x$ or $\partial K_x = \emptyset$ for a.e. $x$. 

\noindent {\bf Claim 1}. For a.e. $x\in Y$ and $y\in A$, the limit set $\cL_x$ is not contained in $\alpha(x,y){\rm{Shadow}}_s(y,\psi y)$.

\begin{proof}[Proof of Claim 1]
Let $\xi,\eta \in {\rm{Shadow}}_s(y,\psi y)$ be distinct elements. By Gromov's inequality (see \S \ref{sec:hyperbolic}), 
$$(\xi|\eta)_{\sigma(y)} \ge \min\{ (\xi|\psi y)_{\sigma(y)}, (\psi y|\eta)_{\sigma(y)} \} - \delta \ge s-\delta.$$
Because $y \in Y$ there exist distinct elements $\xi',\eta' \in \cL_y$ such that 
$$s-\delta > (\xi'|\eta')_{\sigma(y)}.$$
Therefore, $\xi',\eta'$ cannot both be in ${\rm{Shadow}}_s(y,\psi y)$. Thus $\cL_x$ is not contained in $\alpha(x,y){\rm{Shadow}}_s(y,\psi y)$. 
\end{proof}

\noindent {\bf Claim 2}. For $x\in Y$, let
$$K_x := \bigcap_{y \in A \cap [x]_\cR} \alpha(x,y){\rm{Shadow}}_s(y,\psi y).$$
Then $K_x=\emptyset$ for a.e. $x$.
%$K_x = \emptyset$ for a.e. $x$.

\begin{proof}[Proof of Claim 2]
Note $K_x$ is closed and invariant (in the sense that $\alpha(x,y)K_y=K_x$). By Claim 1 and Lemmas \ref{lem:greenberg} and \ref{limit-set1}, $\Hull(K_x) \cap \sH_x = \emptyset$ for a.e. $x$. So either $K_x = \emptyset$ for a.e. $x$ or $\Hull(K_x)$ consists of a single point in $\partial \sH_x$ for a.e. $x$. However the latter possibility implies that $\cR$ is hyperfinite (Theorem \ref{thm:K}) contradicting our hypotheses.
\end{proof}

By symmetry, the claims above also hold with $B$ in place of $A$. Because $\overline{\cO^\sigma_x} \subset \bsH_x$ is compact, Claim 2 implies that for a.e. $x$ there exist finite sets $A_x \subset A \cap [x]$ and $B_x \subset B \cap [x]$ such that
$$\overline{\cO^\sigma_x} \cap \bigcap_{y \in A_x} \alpha(x,y){\rm{Shadow}}_s(y,\psi y) = \overline{\cO^\sigma_x} \cap \bigcap_{y \in B_x} \alpha(x,y){\rm{Shadow}}_s(y,\psi y) =\emptyset.$$

For $D>0$, let $X_D$ be the set of all $x\in Y$ such that there exist finite sets $S \subset A \cap [x], T \subset B \cap [x]$ satisfying
\begin{itemize}
\item $\rm{diam}(S) \le D$ and $\rm{diam}(T) \le D$ (where diameter is computed with respect to $d_\sigma$),
\item $$\overline{\cO^\sigma_x} \cap \bigcap_{y \in T} \alpha(x,y){\rm{Shadow}}_s(y,\psi y)=\overline{\cO^\sigma_x} \cap \bigcap_{y \in S} \alpha(x,y){\rm{Shadow}}_s(y,\psi y) = \emptyset.$$
\end{itemize}
Because $Y = \cup_{D>0} X_D$, there exists some $D>0$ such that $X_D$ has positive measure. Since $X_D$ is $\cR$-invariant, this implies $Y=X_D$ (up to measure zero).

%Let $J_x$ be the closure in $\bsH_x$ of the set of all $\sigma_x(y)$ such that $y$ is in a finite subset $S \subset A \cap [x]$ of diameter $\le D$ satisfying
%$$\bigcap_{y \in S} \alpha(x,y){\rm{Shadow}}_s(y,\psi y) = \emptyset.$$
%Because $x\mapsto \partial J_x$ is $\cR $-invariant, Theorem \ref{thm:minimal} implies $\partial J_x = \cL_x$ for a.e. $x$. 

Let $I_x$ be the closure in $\bsH_x$ of the set of all $\sigma_x(y)$ such that $y$ is in a finite subset $S \subset B \cap [x]$ of diameter $\le D$ satisfying
$$ \overline{\cO^\sigma_x} \cap \bigcap_{y \in S} \alpha(x,y){\rm{Shadow}}_s(y,\psi y) = \emptyset.$$
Because $x\mapsto \partial I_x$ is $\cR $-invariant (where $\partial I_x = I_x \cap \cL_x$) Theorem \ref{thm:minimal} implies $\partial I_x = \cL_x$ for a.e. $x$. 

So for a.e. $a\in A$ and every $\xi \in \cL_a$ there exists a sequence $\{S_i\}$ of finite sets $S_i \subset B \cap [a]$ satisfying $\diam(S_i)\le D$ and 
$$\overline{\cO^\sigma_x} \cap \bigcap_{y \in S_i} \alpha(a,y){\rm{Shadow}}_s(y,\psi y) = \emptyset$$
while $\lim_i  \sigma_a(S_i) = \xi$. By Claim 1 we may choose $\xi \in \cL_a \setminus {\rm{Shadow}}_s(a,\psi a)$. So there must exist $b\in B$ such that 
\begin{itemize}
\item $\sigma(b) \notin \alpha(b,a){\rm{Shadow}}_s(a,\psi a)$ and $\sigma(a) \notin \alpha(a,b){\rm{Shadow}}_s(b,\psi b)$
\item $d_\sigma(a,b) >r$.
\end{itemize}
Equivalently $(a,b) \in F$. Note that $F$ is a Borel set and we just proved that almost every vertical section of $F \cap (A \times B)$ is nonempty. So $\hmu(F \cap A \times B) \ge \mu(A)>0$. This proves $F$ is extremely expansive.
\end{proof}

%Let us now see how to finish the proof of Proposition \ref{thm:exists}.
\begin{proof}[Proof of Theorem \ref{thm:exists}]
As above, let $s>0$ be such that $\mu(X_{s-\delta})>0$. Let $r>0$ be as in Lemma \ref{lem:local-global}. To simplify notation, let $Y=X_{s-\delta}$ and $\mu_Y$ be the probability measure on $Y$ obtained by restricting $\mu$ and normalizing. We will first show that there exists an ergodic loxodromic element $\Phi\in [\cR \resto Y]$.

Recall the definition of partial transformation from \S \ref{sec:sketch}. There is a natural partial order on $\lb \cR \resto Y\rb$ given by $\phi \le \psi$ if $\dom(\phi) \subset \dom(\psi)$ $\mod \mu$ and $\psi(x)=\phi(x)$ for a.e. $x\in \dom(\phi)$ (this is just containment of graphs mod $\hmu$). 

Let $\cF_\infty \subset L^2(Y)$ be the set of all functions $f$ such that
\begin{itemize}
\item $0 \le f \le 1$,
\item the map $t \mapsto \mu(\{x\in Y:~ t>f(x)\})$ is strictly increasing for $0<t<1$.
\end{itemize}
Then $\cF_\infty$ has dense linear span in $L^2(Y)$. So there exists an increasing sequence  $\{\cF_i\}_{i=1}^\infty$  of finite subsets $\cF_i \subset \cF_\infty$ whose union has dense linear span in $L^2(Y)$. We let $\| \cdot \|$ denote the $L^2$-norm on $L^2(Y)$.

By induction we will construct an increasing sequence $\{\phi_i\}_{i\in \N} \subset \lb \cR \resto Y \rb$ satisfying the following.
\begin{enumerate}
\item[0.] $Y = \dom(\phi_1) \cup \rng(\phi_1)$ (modulo null sets).
\item If $\cS_i \le (\cR\resto Y)$ is the subequivalence relation generated by $\phi_i$ (so $x\cS_i y \Leftrightarrow \exists n\in \Z$ such that $\phi_i^n(x)=y$) then almost every $\cS_i$-equivalence class is finite.
%\item for a.e. $x\in X$, there is a natural number $n=n(i,x)$ such that $x \notin \dom(\phi^n_i)$
%\item $Y = \dom(\phi_i) \cup \rng(\phi_i)$ (up to measure zero)

%item $\emptyset = \dom(\phi_i) \cap \rng(\phi_i)$.

\item For a.e. $x$, $[x]_{\cS_{i-1}}\varsubsetneqq [x]_{\cS_{i}}$ (if $i>1$).

\item $d_\sigma(x,\phi_i(x)) \ge r$ for a.e. $x \in \dom(\phi_i)$.
\item $(x|\phi^2_i(x))_{\phi(x)}^\sigma \le s$ for a.e. $x\in \dom(\phi_i^2)$.

\item If $A_i:L^2(Y) \to L^2(Y)$ is the averaging operator defined by 
$$A_i(f)(x)= \frac{1}{|[x]_{\cS_i}|} \sum_{y\in [x]_{\cS_i}} f(y)$$
then $\|A_{i}(f) - \int_{Y} f~d\mu_Y\| \le 2\|f\|/i$ for all $f\in \cF_i$.

%$A_i(f)$ converges to the constant $\int f~d\mu$ in $L^2$ as $i\to\infty$ for every $f\in L^2(Y)$.

%\item $\cup_i \dom(\phi_i) = Y$ (up to measure zero).
\end{enumerate}

Suppose for the moment that we have constructed such a sequence. Define $\phi \in \lb \cR \resto Y\rb$ by $\phi(x)=\phi_i(x)$ for $x \in \dom(\phi_i)$. Because the $\phi_i$'s are increasing this is well-defined. By item (2), almost every orbit of $\phi$ is infinite. Therefore $\dom(\phi) = \rng(\phi)$ modulo null sets. By item (0), $\dom(\phi)=Y$ ($\mod \mu$) So $\phi \in [\cR\resto Y]$. 

Let $\cS_\infty = \cup_i \cS_i$. Observe that $\cS_\infty$ is the sub-equivalence relation generated by $\phi$. Recall that a subset $Z \subset Y$ is $\cS_\infty$-saturated if it is a union of $\cS_\infty$-classes. By the Martingale Convergence Theorem, for every $f\in \cup_i \cF_i$, $A_i(f)$ converges to $\E[f|\cS_\infty]$ which denotes the conditional expectation of $f$ on the sigma-algebra of $\cS_\infty$-saturated measurable sets. Item (5) now implies that for any $f\in \cup_i \cF_i$, $ \E[f|\cS_\infty]=\int f~d\mu_Y$ is constant. Since $\cup_i \cF_i$ is dense in $L^2(Y)$, it follows by continuity that $f\mapsto \E[\cdot | \cS_\infty]$ is projection onto the constants for all $f\in L^2(Y)$. Therefore $\cS_\infty$ is ergodic. Because $\phi$ generates $\cS$, $\phi$ is ergodic. 

Conditions (3,4) and Lemma \ref{lem:local-global} imply that for a.e. $x$ the orbit $\{\phi^n(x)\}_{n\in \Z}$ is a quasi-geodesic with respect to $d_\sigma$. Thus $\phi$ is loxodromic as required.

We will construct the $\phi_i$'s by induction. Here is the base case: let $F_1$ be the set of all $(x,y) \in \cR \resto Y$ such that $d_\sigma(x,y)\ge r$. Clearly, $F_1$ is extremely expansive. So Proposition \ref{prop:expansive} implies the existence of a perfect directed matching $\phi_1 \in \lb\cR \resto Y\rb$ whose graph is contained in $F_1$; that is $d_\sigma(x,\phi_1(x))\ge r$ for a.e. $x$. 

Observe that $\phi_1$ satisfies (0-5). It satisfies item (0) since $\phi_1$ is perfect. The equivalence class of $x$ generated by $\phi_1$ is just $[x]_{\cS_1}=\{x,\phi_1(x)\}$. This proves (1). Item (3) holds by design. Items (2,4) hold vacuously since $\dom(\phi_1^2)=\emptyset$. Item (5) holds because $\|A_i(f)\| \le \|f\|$ for any $f\in L^2(Y)$ since $\phi_1$ is measure-preserving.

Now suppose that $\phi_1,\ldots, \phi_i$ have been constructed so that items (1-5) above hold and $\phi_1$ is a perfect matching. We now turn to constructing $\phi_{i+1}$. Because $\phi_1$ is a perfect matching and $\phi_1 \le \phi_i$ it follows that $\dom(\phi_i)\cup\rng(\phi_i) = Y$ (up to null sets). So we have a natural partition
$$Y = (Y\setminus \rng(\phi_i)) \sqcup (\dom(\phi_i) \cap \rng(\phi_i)) \sqcup (Y \setminus \dom(\phi_i)).$$
Because every $f \in \cF_{i+1}$ is bounded  there exists a partition $\cP$ of $Y \setminus \rng(\phi_i)$ such that for each $f\in \cF_{i+1}$, $P\in \cP$ and $x,y \in P$  $$|A_if(x)-A_if(y)|\le \|f\|/(i+1).$$
By further refining $\cP$ and perturbing it slightly (using the fact that the map $t \mapsto \mu(\{x\in Y:~ t>f(x)\})$ is strictly increasing for $0<t<1$) we may assume that $\cP=\{P_1,\ldots, P_n\}$ for some $n\ge 2$ and  $\mu_Y(P_i)=\mu_Y(P_j)$ for all $i,j$.

Let $P'_j$ be the set of all $x \in Y \setminus \dom(\phi_i)$ such that there exists $y \in P_j$ with $\phi_i^m(y)=x$ for some $m\ge 1$. So $\cP'=\{P'_1,\ldots, P'_n\}$ is a partition of $Y\setminus \dom(\phi_i)$ and $\mu_Y(P_j)=\mu_Y(P'_j)$ for all $j$.

Let $\psi:X_{s-\delta} \to X_{s-\delta}$ be the map 
\begin{displaymath}
\psi(x)=\left\{\begin{array}{cc}
\phi_i(x) & x\in \dom(\phi_i) \\
\phi_i^{-1}(x) & x\notin \dom(\phi_i)
\end{array}\right.\end{displaymath}

Let $F$ be the set of all pairs $(x,y) \in \cR \resto Y$ such that 
%\item $x\notin \dom(\phi_i)$, $y \notin \rng(\phi_i)$
%\item $x\in P'_j$, $y\in P_{j+1}$,
$$d_\sigma(x,y)\ge r, \quad (y|\psi x)_x^\sigma \le s, \quad (x|\psi y)_y^\sigma \le s.$$
By Proposition \ref{prop:key}, $F$ is extremely expansive. By Proposition \ref{prop:expansive},  there exists a partial transformation $\psi_j \in \lb\cR \rb$ such that $\dom(\psi_j) = P'_j$, $\rng(\psi_j) = P_{j+1}$ and $\grph(\psi_j) \subset F \cap (P'_j\times P_{j+1})$ ($1\le j<n$). 

We define $\phi_{i+1}$ by: $\phi_{i+1}(x)=\phi_i(x)$ if $x\in \dom(\phi_i)$ and $\phi_{i+1}(x)=\psi_j(x)$ if $x \in P'_j$ for some $1\le j < n$. Because $\phi_{i+1}$ is not defined if $x \in P'_n$ each $\cS_{i+1}$-class is finite. This proves (1). Items (2) and (3) are immediate. 

To check item (4) let $x \in \dom(\phi_{i+1}^2)$. If $x \in \dom(\phi_i^2)$ then (4) holds by the induction hypothesis. If $x \in \dom(\phi_i)$ but not in $\dom(\phi_i^2)$ then $\phi_{i+1}(x) = \phi_i(x)$ and $\phi_{i+1}^2(x) =\psi_j\phi_i(x)$ for some $j$. Since $\grph(\psi_j) \subset F$, $(\phi_{i+1}x,\phi_{i+1}^2x) \in F$ which implies
$$(\phi_{i+1}^2 x|\psi \phi_{i+1} x)_{\phi_{i+1} x}^\sigma \le s.$$
Since $\psi \phi_{i+1} x = \psi \phi_i x = x$ this shows
$$(\phi_{i+1}^2 x| x)_{\phi_{i+1} x}^\sigma \le s.$$
The last case, when $x \notin \dom(\phi_i)$, is similar.

Next we check item (5). Note that for each $1\le j \le n$, every $\cS_{i+1}$-class contains exactly one $\cS_i$-class that nontrivially intersects $P_j$. Also $A_{i+1}(f)(x)$ is a convex sum of $A_i(f)(y)$ where $y$ varies over any set of representatives of the $\cS_i$-classes in $[x]_{\cS_{i+1}}$. Because $|A_if(x)-A_if(y)|\le \|f\|/(i+1)$ for each $x,y \in P_j$ and $f\in \cF_{i+1}$, it follows that $|A_{i+1}(f)(x)-A_{i+1}(f)(y)| \le \|f\|/(i+1)$ for each $x,y \in Y$. This implies
\begin{eqnarray*}
 \left|A_{i+1}(f)(x) - \int_{Y} f~d\mu_Y \right| = \left|A_{i+1}(f)(x) - \int A_{i+1}(f)~d\mu_Y\right|   \le \|f\|/(i+1).
\end{eqnarray*}

This proves the induction step. As explained above, this shows the existence of an ergodic loxodromic element $\phi \in [\cR\resto Y]$.

It is left to construct an ergodic loxodromic subequivalence relation of $\cR$ (as opposed to $\cR  \resto Y)$. For this purpose, let $\Psi:X \to Y$ be any measurable map with graph contained in $\cR $ satisfying $\Psi(x)=x$ for all $x\in Y$. Define $\cS \le \cR $ by: $x\cS y \Leftrightarrow$ $\Psi(x)$ and $\Psi(y)$ are in the same $\phi$ orbit. Note that $\cS \resto Y$ is the subequivalence relation generated by $\phi$. It follows that $\cS$ is ergodic. Let $\eta\in \Fix(\phi)$ and define $\nu \in \Fix(\cS)$ by $\nu_x = \alpha(x,y)_*\eta_y$ for any $y\in Y$ with $(x,y)\in \cS$ (for example, $y=\Psi(x)$). Because $\eta$ is fixed by $\phi$, $\nu$ is fixed by $\cS$. It follows that, for a.e. $x\in X$, the support of $\nu_x$ contains 2 elements. So $\cS$ is loxodromic.

It is a well-known fact that any ergodic equivalence relation contains an ergodic element in its full group \cite[Theorem 3.5]{Kechris-global-aspects}. So there exists an ergodic $\psi \in [\cS]$. Because $\psi \in [\cS]$, $\Fix(\cS) \subset \Fix(\psi)$. Therefore $\psi$ is also loxodromic.

\end{proof}

\section{Tits' alternative}\label{sec:tits}

The goal of this section is to prove Theorem \ref{thm:tits}. We do this by first constructing an $\F_2$-action on a positive measure subset $Y \subset X$ with orbits contained in $\cR\resto Y$. In order to prove that the constructed action is essentially free we will use the following criterion:

\begin{lem}\label{lem:local-global2}
Let $(\cH,d)$ denote a complete $\delta$-hyperbolic geodesic metric space. Let $T_4$ denote the 4-regular tree. Then for every $s>0$ there exists an $r>0$ (depending only on $s$ and $\delta$) such that if $\phi:T_4 \to \cH$ is any map satisfying
\begin{itemize}
\item for every edge $\{v,w\}$ of $T_4$, $d(\phi(v),\phi(w)) \ge r$,
\item for any three distinct vertices $u,v,w \in V(T_4)$ such that $\{u,v\},\{v,w\}$ are edges of $T_4$, 
$$(\phi(u)|\phi(w))_{\phi(v)} \le s$$
\end{itemize}
then $\phi$  is injective. Moreover, if we consider $T_4$ to be a metric tree in which each edge $\{v,w\}$ has length equal to $d(\phi(v),\phi(w))$ and $\phi$ maps the edge from $v$ to $w$ isometrically onto a geodesic segment $[\phi(v),\phi(w)]$, then $\phi$ is a quasi-isometric embedding. 
\end{lem}

\begin{proof}
This follows from Lemma \ref{lem:local-global}.
\end{proof}

%Recall that a {\bf graphing} of a subequivalence relation $\cS \le \cR$ is a measurable subset $G \subset \cS$ such that: (a) $G$ is symmetric (so $(x,y) \in G \Rightarrow (y,x) \in G$) and (b) if, for $x\in X$, $G_x$ is the graph with vertex set $[x]_\cS$ and edges $\{ \{y,z\}:~(y,z)\in G, x\cS y \cS z\}$ then $G_x$ is connected.

%A graphing is a {\bf treeing} if for a.e. $x$, $G_x$ is a tree ; that is to say $G_x$ is connected and has no cycles. We say that $\cS$ is {\bf treeable} if it admits a treeing. 

The next lemma proves Theorem \ref{thm:tits} up to a compression of $\cR$.

\begin{lem}\label{lem:Tits}
Suppose the Main Assumption is satisfied and $\cR$ is non-hyperfinite.  Then there exists a subset $Y \subset X$ with $\mu(Y)>0$ and an essentially free ergodic action $\F_2 \cc Y$ of the free group of rank 2 whose orbits are contained in $\cR \resto Y$. In fact, we obtain a subequivalence relation $\cS \le \cR \resto Y$ and a treeing $\cT$ of $\cS$ such that 
\begin{itemize}
\item for a.e. $y\in Y$, $\cT_y$ is a 4-regular tree (so $\cS$ is treeable)
\item if we consider $\cT_y$ to be a metric tree such that each edge $\{w,z\}$ has length $d(\sigma_y(w),\sigma_y(z))$ then $\sigma_y$ gives a quasi-isometric embedding of $\cT_y$ into the fiber $\sH_y$.
\end{itemize}
%for a.e. $y\in Y$, the graph $G_y$ with vertex set $[y]_\cS$ and edge set $\{ \{w,z\}:~(w,z)\in G, w\cS z\cS y\}$ is quasi-isometric to a 4-regular metric tree 

%In fact, we obtain elements $a,b \in [\cR\resto Y]$ such that $a,b$ freely generate a rank 2 free group and for a.e. $y\in Y$ the .....***

%we obtain a subequivalence relation $\cS \le \cR \resto Y$ and a graphing $G$ of $\cS$ such that for a.e. $y\in Y$, the graph $G_y$ with vertex set $[y]_\cS$ and edge set $\{ \{w,z\}:~(w,z)\in G, w\cS z\cS y\}$ is quasi-isometric to a 4-regular metric tree 
\end{lem}

\begin{proof}
By Theorem \ref{thm:exists} there exists an ergodic loxodromic element $f \in [\cR ]$ such that for a.e. $x\in X$, the path $\gamma_x$ obtained by concatenating consecutive geodesic segments $[\sigma_x(f^nx),\sigma_x(f^{n+1}x)]$ is a quasi-geodesic. Let $\cL_x(f) \subset \partial \sH_x$ denote the limit set of the orbit $\{\sigma_x(f^nx)\}_{n\in \Z}$. Because $\gamma_x$ is a quasi-geodesic, the limit points
$$f^{-\infty}(x) := \lim_{n\to-\infty} \sigma_x(f^n(x)), \quad f^{+\infty}(x) := \lim_{n\to\infty} \sigma_x(f^n(x))$$
exist and $\cL_x(f)=\{f^{-\infty}(x), f^{+\infty}(x)\}$. Moreover, if $\xi^-_x,\xi^+_x$ are the Dirac measures concentrated on $f^{-\infty}(x),$ $f^{+\infty}(x)$ respectively then both $\xi^-$ and $\xi^+$ are invariant fields of boundary measures for the subequivalence relation $\cR_f \le \cR$ generated by $f$.

%Define
%$$f^{+\infty}(x) = \lim_{n\to\infty} \sigma_x(f^n(x)), \quad f^{-\infty}(x) = \lim_{n\to-\infty} \sigma_x(f^n(x)).$$
%We may assume $f$ is chosen to be {\bf monotone} in the sense that for any $x$ if $\gamma_x$ is a geodesic in $\cH_x$ from $f^{-\infty}(x)$ to $f^{+\infty}(x)$ and $\pi:\cH_x \to \gamma_x$ is a projection map then $\pi(\sigma(fx))$ lies between $\pi(\sigma x)$ and $\pi(\sigma f^2(x))$ for a.e. $x$.

% $\sigma_x(f(x))$ is at least as close to $f^{+\infty}(x)$ as $\sigma(x)$ is. To be precise, this means that that if $\xi =f^{+\infty}(x)$ then $h_\xi(\sigma_x(f(x))) \le h_\xi(\sigma(x))$ where $h_\xi$ is the horofunction associated to $\xi$ as in \S {sec:horofunctions}.

\noindent {\bf Claim}.
There exists a loxodromic element $g \in [\cR]$ such that $\cL_x(f) \cap \cL_x(g)  = \emptyset$ for a.e. $x$. In fact, we can choose $g$ to be a conjugate of $f$ in $[\cR]$.

\begin{proof}
%To obtain a contradiction, suppose the claim is false. Let $g \in [\cR]$ be loxodromic and let $\cS \le \cR$ be the subequivalence relation generated by $f,g$. Since $\{f^{-\infty}(x),f^{+\infty}(x)\} \cap \{g^-(x),g^+(x)\} \ne \emptyset$, it follows that $\Fix(\cS)$ is nonempty. By Theorem \ref{thm:K}, $\cS$ is hyperfinite. By Theorem \ref{thm:non-nested}, $\cS$ cannot be parabolic. Since $f\in [\cS]$, $\cS$ is ergodic and so $\cS$ is loxodromic. This implies $\{f^{-\infty}(x),f^{+\infty}(x)\} = \{g^-(x),g^+(x)\}$. So we have shown that every loxodromic element of $\cR$ has the same limit points as $f$. We will show that this implies $\Fix(\cR)\ne\emptyset$. 

Let $\nu \in \Fix(f)$ be the canonical $f$-invariant field of boundary measures (so $\nu_x=\xi^-_x/2+\xi^+_x/2$). 

To obtain a contradiction, suppose the claim is false. Let $h\in [\cR]$ be arbitrary. Since $\Fix(hfh^{-1})=h\Fix(f)$ and $h\nu \in \Fix(hfh^{-1})$ is such that $(h\nu)_x$ is supported on two elements, it must be that $hfh^{-1}$ is loxodromic. Moreover, $h\nu$ is the canonical $hfh^{-1}$-invariant field of boundary measures.

Let $\cS$ denote the subequivalence relation generated by $f$ and $hfh^{-1}$. Then $\cS$ fixes $\cL(f)\cap \cL(hfh^{-1})$. Since we are assuming $\cL_x(f) \cap \cL_x(hfh^{-1}) \ne \emptyset$, Theorem \ref{thm:K} implies $\cS$ is hyperfinite. So Lemma \ref{cor:non-nested} implies $\nu$ and $h\nu$ are both the canonical $\cS$-invariant field of boundary measures. In particular, $h\nu=\nu$. Since $h \in [\cR]$ is arbitrary,  $\nu \in \Fix(\cR)$. By Theorem \ref{thm:K} this contradicts the assumption that $\cR$ is not hyperfinite. 

%Let $\nu \in \Fix(f)$ be the canonical invariant field of boundary measures (so $\nu_x$ has two atoms of measure 1/2 each). 
%Let $g \in [\cR]$ be arbitrary. Let $\nu \in \Fix(f)$. Without loss of generality, we may assume $\nu(\{f^{-\infty}(x)\})=\nu(\{f^{+\infty}(x)\})=1/2$. Since $\Fix(gfg^{-1})=g\Fix(f)$, $g\nu \in \Fix(gfg^{-1})$. Note $\nu_x$ is supported on   $\{f^{-\infty}(x),f^{+\infty}(x)\}$. So $(g\nu)_x$ is supported on $\{ (gfg^{-1})^-(x),(gfg^{-1})^+(x)\}$. Since $gfg^{-1}$ is loxodromic, we must have that $(g\nu)_x$ is supported on $\{f^{-\infty}(x),f^{+\infty}(x)\}$. So $g\nu=\nu$. Since $g$ is arbitrary, this implies $\nu \in \Fix(\cR)$. By Theorem \ref{thm:K} this contradicts the assumption that $\cR$ is hyperfinite. 
\end{proof}
Let $g=hfh^{-1}$ be a loxodromic element with $\cL_x(f) \cap \cL_x(g)  = \emptyset$. Because $f^{-\infty}:X \to \cL$ is an $f$-invariant section, $hf^{-\infty}$ defined by $(hf^{-\infty})_x:= \alpha(x,h^{-1}x)f^{-\infty}_{h^{-1}x}$ is an $hfh^{-1}$-section. Similarly, $hf^{+\infty}$ is an $hfh^{-1}$ section. 

For $t>0$, let $X_t$ be the set of all $x\in X$ such that
\begin{enumerate}
\item $(f^\epsilon(x)|hf^\eta(x))_x^\sigma \le t~\forall \epsilon, \eta \in \{-\infty,+\infty\}$;
\item if $\gamma$ is any geodesic with endpoints in $\cL_x(f)$ then $d(\sigma(x),\gamma) \le t$;
\item if $\gamma$ is any geodesic with endpoints in $\cL_x(g)$  then $d(\sigma(x),\gamma) \le t$.
\end{enumerate}
Then $X=\cup_{s>0} X_s$. So there exists $t>0$ with $\mu(X_t)>0$. 

Given $u,v,w\in \bsH_x$ and a geodesic $\gamma$ we say that $v$ lies {\bf between} $u$ and $w$ with respect to $\gamma$ if: for every triple $u',v',w' \in \gamma$ of points such that $u'$ is a closest point to $u$ on $\gamma$ (and similarly with $v',w'$) we have that $u'$ and $w'$ are in different components of $\gamma - \{v'\}$.

By continuity and the choice of $f$ there exists a number $N>0$ satisfying the following.
\begin{enumerate}
\item  If $u,v,w \in \sH_x$ are $t$-close to a geodesic $\gamma \subset \sH_x$, $d(u,v)\ge N, d(v,w)\ge N$ and $v$ lies between $u$ and $w$ with respect to $\gamma$ then $(u|w)_v \le t+1$. 
\item If $u,v,w \in \sH_x$ and $\gamma_1,\gamma_2 \subset \sH_x$ are geodesics such that 
\begin{itemize}
\item $(\gamma^{\pm}_1| \gamma^{\pm}_2)_v \le t$ (where $\gamma^{\pm}_i \subset \partial \sH_x$ are the endpoints of $\gamma_i$),
\item $d(u,\gamma_1)\le t, d(v,\gamma_1)\le t, d(v,\gamma_2)\le t, d(w,\gamma_2)\le t$, 
\item $d(u,v)\ge N, d(v,w)\ge N$
\end{itemize}
then $(u|w)_v \le t+1$.
\item If $n>0$ and $d(\sigma(x),\sigma_x(f^nx)) \ge N$ then $\sigma_x(f^nx)$ lies between $\sigma(x)$ and $\sigma^+(x)$ with respect to any geodesic with endpoints in $\{f^{-\infty}(x),f^{+\infty}(x)\}$.
\end{enumerate}

Let $s=t+1$. Let $r$ be as in Lemma \ref{lem:local-global2}. We assume without loss of generality that $r>N$. 

Let $G$ be the set of all pairs $(x,y) \in \cR \resto X_t$ such that $d_\sigma(x,y) \le r$. Since $\sigma$ is proper, $G$ is a locally finite in the sense that for a.e. $x \in X$ there are only finitely many $y$ with $(x,y)\in G$ (equivalently $(y,x) \in G)$).  By \cite[Proposition 4.5]{KST99}  there exists a proper Borel vertex coloring $K:X_t \to \N$ of $G$. In other words, if $(x,y) \in G$ then $K(x)\ne K(y)$. It follows that, for some $n\in \N$, $K^{-1}(n)$ has positive measure. Let $X'_t=K^{-1}(n)$. The relevant properties of $X'_t$ are: $X'_t \subset X_t$ and if $x\ne y \in X'_t$ and $x\cR y$ then $d_\sigma(x,y)> r$. 

%a subset $X'_t \subset X_t$ with $\mu(X'_t)>0$ such that if $x\ne y \in X'_t$ and $x\cR y$ then $d_\sigma(x,y)\ge r$.

 Let $f_0\in [\cR\resto X_t']$ be the first-return time map of $f$ to $X'_t$. Precisely, $f_0(x)=f^n(x)$ where $n>0$ is the smallest positive integer such that $f^n(x) \in X'_t$. Because $f$ is ergodic, $f_0$ is also ergodic. 

Define $g_0 \in [\cR\resto X'_t]$ by: $g_0(x)=g^n(x)$ where $n\ge 1$ is the smallest positive integer such that $g^n(x) \in X'_t$ and $g^n(x)$ lies between $x$ and $hf^{+\infty}(x)$ with respect to any geodesic with endpoints in $\{hf^{-\infty}(x),hf^{+\infty}(x)\}$.

We observe that $d(x,f_0(x)) \ge r$, $d(x,g_0(x))\ge r$ and $(f_0x|f_0^{-1}x)_x \le s$, $(g_0x|g_0^{-1}x)_x \le s$, $(f_0^{\pm 1} x|g_0^{\pm 1}x)_x \le s$. 
It now follows from Lemma \ref{lem:local-global2} that $f_0,g_0$ freely generate a rank 2 free group that acts essentially freely and ergodically on $X'_t$. Moreover $\cT=\{ (x,f_0^{\pm 1} x):~x\in X'_t\} \cup \{ (x,g_0^{\pm 1} x):~x\in X'_t\}$ is a treeing of $\cR \resto X'_t$ satisfying the conclusions to this lemma.

\end{proof}

\begin{lem}\label{lem:ergodicity}
Suppose $\cR $ is ergodic. Let $X' \subset X$ be a non-null set. If $\cS' \le \cR \resto X'$ is an ergodic treeable subequivalence relation then there exists an ergodic treeable subequivalence relation $\cS \le \cR $ with $\cS \resto X' =\cS'$. In particular, if $\cS'$ is non-hyperfinite then $\cS$ is also non-hyperfinite.
\end{lem}

\begin{proof}
This Lemma is a special case of \cite[Lemme II.8]{gaboriau-cost}. For the sake of convenience we provide a proof here. Let $\phi:X \to X'$ be a measurable map whose graph is contained in $\cR $ such that $\phi(x)=x$ for all $x\in X'$. Define $\cS$ by $x\cS y \Leftrightarrow \phi(x)\cS'\phi(y)$. Clearly $\cS \resto X' = \cS'$. Since $\cS'$ is ergodic and $\cR $ is ergodic this implies $\cS$ is ergodic. If $\cT'\le \cS'$ is a treeing then $\cT' \cup \{(x,\phi x), (\phi x,x):~x\in X \setminus X'\}$ is a treeing of $\cS$. Thus $\cS$ is treeable. By construction, $\cS \resto X' = \cS'$. 
\end{proof}

\begin{proof}[Proof of Theorem \ref{thm:tits}]
The Theorem follows from Lemmas \ref{lem:Tits} and \ref{lem:ergodicity}. %there exists a treeable ergodic non-hyperfinite subequivalence relation $\cS \le \cR$. %The Theorem now follows from \cite[Proposition 14]{gaboriau-lyons} and the well-known fact that an ergodic treeable equivalence relation is non-hyperfinite if and only if its cost is larger than 1 (for example, see \cite{gaboriau-cost}).
%By ** $\cS$ contains a subequivalence relation $\cT \le \cS$ such that cost$(\cT) \ge 2$, $\cT$ is ergodic, treeable and non-hyperfinite. The Theorem now follows from Hjorth's Lemma for cost-attained \cite{}.
\end{proof} 

\section{Parabolic elements}

As a consequence of the Tits Alternative, we will show that parabolic elements exist and, in fact, are generic. Using this we prove that the action of the full group $[\cR]$ on $\Prob(\cL \to X)$ is minimal. 

\begin{thm}\label{thm:parabolic}
Suppose the Main Assumption is satisfied and $\cR$ is non-hyperfinite.  Then the set of all parabolic elements of $[\cR]$, denoted $\rm{PARA}$, is a dense $G_\delta$ subset of $\rm{APER} \subset [\cR]$, the set of all aperiodic elements. In particular, there exist ergodic parabolic elements of $[\cR ]$.
\end{thm}

To begin, we first prove the existence of a single  parabolic element.

\begin{lem}\label{lem:exists}
Suppose the Main Assumption is satisfied and $\cR$ is non-hyperfinite.  Then $[\cR]$ contains a parabolic element.
\end{lem}

\begin{proof}
By Lemma \ref{lem:Tits} there exist a subset $Y \subset X$ with $\mu(Y)>0$, a subequivalence relation $\cS \le \cR \resto Y$ and a treeing $\cT$ of $\cS$ such that $\sigma_x(\cT_x) \subset \sH_x$ is quasi-isometric to a 4-regular metric tree for a.e. $x$. 

By \cite[Lemma 5.3]{conley-brooks} there exists a one-ended measurable function $f:X\to X$ whose graph is contained in $\cT$. One-ended means that for a.e. $x$, if $n\ge 1$ then $f^nx \ne x$ and the backward orbit $\cup_{n=1}^\infty f^{-n}x$ is finite. Therefore,  the connected components of the graph of $f$ are 1-ended infinite trees. Since $\sigma_x(\cT_x)$ is quasi-isometric to a 4-regular metric tree,  if $\cP \le \cS$ is the subequivalence relation generated by $f$ then the limit set $\cL_x(\cP)$ has cardinality one for a.e. $x$. In particular, $\cP$ is hyperfinite since the map that assigns $x\in X$ to the Dirac measure concentrated on $\cL_x(\cP)$ is a $\cP$-invariant field of probability measures. By Lemmas \ref{lem:loxodromic-case} and \ref{lem:mixed} the fact that $\cL_x(\cP)$ has cardinality 1 for a.e. $x$ implies $\cP$ is parabolic. Again by Lemma \ref{lem:loxodromic-case}, it follows that any aperiodic element $g\in [\cP]$ is parabolic. 
\end{proof}

\begin{proof}[Proof of Theorem \ref{thm:parabolic}]
Recall that $\Prob(\cL \to X)$ denotes the space of fields of boundary measures of the bundle $\cL \to X$. By Lemma \ref{lem:fields}, $\Prob(\cL\to X)$ is affinely homeomorphic to a compact convex metrizable subset of a Banach space with the weak* topology.

 Let $\rm{Closed}(\Prob(\cL \to X))$ be the space of all closed subsets of $\Prob(\cL \to X)$ with the Hausdorff topology. Let $\Fix:[\cR ] \to \rm{Closed}(\Prob(\cL \to X))$ be the map $\Fix(\phi)=\{\nu \in \Prob(\cL \to X):~\phi \nu = \nu\}$. This map is upper semi-continuous in the following sense:  if $\{\phi_n\}_{n\in \N} \subset [\cR ]$ converges to $\phi_\infty \in [\cR ]$ then 
$$\Fix(\phi_\infty) \supset \limsup_{n\to\infty} \Fix(\phi_n) = \bigcap_{N \in \N} \bigcup_{n\ge N} \Fix(\phi_n)$$
Indeed, this follows from the fact that the map $[\cR ] \times \Prob(\cL \to X) \to \Prob(\cL \to X)$ given by $(\phi,\nu) \mapsto \phi\nu$ is continuous (see Theorem \ref{thm:jcontinuous}). 

We say that a subset $\cK \subset \rm{Closed}(\Prob(\cL \to X))$ is {\bf upwards closed} if $K \in \cK$ and $K \subset L \in \rm{Closed}(\Prob(\cL \to X))$ implies $L \in \cK$. It follows from upper semi-continuity that: if $\cK \subset \rm{Closed}(\Prob(\cL \to X))$ is closed (in the Hausdorff topology) and upwards closed then $\Fix^{-1}(\cK) \subset [\cR ]$ is closed.

Fix a compatible metric $d_*$ on $\Prob(\cL \to X)$. Given a natural number $n$, let $J_n$ be the set of all $\nu \in \Prob(\cL \to X)$ such that there exists $\nu_1,\nu_2$ with $\nu=(\nu_1+\nu_2)/2$ and $d_*(\nu_1,\nu_2)\ge 1/n$. Let $\cK_n$ be the set of all $L  \in \rm{Closed}(\Prob(\cL \to X))$ such that $J_n \cap L \ne \emptyset$. Observe that $\cK_n$ is closed in $\rm{Closed}(\Prob(\cL \to X))$ and $\cK_n$ is upwards closed. So $\Fix^{-1}(\cK_n) \subset [\cR ]$ is closed. By definition, 
$$\rm{PARA} = \cap_{n=1}^\infty \APER\setminus \Fix^{-1}(\cK_n)$$
is a $G_\delta$ subset of $\rm{APER}$. By Lemma \ref{lem:exists}, \rm{PARA} is nonempty. By \cite[Theorem 3.4]{Kechris-global-aspects}, every conjugacy class in \rm{APER} is dense. Since parabolicity is a conjugacy-invariant, $\rm{PARA}$ is a dense $G_\delta$ subset of $\rm{APER}$.

Since the set $\rm{ERG} \subset \rm{APER}$ of ergodic elements is a dense $G_\delta$ (by \cite[Theorem 3.6]{Kechris-global-aspects}) it follows that $\rm{PARA} \cap \rm{ERG}$ is a dense $G_\delta$ subset of $\rm{APER}$. In particular, there exists an ergodic parabolic element.
\end{proof}

We can now prove:

\noindent{\bf Theorem \ref{thm:classification}}
{\em If the Main Assumption is satisfied and $\cR$ is non-hyperfinite then both $\PARA$ and $\LOXO$ are nonempty. Moreover, $\PARA$ is a dense $G_\delta$ subset of $\APER$. 
On the other hand, if $\cR$ is hyperfinite then either $\APER=\PARA$ or $\APER=\LOXO$. }
\begin{proof}
This first statement follows immediately from Theorems \ref{thm:exists} and \ref{thm:parabolic}. To prove the second statement, assume $\cR$ is hyperfinite. Since the Main Assumption implies $\mu$ is $\cR$-ergodic, we know from Theorem \ref{thm:K} that $\cR$ is either parabolic or loxodromic. If it is parabolic then every aperiodic subequivalence relation is also parabolic by Theorem \ref{thm:non-nested}. If it is loxodromic then there exists an $\cR$-invariant field of boundary measures $\nu \in \Fix(\cR)$ such that the support of $\nu_x$ contains two elements for a.e. $x$. If $\cS\le \cR$ is any subequivalence relation then $\nu$ is also $\cS$-invariant and therefore, if $\cS$ is aperiodic then it must be loxodromic. This proves the last statement.
\end{proof}

\begin{cor}\label{cor:minimal}
Suppose the Main Assumption is satisfied. In addition, assume $\cR$ is non-hyperfinite. Then the action of $[\cR]$ on $\Prob(\cL \to X)$ is minimal.
\end{cor}

\begin{proof}
Let $\nu \in \Prob(\cL \to X)$. It suffices to show that the orbit $[\cR]\nu$ is dense in $\Prob(\cL \to X)$. By Lemma \ref{lem:exists} there exists a parabolic element $f \in [\cR]$. Since $f$ is parabolic it admits a unique fixed point in $\Prob(\cL \to X)$. Moreover, there is a Borel section $\xi:X \to \cL$ such that the fixed point of $f$ is $\delta_{\xi}$ which denotes the field of Dirac measures $x \mapsto \delta_{\xi(x)}$. Note $\delta_\xi$ is an extreme point of $\Prob(\cL\to X)$.   

Observe that any limit point of the sequence $n \mapsto \frac{1}{n+1} \sum_{i=0}^n f^i\nu$ must be a fixed point of $f$ and therefore must equal $\delta_\xi$. It follows that $\delta_\xi$ is in the closure of the orbit $[\cR]\nu$. 

Let $\cO^\xi_x = \{ \alpha(x,y)\xi(y):~y \in [x]_\cR\}$ be the orbit of $\xi$ and $\cO^\xi=\cup_x \cO^\xi_x \subset \cL$. Because the action of $\cR$ on $\cL$ is minimal (in the sense of Theorem \ref{thm:minimal}) it follows that $\cO^\xi_x$ is dense in $\cL_x$ for a.e. $x$.

Fix an arbitrary Borel section $\sigma:X \to \sH$. As in \S \ref{sec:hyperbolic-bundle} define a fibrewise metric $\rho_\sigma:\bsH*\bsH \to [0,\infty)$ by 
\begin{displaymath}
\rho_\sigma(\xi,\eta) = \left\{\begin{array}{cc}
\inf  \sum_{i=1}^{n-1}  \exp(-\epsilon (\xi_i|\xi_{i+1})_{\sigma(x)}) & \xi\ne \eta \\
0 & \xi =\eta
\end{array}\right.
\end{displaymath}
where $\epsilon>0$ is such that $\epsilon \delta \le 1/5$ and  the infimum is over all sequences $\xi_1,\ldots, \xi_n \in \bsH_x$ with $\xi_1=\xi, \xi_n = \eta$.

Let $\eta:X \to \cL$ be an arbitrary Borel section and $r>0$. We will show that $\delta_\eta$ is in the orbit closure of $\delta_\xi$.

Recall from \S \ref{sec:sketch} that $\lb\cR \rb$ denotes the partially ordered set of all partial transformations. Let $\cF \subset \lb \cR \rb$ be the set of all $\phi$ such that
$$\rho_\sigma( \alpha(x,\phi^{-1}x)\xi(\phi^{-1}x), \eta(x) ) < r\quad \forall x\in \rng(\phi).$$
 By Zorn's Lemma there exists a maximal element $\phi$ of $\cF$. 
 
To obtain a contradiction, suppose $\phi \notin [\cR]$. For $x \notin \rng(\phi)$, let 
$$\cO^{\xi,\phi}_x = \{ \alpha(x,y)\xi(y):~y  \in X \setminus \dom(\phi) \}.$$
By Theorem \ref{thm:minimal}, $\cO^{\xi,\phi}_x$ is dense in $\cL_x$ for a.e. $x$. So there exists a Borel section $\beta:X \setminus \rng(\phi) \to \cL$ such that $\beta(x) \in \cO^{\xi,\phi}_x$ for every $x$ and $\rho_\sigma(\beta(x), \eta(x))<r$. Since $\beta(x) \in \cO^{\xi,\phi}_x$ there exists a Borel map $\psi:X \setminus \rng(\phi) \to X \setminus \dom(\phi)$ with graph contained in $\cR$ such that $\alpha(x,\psi(x))\xi(\psi(x)) = \beta(x)$ for a.e. $x$. 

By Lemma \ref{lem:1-1}, there exists a subset $Y \subset X \setminus \rng(\phi)$ with positive measure such that $\psi$ restricted to $Y$ is 1-1. Thus we map define $\kappa \in \lb\cR\rb$ by $\kappa(x)=\phi(x)$ for $x\in \dom(\phi)$ and $\kappa(x) =\psi^{-1}_Y(x)$ for $x \in \psi(Y)$ where $\psi_Y$ denotes the restriction of $\psi$ to $Y$. 

Observe that $\kappa \in \cF$ by construction. Since $\kappa>\phi$, this contradicts maximality of $\phi$. Thus we must have $\phi \in [\cR]$.

Note that $\phi\delta_\xi \to \delta_\eta$ as $r\to 0$ ($\phi$ depends implicitly on $r$). Thus $\delta_{\eta} \in  \overline{[\cR]\nu}$. Since $\eta$ is arbitrary, $\overline{[\cR]\nu}$ contains every extreme point of $\Prob(\cL \to X)$ (by Lemma \ref{lem:extreme}). Because the extreme points of $\Prob(\cL \to X)$ are dense in $\Prob(\cL \to X)$ (Lemma \ref{lem:extreme2}), this implies the corollary.

%$K^\sigma_x=\cL_x$ for a.e. $x$. 
%Also let $K^\sigma_x$ denote the closure of $\cO^\sigma_x$ in $\cL_x$ and $K^\sigma=\cup_x K^\sigma_x$. 

\end{proof}

\appendix %%%%%%%%%%%%%%%%%%%%%%%%%%%%%%%%%%%%%%%%%%%%%%%%%%%%%%%%%%%%%%%%%%%%%%%%%%%%%%%%%%%%%%

\section{Hyperbolic geometry}\label{sec:hyperbolic}

Let $(\cH,d)$ be a metric space. For $x,y,z \in \cH$ define the {\bf Gromov product} of $x$ and $y$ with respect to $z$ by
$$(x|y)_z = (1/2)( d(x,z) + d(y,z) - d(x,y)).$$
A metric space $(\cH,d)$ is {\bf $\delta$-hyperbolic} if for every $x,y,z,w \in \cH$,
\begin{eqnarray}\label{eqn:gromovproduct}
(x|y)_z \ge \min\{ (x|w)_z, (w|y)_z\} - \delta.
\end{eqnarray}

\subsection{The Gromov boundary}

We say that $\{x_i\}_{i=1}^\infty \subset \cH$ is a {\bf Gromov sequence} if for some (any) $x \in \cH$
$$\lim_{i,j \to \infty} (x_i|x_j)_x = +\infty.$$
Two Gromov sequence $\{x_i\}_{i=1}^\infty, \{y_i\}_{i=1}^\infty \subset \cH$ are {\bf equivalent} if 
$$\lim_{i,j \to \infty} (x_i|y_j)_x = +\infty.$$
Let $\partial \cH$ denote the set of all equivalence classes of Gromov sequences. The Gromov product extends to $\bcH:=\cH\cup \partial \cH$ by:
$$(\xi|\eta)_z = \inf \liminf_{i\to\infty} (1/2)( d(x_i,z) + d(y_i,z) - d(x_i,y_i))$$
where $\xi, \eta \in \partial \cH$,  $z\in \cH$ and the infimum is over all sequences $\{x_i\}_{i=1}^\infty \in \xi, \{y_i\}_{i=1}^\infty \in \eta$. We also define
$$(\xi|y)_z  = (y|\xi)_z = \inf \liminf_{i\to\infty} (1/2)( d(x_i,z) + d(y,z) - d(x_i,y))$$
where $\xi \in \partial \cH, y,z \in \cH$ and the infimum is over all sequences $\{x_i\}_{i=1}^\infty \in \xi$. An elementary computation shows that at the cost of increasing $\delta$, the equation (\ref{eqn:gromovproduct}) holds for all $x,y,w \in \bcH$ and $z\in \cH$. We will therefore assume that $\delta>0$ has been chosen so that equation (\ref{eqn:gromovproduct}) holds for all $x,y,w \in \bcH$ and $z\in \cH$. For $\epsilon>0$ and $\xi,\eta \in \bcH$, %let
%$$\rho_\epsilon(\xi, \eta) = e^{-\epsilon (\xi|\eta)_x}$$
%and 
define
\begin{displaymath}
\rho_\epsilon(\xi,\eta) = \left\{\begin{array}{cc}
\inf  \sum_{i=1}^{n-1}  \exp(-\epsilon (\xi_i|\xi_{i+1})_{x}) & \xi\ne \eta \\
0 & \xi =\eta
\end{array}\right.
\end{displaymath}
where the infimum is over all sequences $\xi_1,\ldots, \xi_n \in \bcH$ with $\xi_1=\xi, \xi_n = \eta$.

%Recall that a {\bf metametric} on a set $Z$ is a function $d:Z\times Z \to [0,\infty)$ satisfying all the usual properties of a metric with the except that it is possible that for some $z\in Z$, $d(z,z)>0$ (see \cite[\S 4]{vaisala-hyperbolic} for details).

\begin{lem}\label{lem:meta}
If $\epsilon\delta \le 1/5$ then $\rho_\epsilon$ is a metric on $\bcH$.
\end{lem}
\begin{proof}
This is implied by \cite[Proposition 5.16]{vaisala-hyperbolic}.
\end{proof}

%Define $d_\epsilon:\b\cH \to \b\cH$ by 
%\begin{displaymath}
%d_\epsilon(x,y) =\left\{ \begin{array}{cc}
% d'_\epsilon(x,y) & x \ne y \\
% 0 & x = y \end{array}\right.\end{displaymath}
% This is a metric on $\b\cH$.
 
 \begin{lem}\label{lem:compactification}
 If $d$ is a complete metric on $\cH$ then $\rho_\epsilon$ is a complete metric on $\bcH$ ($\epsilon\delta\le 1/5$). If $d$ is also proper and geodesic then $\bcH$ is compact.
 \end{lem}
 \begin{proof}
The first statement is \cite[Proposition 5.31]{vaisala-hyperbolic}. The second statement is \cite[Chapter III.H, Proposition 3.7]{bridson-haefliger-book}.
\end{proof}
 
 \begin{defn}[Geodesics]
 A path $\gamma:I \to \cH$ is a {\bf geodesic} if $d(\gamma(t),\gamma(s))=|t-s|$ for any $t,s \in I$ where $I \subset \R$ is an interval (possibly an infinite interval). If $I=[a,b]$ then the {\bf endpoints of $\gamma$} are $\gamma(a),\gamma(b)$. If $I=(-\infty,+\infty)$ then the  {\bf endpoints of $\gamma$} are $\lim_{t\to\pm \infty} \gamma(t)$. If $p,q \in \overline{\cH}$ then we may write $[p,q]$ to denote a geodesic with endpoints $p,q$. This geodesic might not be unique. By abuse of notation, we may identify this geodesic with its image as a subset of $\cH$. We say $(\cH,d)$ is a {\bf geodesic metric space} if for every $a,b \in \cH$ there is a geodesic with endpoints $\{a,b\}$.
 %  If $I=[a,b]$ then we say $\gamma$ is a {\bf geodesic segment} from $\gamma(a)$ to $\gamma(b)$. If $I=(-\infty,\infty)$ and $\xi^{\pm} = \lim_{t\to \pm \infty} \gamma(t) \in \partial \cH$ then we say $\gamma$ is a {\bf geodesic with endpoints $\xi^-, \xi^+$}.
  \end{defn}

 \begin{lem}[Thin triangles]\label{lem:thin}
Suppose $(\cH,d)$ is a complete geodesic hyperbolic metric space and $\gamma_1,\gamma_2,\gamma_3$ are geodesics forming a triangle (so if $\gamma_i$ has endpoints $p_i,q_i$ then $q_i = p_{i+1}$ mod 3). Then there is a $\delta'>0$ such that the $\delta'$-neighborhood of $\gamma_1 \cup \gamma_2$ contains $\gamma_3$. In particular, if $\gamma_1,\gamma_2$ are geodesics with the same endpoints then $\gamma_2$ is contained in the $2\delta'$-neighborhood of $\gamma_1$. 
 \end{lem}
 
  \begin{proof}
 This is \cite[Chapter III.H, Proposition 1.22]{bridson-haefliger-book}.
\end{proof}
%[[This lemma is used to show loxodromic relations have exactly 2 limit points]]

\subsection{Quasi-isometries}\label{sec:qi}

%Let $(\cH,d)$ be a $\delta$-hyperbolic metric space. We also assume $(\cH,d)$ is separable complete and geodesic. Given $x,y \in \cH$, we denote a geodesic from $x$ to $y$ by $[x,y]$. Observe that this geodesic need not be unique.

\begin{defn}
Let $(X,d_X)$, $(Y,d_Y)$ be metric spaces. For $\lambda \ge 1$ and $c\ge 0$, a map $\phi:X \to Y$ is a {\bf $(\lambda,c)$-quasi-isometric embedding} if for all $x,y \in X$,
$$\lambda^{-1} d_X(x,y) - c \le d_Y(\phi(x),\phi(y)) \le \lambda d_X(x,y) + c.$$
%If, in addition, for every $y\in Y$ there is a $x \in X$ such that $d_Y(\phi(x),y) \le c$ then $\phi$ is a {\bf $(\lambda,c)$-quasi-isometry}.
\end{defn}

\begin{defn}
Let $(\cH,d)$ be a Gromov hyperbolic space. A $(\lambda,c)$-quasi-isometric embedding $q$ of an interval $I \subset \R$ into $\cH$  is called a {\bf $(\lambda,c)$-quasi-geodesic}. A {\bf quasi-geodesic} is a $(\lambda,c)$-quasi-geodesic for some $\lambda,c$. If $I$ is a finite interval and its endpoints are mapped to $x,y \in \cH$ respectively, then we say it is a quasi-geodesic {\bf from $x$ to $y$}. If $I=(-\infty, \infty)$ and $\lim_{t \to -\infty} q(t) = \xi_-, 
\lim_{t \to +\infty} q(t) = \xi_+$, then we say $q$ is a quasi-geodesic from $\xi_-$ to $\xi_+$. A similar definition holds for half-infinite intervals.

%The space $(\cH,d_X)$ is {\bf $(\lambda,c)$-quasi-geodesic} if for every pair of points $x,y \in X\cup \partial X$ there exists a $(\lambda,c)$-quasi-geodesic from $x$ to $y$.  By abuse of notation, we sometimes identify a quasi-geodesic with its image. For example, it is convenient to denote by $[x,y]$ a $(\lambda,c)$-quasi-geodesic from $x$ to $y$ ($x,y \in X$).
\end{defn}

\subsection{Busemann functions}\label{sec:horofunctions}

Let $(\cH,d)$ be a $\delta$-hyperbolic metric space. Given $\xi \in \partial \cH$, the associated {\bf Busemann function} $\beta:\cH \times \cH  \to \R$ is defined by 
$$\beta(x,y) = \sup \limsup_{i\to\infty} d(x,z_i) - d(y,z_i) =\sup \limsup_{i\to\infty}  2(y|z_i)_x - d(x,y) $$
where the supremum is over all sequences $\{z_i\}_{i\in \N}$ of points $z_i \in \cH$ such that $\lim_{i\to\infty} z_i = \xi$.

\begin{lem}\label{lem:buse-main}
Suppose $\{p_i\} \subset \cH$ is a sequence converging to $\xi$ in $\overline{\cH}$. Then for any $x,y \in \cH$,
$$|\beta(x,y) -( \liminf_{i\to\infty} d(x,p_i) - d(y,p_i) )| \le 2\delta$$
and
$$|\beta(x,y) -( \limsup_{i\to\infty} d(x,p_i) - d(y,p_i) )| \le 2\delta.$$
\end{lem}
\begin{proof}
By definition of $\beta$, it suffices to show 
$$\limsup_{i,j\to\infty} |2(y|z_i)_x - 2(y|p_j)_x| \le 2\delta.$$
By (\ref{eqn:gromovproduct}), $(y|z_i)_x \ge \min\{ (y|p_j)_x, (p_j|z_i)_x\} - \delta$. Since $(p_j|z_i)_x \to \infty$ as $i,j\to\infty$, this implies that for all $i,j$ sufficiently large, $(y|z_i)_x \ge (y|p_j)_x  - \delta$. By symmetry, this implies the claim above.
\end{proof}

\begin{lem}\label{lem:geodesic}
Suppose $x,y,\xi$ all lie on a geodesic $\gamma$. Then $||\beta(x,y)| - d(x,y)| \le 2\delta$.
\end{lem}

\begin{proof}
This follows from the previous lemma by choosing the points $p_i$ to lie on the geodesic $\gamma$.
\end{proof}

\begin{lem}\label{lem:buse}
For any $\xi \in \partial \cH$ and $x,y,z \in \cH$ the cocycle equation holds up to $4\delta$:
$$|\beta(x,y)+\beta(y,z)-\beta(x,z)|\le 4\delta.$$
\end{lem}

\begin{proof}
Choose a sequence $\{p_i\} \subset \cH$ converging to $\xi$ such that
$$\beta(x,y) = \lim_{i\to\infty} d(x,p_i) - d(y,p_i).$$
Lemma \ref{lem:buse-main} implies that, up to a maximum error of $4\delta$, $\beta(x,y)+\beta(y,z)-\beta(x,z)$ equals 
$$\limsup_i d(x,p_i) - d(y,p_i) + d(y,p_i) - d(z,p_i) - (d(x,p_i)-d(z,p_i)) = 0.$$
\end{proof}

\section{Measured equivalence relations}\label{sec:mer}

Let $(X,\mu)$ be a standard Borel probability space and $\cR \subset X \times X$ be a Borel equivalence relation. We say that $\cR$ is 
\begin{itemize}
\item {\bf discrete} if every $\cR$-class is countable,
\item {\bf probability-measure-preserving} (pmp) if for every Borel isomorphism $\phi:X \to X$ with $x\cR\phi(x)$ for all $x$, we have $\phi_*\mu=\mu$,
\item {\bf ergodic} if for every Borel set $A\subset X$, $\mu([A]_\cR) \in \{0,1\}$ where $[A]_\cR$ is the union of all $\cR$-classes that nontrivially intersect $A$.
\end{itemize}
Two discrete pmp equivalence relations $(X_i,\mu_i,\cR_i)$ (for $i=1,2$) are {\bf isomorphic} if there exist conull sets $X'_i \subset X_i$ and a measure-space isomorphism $\phi:(X'_1,\mu_1)\to (X'_2,\mu_2)$ such that $(x,y) \in \cR_1 \Leftrightarrow (\phi(x),\phi(y)) \in \cR_2$. More precisely, we only require that $\phi$ is defined on a set of full measure. 

%We use \cite{Kechris-global-aspects} as a general reference.

Define measures $\mu_L, \mu_R$ on $\cR$ by
$$\mu_L(F) = \int | F \cap \pi_L^{-1}(x)| ~d\mu(x), \quad \mu_R(F) = \int | F \cap \pi_R^{-1}(x)| ~d\mu(x)$$
where $\pi_L:\cR \to X, \pi_R: \cR \to X$ are the left and right projection maps. It is a standard fact (and a good exercise) to show that $\cR$ is pmp if and only if $\mu_L=\mu_R$.  In this case, we let $\hmu$ denote either $\mu_L$ or $\mu_R$. In the sequel, the phrase ``for a.e. $(x,y) \in \cR$'' is taken to mean with respect to $\hmu$. Another formulation of the probability measure-preserving property is the following:

\begin{lem}[Mass Transport Principle]\label{lem:mtp}
If $\cR$ is a discrete pmp Borel equivalence relation on $(X,\mu)$ and suppose either $F\in L^1(\cR,\hmu)$ or $F\ge 0$ is measurable. Then
$$\int \sum_{x\in [y]_\cR} F(x,y)~d\mu(y) = \int \sum_{y\in [x]_\cR} F(x,y)~d\mu(x) = \int F~d\hmu.$$
\end{lem}
\begin{proof}
Apply Fubini's Theorem.
\end{proof}

Measured equivalence relations arise from actions of groups: if $G$ is a countable group and $G \cc (X,\mu)$ a measure-class-preserving action then $\cR=\{(x,gx):~x\in X, g\in G\}$ is a discrete Borel equivalence relation called the {\bf orbit equivalence relation}.  This action of $G$ is measure-preserving if and only if $\cR$ is pmp; the action of $G$ is ergodic if and only if $\mu$ is $\cR$-ergodic.  Feldman-Moore \cite{feldman-moore-1} proved that every discrete pmp equivalence relation $\cR$ is the orbit-equivalence relation of the action of some countable group. One of the useful consequences of this is:

\begin{lem}\label{lem:1-1}
Let $\cR$ be a discrete pmp equivalence relation, $Y \subset X$ Borel and $\phi: Y \to X$ a Borel map with $x \cR \phi(x)$ for all $x \in Y$. If $\mu(Y)>0$ then there exists $Z \subset Y$ with $\mu(Z)>0$ such that $\phi$ restricted to $Z$ is 1-1.
\end{lem}

\begin{proof}
As mentioned above. we may assume $\cR$ is generated by a countable group $G$. For $g\in G$, let $Y_g=\{x\in Y:~\phi x = gx\}$. Since $Y=\cup_g Y_g$, there exists $g\in G$ such that $\mu(Y_g)>0$. Since $\phi$ restricted to $Y_g$ is 1-1, we are done.
\end{proof}

\subsection{The full group}\label{sec:full-group}

The full group of $\cR$, denoted $[\cR]$, is the group of all (equivalence classes of) invertible Borel transformations $f$ such that $\textrm{graph}(f)=\{(x,fx):~x\in X\}\subset \cR$. Two transformations are equivalent if they agree on a conull subset. By \cite[Proposition 3.2]{Kechris-global-aspects}, $[\cR]$ with the uniform metric, defined by 
$$d_u(\phi,\psi) = \mu(\{x\in X:~\phi(x)\ne \psi(x)\}),$$
is a Polish group.

An element $\phi\in [\cR]$ is 
\begin{itemize}
\item {\bf aperiodic} if for a.e. $x\in X$, $\{\phi^n(x)\}_{n\in\Z}$ is infinite
\item {\bf ergodic} if every measurable $\phi$-invariant subset $A \subset X$ satisfies $\mu(A)\in\{0,1\}$.
\end{itemize}
Since $\phi$ is aperiodic if and only if $d_u(\phi^n,\rm{id})>1-1/m$ for every $n,m\ge 1$, the subset $\APER(\cR)$ of aperiodic elements of $[\cR]$ is a $G_\delta$ subset. So it is Polish. By \cite[Theorem 3.6]{Kechris-global-aspects} the subset $\ERG(\cR) \subset \APER(\cR)$ of ergodic elements is a dense $G_\delta$ subset of $\APER(\cR)$.

\subsection{Subequivalence relations}

A Borel subset $\cS \subset X \times X$ is a {\bf subequivalence relation} of $\cR$ if it is an equivalence relation and $\cS \subset \cR$. This is denoted by $\cS \le \cR$. For example if $\cR$ is the orbit-equivalence relation of the action $G \cc (X,\mu)$ of a countable group $G$ and $H<G$ is a subgroup then orbit-equivalence relation of $H$ is a subequivalence relation of $\cR$.

%A subequivalence relation $F$ is {\bf normal} if ... (there's a lot to put here?)

\subsection{Compressions and amplifications}\label{sec:compression}

Let $(X,\mu,\cR)$ be a pmp equivalence relation and $Y \subset X$ a measurable subset with $\mu(Y)>0$. Then $\cR\resto Y:=\cR \cap (Y\times Y)$ is an equivalence relation on $Y$ called the {\bf compression (or restriction) of $\cR$ to $Y$}. It preserves the restricted measure $\mu \resto Y$. Moreover, if $(\cR,\mu)$ is ergodic/treeable/hyperfinite then $(\cR \resto Y, \mu \resto Y)$ is also ergodic/treeable/hyperfinite. If $(\cR,\mu)$ satisfies the Main Assumption (Definition \ref{defn:main}) then $(\cR\resto Y,\mu\resto Y)$ also satisfies the Main Assumption by restricting the bundle $\sH$ to $\pi^{-1}(Y)$. 

Let $F$ denote a finite or countably infinite set. We define an equivalence relation $\tilde{\cR}$ on $X\times F$ by $(x,n)\tilde{\cR}(y,m)\Leftrightarrow x\cR y$. Fixing an element $f_0 \in F$, we may identify $X$ with $X \times \{f_0\}$ and $\cR$ with the compression $\tilde{\cR}\resto X\times \{f_0\}$. We call $\tilde{\cR}$ the {\bf amplification} of $\cR$ to $X\times F$. It is measure-preserving if we give $X\times F$ the measure $\mu \times c_F$ where $c_F$ denotes counting measure on $F$. Again it is elementary to check that if $(\cR,\mu)$ is ergodic/treeable/hyperfinite then $(\tilde{\cR},\mu\times c_F)$ is ergodic/treeable/hyperfinite. If $(\cR,\mu)$ satisfies the Main Assumption then $(\tilde{\cR},\mu\times c_F)$ does too: we simply let $\sH\times F \to X\times F$ denote the obvious extension of the bundle $\sH \to X$. We leave the details to the reader.

\subsection{Finite equivalence relations}

\begin{defn}
Let $(X,\mu,\cR)$ be a pmp discrete equivalence relation. A {\bf selector} is a measurable map $f:X \to X$ such that $f(x)\cR x$ for a.e. $x$ and $f(x)=f(y)$ for a.e. $(x,y)\in \cR$. 
A {\bf set function} for $\cR$ is a map $F$ on $X$ such that $F(x)$ is a subset of $[x]_\cR$. We require that $F$ is measurable which means that the subset $\{(x,y)\in \cR:~y \in F(x)\} \subset \cR$ is measurable. A set function is {\bf finite} if $F(x)$ is finite for a.e. $x$. It is {\bf invariant} if $F(x)=F(y)$ for a.e. $(x,y)\in \cR$. 
\end{defn}

\begin{lem}\label{lem:smooth}
Let $(X,\mu,\cR)$ be a pmp discrete equivalence relation. The following are equivalent:
\begin{itemize}
\item There is a finite invariant set function for $\cR$;
\item There is a selector for $\cR$;
\item for a.e. $x\in \cR$, $[x]_\cR$ is finite.
\end{itemize}
\end{lem}

%[[** this lemma is used in the proof that $\cL=\cL^\eta$ **]]

\begin{proof}
%The equivalence of items (1) and (2) is contained in \cite[Chapter II, Proposition 6.4]{MR2095154} and the paragraph above it.  

The lemma is trivial in the atomic case so without loss of generality we may assume $\mu$ is purely non-atomic. So there exists a Borel isomorphism $\phi:X \to [0,1]$. Suppose $F$ is a  finite invariant set function for $\cR$. Define $f:X \to X$ by 
$$f(x)=y \Leftrightarrow y \in F(x) \textrm{ and } \phi(y) = \min\{\phi(z):~z \in F(x)\}.$$
Then $f$ is a selector for $\cR$. Observe that $\phi\circ f:X \to [0,1]$ is an $\cR$-invariant function satisfying $x\cR y \Leftrightarrow \phi(f(x))=\phi(f(y))$ for a.e. $(x,y)\in \cR $. In particular, if $\cR $ is ergodic then $\phi \circ f$ must be constant and therefore there must be only one $\cR $-class (up to measure zero). Since $\mu$ is an invariant probability measure, this implies that $[x]_\cR $ is finite for a.e. $x\in \cR $. The general case follows from the ergodic decomposition theorem.

If $[x]_\cR $ is finite for a.e. $x$ then the set function $F(x)=[x]_\cR $ is a finite invariant set function.
\end{proof}

\section{Fields of probability measures}\label{sec:fields}

Throughout this section, we let $\pi:B \to X$ and $d:B*B \to \R$ denote a separable metric bundle (as in Definition \ref{defn:main0}). We also fix a Borel probability measure $\mu$ on $X$. A {\bf field of probability measures} is an assignment $x\mapsto \nu_x$ of probability measures on $B_x$ (for $x\in X$). We will topologize the space of Borel fields of probability measures and prove that it is Hausdorff, separable and even compact under appropriate hypotheses. We will then investigate its convex structure.

To begin, we need a little notation.  We say two functions $F_1, F_2$ on $B$ are {\bf equivalent} if for a.e. $x\in X$ $F_1 \resto B_x=F_2 \resto B_x$ where $\resto$ means ``restricted to''. For each $x\in X$, let $C_0(B_x)$ denote the Banach space of continuous functions on $B_x$ that vanish at infinity with the sup norm. Suppose $F:B \to \C$ is a Borel function such that $F_x \in C_0(B_x)$ for a.e. $x\in X$. Then we define its norm by
$$\|F\|:= \| x \mapsto \|F_x\| \|_{L^\infty(X,\mu)}.$$
Let $C_0(\pi)$ denote the set of all equivalence classes of Borel functions $F:B \to \C$ such that for a.e. $x\in X$, $F_x \in C_0(B_x)$ and $\|F\|<\infty$. 

\begin{lem}\label{lem:sep}
There exists a countable set $\Omega \subset C_0(\pi)$ such that for a.e. $x\in X$, $\{F_x\}_{F \in \Omega}$ is dense in $C_0(B_x)$. 
\end{lem}

\begin{proof}
Because the bundle is separable there exists a sequence $\{\sigma_i\}_{i\in\N}$ of Borel sections and $Y\subset X$ a conull set such that for every $x\in Y$, $\{\sigma_i(x)\}_{i\in\N}$ is dense in $B_x$. For $n,i \in \N$ define $F_{n,i} \in C_0(\pi)$ as follows. For $p\in B$ with $\pi(p)=x$ let 
$$F_{n,i}(p) = \max(1/n-d(\sigma_i(x), p),0).$$
Let $\F$ be a countable dense subfield of $\C$. Let $\Omega$ be the $\F$-subalgebra of $C_0(\pi)$ generated by $\{F_{n,i}\}_{n,i\in \N}$. By the Stone-Weierstrauss Theorem, for every $x\in Y$, the $\C$-linear span of $\{F_x\}_{F \in \Omega}$ is dense in $C_0(B_x)$. Since $\F$ is dense in $\C$, we obtain that in fact $\{F_x\}_{F \in \Omega}$ is dense in $C_0(B_x)$.
\end{proof}

Let $\cP(B_x)$ denote the set of all Borel probability measures on $B_x$ and let $\cP(B)=\sqcup_x \cP(B_x)$ denote the disjoint union. This is a bundle of spaces of probability measures over $X$. A {\bf Borel field of regular probability measures} is a map $\nu:X \to \cP(B)$ satisfying:
\begin{itemize}
\item for every $x\in X$, $\nu_x$ is a regular Borel probability measure on the fiber $B_x$ 
\item for every $F \in \cF(B)$, the map $x \mapsto \int F(p)~d\nu_x(p)$ is measurable.
\end{itemize}
Two fields $\nu,\eta$ are {\bf equivalent} if $\nu_x=\eta_x$ for a.e. $x$. By abusing notation, we will not distinguish between equivalent fields. 

Let $\Prob(\pi)$ denote the set of all (equivalence classes of) Borel fields of regular probability measures on $B$. Given $F \in C_0(\pi)$ and an open set $O \subset \C$, let $\Omega(F,O)$ be the set of all $\nu \in \Prob(\pi)$ such that $\int \nu_x(F_x)~d\mu(x) \in O$. We give $\Prob(\pi)$ the topology generated by sets of the form $\Omega(F,O)$. 

Let $C_0(\pi)^*$ denote the Banach dual of $C_0(\pi)$. We always consider  $C_0(\pi)^*$ with the weak* topology. This means that $\Lambda_i \to \Lambda$ in $C_0(\pi)^*$ if and only if $\Lambda_i(F) \to \Lambda(F)$ for every $F\in C_0(\pi)$. Let $\Psi:\Prob(\pi) \to C_0(\pi)^*$ denote the map $\Psi(\nu)(F) = \int \nu_x(F_x)~d\mu(x)$. 

\begin{lem}\label{lem:fields}
The map $\Psi$ is an affine homeomorphism onto its image. Thus $\Prob(\pi)$ is metrizable and convex.  If $B_x$ is compact (for a.e. $x$) then $\Prob(\pi)$ is compact. 
 \end{lem}

\begin{proof}
It is easy to check that $\Psi$ is affine (this means that $\Psi(t\nu + (1-t)\eta)=t\Psi(\nu) + (1-t)\Psi(\eta)$ for any $\nu,\eta$ and $t\in [0,1]$) and continuous. To see that it is injective, let $\nu,\eta \in \Prob(\pi)$ and suppose $\Psi(\nu)=\Psi(\eta)$. So for every $F \in C_0(\pi)$ we have 
$$\int \nu_x(F_x)~d\mu(x) = \int \eta_x(F_x)~d\mu(x).$$
We claim that $\nu_x(F_x)=\eta_x(F_x)$ for a.e. $x$. Indeed, for any Borel set $E \subset X$, $\chi_{\pi^{-1}E} F \in C_0(\pi)$ where $\chi_{\pi^{-1}E}$ denotes the characteristic function of $\pi^{-1}E \subset B$. So
$$\int_E \nu_x(F_x)~d\mu(x) = \int \nu_x((\chi_{\pi^{-1}E} F)_x)~d\mu(x) = \int \eta_x((\chi_{\pi^{-1}E} F)_x)~d\mu(x)=\int_E \eta_x(F_x)~d\mu(x).$$
Since $E \subset X$ is arbitrary, for every $F \in C_0(\pi)$ we have $\nu_x(F_x)=\eta_x(F_x)$ for a.e. $x$. 

Let $\Omega \subset C_0(\pi)$ be as in Lemma \ref{lem:sep}. Because $\Omega$ is countable, there is a conull set $Z\subset X$ such that $\nu_x(F_x)=\eta_x(F_x)$ for every $F \in\Omega$ and $x \in Z$.  Because $\nu_x$ and $\eta_x$ are regular, they are determined by their values on $C_0(B_x)$. Therefore $\eta_x=\nu_x$ for all $x\in Z$. Since $Z$ is conull, this proves $\Psi$ is injective. 

It is easy to check that the inverse $\Psi^{-1}$ is also continuous and therefore $\Psi$ is an affine homeomorphism onto its image.
We define a metric on $\Prob(\pi)$ as follows. Let $\Omega=\{F_i\}_{i\in \N}$ and define
$$d_{\Prob(\pi)}(\nu,\eta) := \sum_{i\in \N} \frac{|\Psi(\nu)(F_i) - \Psi(\eta)(F_i)|}{\|F_i\|2^{i}}.$$

The image of $\Psi$ lies inside the unit ball of $C_0(\pi)^*$ which, by the Banach-Alaoglu Theorem, is weak* compact. Now suppose each fiber $B_x$ is compact. It suffices to show the image of $\Psi$ is weak* closed. So suppose $\eta^j \in \Prob(\pi)$ and $\Psi(\eta^j) \to \Lambda \in C_0(\pi)^*$ as $j\to\infty$.  Let $F \in C_0(\pi)$ and $E \subset X$ be Borel. Since 
$$\lim_i \int \eta^j_x((\chi_{\pi^{-1}E}F)_x)~d\mu(x) = \lim_i \int_E \eta^j_x(F_x)~d\mu(x) = \Lambda(\chi_{\pi^{-1}E}F)$$
it follows that the function $E \mapsto \Lambda(\chi_{\pi^{-1}E}F)$ is a complex valued measure on $X$ that is absolutely continuous $\mu$. By the Radon-Nikodym Theorem there exists a function $\rho_F:X \to \C$ such that 
$$\int_E \rho_F~d\mu = \Lambda(\chi_{\pi^{-1}E}F)$$
for Borel $E \subset X$. Note also that the functions $x \mapsto \eta^j_x(F_x)$ converge in measure to $\rho_F$. So after passing to a subsequence if necessary, we may assume that $\eta^j_x(F_{x})$ converges to $\rho_{F}(x)$ pointwise a.e. as $j\to\infty$ and for every $F \in \Omega$. However this implies that $\{\eta^j_x\}$ converges in the weak* topology on $C_0(B_x)^*$ as $j\to\infty$ (for a.e. $x$) (since $\{F_x\}_{F\in \Omega}$ is dense in $C_0(B_x)$). Moreover, since each $B_x$ is compact, the limiting measure, denoted $\lambda_x$, is a Borel probability measure. Thus we have obtained probability measures $\lambda_x$ such  that $\eta^j_x(F_x) \to \lambda_x(F_x)  = \rho_F(x)$ for a.e. $x$. In particular the field $\lambda$ is in $\Prob(\pi)$ and $\int \lambda_x(F_x)~d\mu(x)=\Lambda(F)$ for every $F \in C_0(\pi)$ which, by injectivity, implies that $\Psi(\lambda)=\Lambda$. This proves that the image of $\Psi$ is closed and therefore $\Prob(\pi)$ is compact, as required.

\end{proof}

Recall that if $\cC$ is a convex subspace of a Banach space then a point $x\in \cC$ is {\bf extreme} if and only if there does not exist element $y,z \in \cC$ and $t\in (0,1)$ such that $y\ne z$ and $x=ty+(1-t)z$. A {\bf Dirac measure} is a probability measure whose support contains only one element.

\begin{lem}\label{lem:extreme}
Assume that $B$ is a separable bundle and for each $x\in X$, $B_x$ is compact and Hausdorff. Then $\beta \in \Prob(\pi)$ is extreme if and only if for a.e. $x$, $\beta_x$ is a Dirac measure on $B_x$. 
\end{lem}

\begin{proof}
Clearly, if $\beta \in \Prob(\pi)$ is such that $\beta(x)$ is a Dirac measure for a.e. $x$ then $\beta$ is extreme.

On the other hand suppose $\nu \in \Prob(\pi)$ and if $Y$ is the set of all $x\in X$ such that $\nu_x$ is not a Dirac measure then $\mu(Y)>0$. It suffices to show $\nu$ is not extreme. 

Let $\supp(\nu_x)$ denote the support of $\nu_x$ and $\diam(\supp(\nu_x))$ its diameter. Because $\nu_x$ is not a Dirac measure for $x\in Y$, there exists a number $r>0$ such that if $Y_r=\{y\in Y:~\diam(\supp(\nu_x))>r\}$ then $\mu(Y_r)>0$. 

Because the bundle $\pi:B\to X$ is separable there exists a Borel section $\sigma:X \to B$ such that
$$d(\sigma(x), \supp(\nu_x)) \le r/10$$
for a.e. $x\in Y_r$. Let $N_{r/3}(\sigma(x))$ denote the closed $r/3$-neighborhood of $\sigma(x)$ in $B_x$. If $x\in Y_r$ then $N_{r/3}(\sigma(x))$ contains a nonempty open subset of $\supp(\nu_x)$ and its complement also contains a nontrivial open subset of $\supp(\nu_x)$. Therefore
$$0<\nu_x(N_{r/3}(\sigma(x)))<1.$$
So there exists a $0<t<1$ such that if $Z_t=\{x\in Y_r:~ t<\nu_x( N_{r/3}(\sigma(x)) ) <\frac{1}{1+t}\}$ then $\mu(Z_t)>0$. 

 Given any subset $C \subset B_x$ with $\nu_x(C)>0$, let $\nu_x\resto C$ denote the probability measure obtained by restricting $\nu_x$ to $C$ and normalizing so that $\nu_x\resto C$ is a probability measure.

Define $\nu^1,\nu^2 \in \Prob(\pi)$ as follows. For $y \notin Z_t$, let $\nu^1_y=\nu^2_y=\nu_y$. 

 For $x\in Y_r$, let
$$\nu^1_x =  (1+t) \nu_x(N_{r/3}(\sigma x)) \left[ \nu_x\resto N_{r/3}(\sigma x)\right] + \big(1-(1+t)\nu_x(N_{r/3}(\sigma x)) \big)\left[  \nu_x\resto(B_x\setminus N_{r/3}(\sigma x))\right]$$
$$\nu^2_x = (1-t)\nu_x(N_{r/3}(\sigma x)) \left[ \nu_x\resto N_{r/3}(\sigma x) \right]+ \big(1-(1-t)\nu_x(N_{r/3}(\sigma x)) \big)\left[  \nu_x\resto(B_x\setminus N_{r/3}(\sigma x) )\right].$$
Observe that $\nu^1 \ne \nu^2$ and yet 
$$\nu = \frac{\nu^1 + \nu^2}{2}.$$
So $\nu$ is not extremal.

\end{proof}

%\begin{lem}
%$\Prob(\pi)$ is separable.
%\end{lem}

%\begin{proof}
%Let $\Omega \subset C_0(\pi)$ be as in Lemma \ref{lem:sep}. 

%\end{proof}

\begin{lem}\label{lem:extreme2}
Let $\Prob^{ex}(\pi) \subset \Prob(\pi)$ denote the subspace of extreme points. Then $\Prob^{ex}(\pi)$ is a dense $G_\delta$ subset of $\Prob(\pi)$.
\end{lem}

\begin{proof}
By Lemma \ref{lem:fields} there exists a metric $d_{\Prob(\pi)}$ on $\Prob(\pi)$. For $n\in \N$, let $F_n$ denote the set of all $\nu \in \Prob(\pi)$ such that there exist $\nu^1,\nu^2 \in \Prob(\pi)$ such that $d_{\Prob(\pi)}(\nu^1,\nu^2) \ge 1/n$ and $\nu = \frac{\nu^1+\nu^2}{2}$. Then $F_n$ is closed in $\Prob(\pi)$ and $\Prob^{ex}(\pi)$ is the complement of $\cup_{n\in \N} F_n$. This proves $\Prob^{ex}(\pi)$ is a $G_\delta$. 

To prove that $\Prob^{ex}(\pi)$ is dense in $\Prob(\pi)$ it suffices to show that the closure of $\Prob^{ex}(\pi)$ is convex. So let $\nu,\eta$ be in the closure of  $\Prob^{ex}(\pi)$. It suffices to show $(1/2)(\nu+\eta)$ is in the closure of $\Prob^{ex}(\pi)$. 

Since $\nu,\eta$ are in the closure, there exist extremal measures $\nu^i,\eta^i$ with $\nu^i \to \nu$ and $\eta^i \to \eta$ in $\Prob(\pi)$. Since $(1/2)(\nu^i +\eta^i)$ converges to $(1/2)(\nu+\eta)$ it suffices to show that $(1/2)(\nu^i +\eta^i)$ is in the closure of $\Prob^{ex}(\pi)$. So we have reduced the problem to showing: if $\nu,\eta \in \Prob^{ex}(\pi)$ then $(1/2)(\nu+\eta)$ is in the closure of $\Prob^{ex}(\pi)$.

Because $(X,\mu)$ is a standard nonatomic probability space, there exists a sequence $\{E_n\}$ of measurable sets $E_n \subset X$ such that $\mu(E_n)=1/2$ for all $n$ and if $A \subset X$ is any measurable subset then
$$\lim_{n\to\infty} \mu(E_n \cap A)=\mu(A)/2.$$
Define $\beta^n \in \Prob(\pi)$ by $\beta^n_x=\nu_x$ if $x\in E_n$ and $\beta^n_x=\eta_x$ if $x\notin E_n$. Since $\beta^n_x$ is Dirac for a.e. $x$, $\beta^n \in \Prob^{ex}(\pi)$. It is an easy exercise to show $\beta^n \to (1/2)(\nu + \eta)$ as $n\to\infty$ (with respect to the weak* topology on $\Prob(\pi)$).

\end{proof}

The reader is cautioned here that $\Prob(\pi)$ is not in general a simplex (unless $\mu$ is a Dirac measure). That is, it is not necessarily true that every $\nu \in \Prob(\pi)$ can be decomposed as a convex integral of extreme points {\em uniquely}. For example, suppose $(C,d_C)$ is a metric space and $B=X\times C$, $\pi:B\to X$ is the usual projection map and $d:B*B \to \R$ is the metric $d((x,c),(x,c'))=d_C(c,c')$. In this case, $\Prob(\pi)$ can be identified with the set of (a.e. equivalence classes of) Borel maps from $X$ into $\Prob(C)$, the space of regular Borel probability measures on $C$. If $\sigma_1,\sigma_2:X \to C$ are two Borel maps such that $\sigma_1(x)\ne \sigma_2(x)$ for a.e. $x$, $Y \subset X$ has measure $1/2$ and $\nu^i_x=\delta_{\sigma_i(x)}$ for $x\in Y$, $\nu^i_x=\delta_{\sigma_{i+1}(x)}$ for $x\notin Y$ (indices mod $2$) then $\nu^1,\nu^2$ are both extreme and
$$(1/2)(\nu^1+\nu^2) = (1/2)(\delta_{\sigma_1} + \delta_{\sigma_2})$$
are two different extremal decompositions of the same Borel field of probability measures. So $\Prob(\pi)$ is not a simplex. 

Now suppose $\cR \subset X \times X$ is a discrete Borel equivalence relation on $X$ and $\{\alpha(x,y):~(x,y) \in \cR\}$ is a Borel action on $B$ by homeomorphisms. To be precise, this means that: 
\begin{itemize}
\item $\alpha(x,y):B_y \to B_x$ is a homeomorphism;
\item $\{ (p,q) \in B \times B:~ \alpha(\pi(p),\pi(q))(q) = p\}$ is Borel.
\end{itemize}
Let $[\cR]$ denote the full group of $\cR$ with the uniform topology (see \S \ref{sec:full-group}). This group acts on $\Prob(\pi)$ by $(f\nu)_{fx} = \alpha(fx,x)_*\nu_x$.

\begin{thm}\label{thm:jcontinuous}
The action of $[\cR]$ on $\Prob(\pi)$ is jointly continuous.
\end{thm}

\begin{proof}
Let $\{f_i\}_{i\in \N} \subset [\cR]$ converge to $f \in [\cR]$ and $\{\eta^i\}_{i\in \N} \subset \Prob(\pi)$ converge to $\eta \in \Prob(\pi)$. Let $F \in C_0(\pi)$. Observe:
\begin{eqnarray*}
\lim_{i\to\infty} \int (f_i\eta^i)_x(F_x)~d\mu(x)&=& \lim_{i\to\infty} \int (\alpha(x,f_i^{-1}x)_*\eta^i_{f^{-1}_ix})(F_x)~d\mu(x) \\
%&=& \lim_{i\to\infty} \int (\alpha(x,f_i^{-1}x)_*\eta^i_{f^{-1}_ix})(F_x)~d\mu(x) \\
&=& \lim_{i\to\infty} \int (\alpha(f_ix,x)_*\eta^i_{x})(F_{f_ix})~d\mu(x) \\
&=& \lim_{i\to\infty} \int \eta^i_{x}(F_{f_ix}\circ \alpha(f_ix,x) )~d\mu(x).
\end{eqnarray*}
Because $f_i \to f$ in uniformly 
$$\lim_{i\to\infty} \mu(\{x\in X:~F_{f_ix}\circ \alpha(f_ix,x) = F_{fx}\circ \alpha(fx,x)\}) = 1.$$
Since $F$ is essentially bounded the previous limit equals 
$$=\lim_{i\to\infty} \int \eta^i_{x}(F_{fx}\circ \alpha(fx,x) )~d\mu(x).$$
Since $\eta^i\to\eta$ as $i\to\infty$ we have
$$=\lim_{i\to\infty} \int \eta_{x}(F_{fx}\circ \alpha(fx,x) )~d\mu(x)=\lim_{i\to\infty} \int (f\eta)_x(F_x)~d\mu(x).$$
Because $F$ is arbitrary, this proves $\lim_{i\to\infty} f_i\eta^i = f\eta$. So $[\cR] \cc \Prob(\pi)$ is jointly continuous.
\end{proof}

%\section{Things to check}

%Finish the Organization.

%we could either cut Svarc-Milner or prove it in a more general context. It has nothing to do with hyperbolicity.

\bibliography{hyperbolic}

\def\cprime{$'$} \def\cprime{$'$} \def\cprime{$'$} \def\cprime{$'$}
  \def\cprime{$'$} \def\cprime{$'$} \def\cprime{$'$}
\begin{thebibliography}{CFRW10}

\bibitem[Ada94]{adams-indecomp}
S.~Adams.
\newblock Indecomposability of equivalence relations generated by word
  hyperbolic groups.
\newblock {\em Topology}, 33(4):785--798, 1994.

\bibitem[Ada96]{adams-reduction}
Scot Adams.
\newblock Reduction of cocycles with hyperbolic targets.
\newblock {\em Ergodic Theory Dynam. Systems}, 16(6):1111--1145, 1996.

\bibitem[AH11]{anderegg-henry}
Martin Anderegg and Philippe P.~A. Henry.
\newblock Actions of amenable equivalence relations on {CAT}(0) fields.
\newblock {\em to appear in Ergodic Theory and Dynamical Systems}, 2011.

\bibitem[BC14]{boutonnet-carderi}
R\'emi Boutonnet, , and Alessandro Carderi.
\newblock Maximal amenable subalgebras of von neumann algebras associated with
  hyperbolic groups.
\newblock {\em arXiv preprint arXiv:1310.5864}, 2014.

\bibitem[BH99]{bridson-haefliger-book}
Martin~R. Bridson and Andr{\'e} Haefliger.
\newblock {\em Metric spaces of non-positive curvature}, volume 319 of {\em
  Grundlehren der Mathematischen Wissenschaften [Fundamental Principles of
  Mathematical Sciences]}.
\newblock Springer-Verlag, Berlin, 1999.

\bibitem[CFRW10]{CFRW-2010}
Jan Cameron, Junsheng Fang, Mohan Ravichandran, and Stuart White.
\newblock The radial masa in a free group factor is maximal injective.
\newblock {\em J. Lond. Math. Soc. (2)}, 82(3):787--809, 2010.

\bibitem[CFW81]{CFW81}
A.~Connes, J.~Feldman, and B.~Weiss.
\newblock An amenable equivalence relation is generated by a single
  transformation.
\newblock {\em Ergodic Theory Dynamical Systems}, 1(4):431--450 (1982), 1981.

\bibitem[CL12]{csoka-lippner}
Endre Csoka and Gabor Lippner.
\newblock Factor of iid perfect matchings.
\newblock {\em submitted, arXiv:1211.2374}, 2012.

\bibitem[CMTD14]{conley-brooks}
Clinton Conley, Andrew Marks, and Robin Tucker-Drob.
\newblock Brooks's theorem for measurable colorings.
\newblock {\em preprint}, 2014.

\bibitem[CS13]{chifan-sinclair-I}
Ionut Chifan and Thomas Sinclair.
\newblock On the structural theory of {${\rm II}_1$} factors of negatively
  curved groups.
\newblock {\em Ann. Sci. \'Ec. Norm. Sup\'er. (4)}, 46(1):1--33 (2013), 2013.

\bibitem[FM77]{feldman-moore-1}
Jacob Feldman and Calvin~C. Moore.
\newblock Ergodic equivalence relations, cohomology, and von {N}eumann
  algebras. {I}.
\newblock {\em Trans. Amer. Math. Soc.}, 234(2):289--324, 1977.

\bibitem[Gab00]{gaboriau-cost}
Damien Gaboriau.
\newblock Co\^ut des relations d'\'equivalence et des groupes.
\newblock {\em Invent. Math.}, 139(1):41--98, 2000.

\bibitem[Gab05]{gaboriau-2005}
D.~Gaboriau.
\newblock Examples of groups that are measure equivalent to the free group.
\newblock {\em Ergodic Theory Dynam. Systems}, 25(6):1809--1827, 2005.

\bibitem[GL09]{gaboriau-lyons}
Damien Gaboriau and Russell Lyons.
\newblock A measurable-group-theoretic solution to von {N}eumann's problem.
\newblock {\em Invent. Math.}, 177(3):533--540, 2009.

\bibitem[GP07]{giordano-pestov-extreme-amenability}
Thierry Giordano and Vladimir Pestov.
\newblock Some extremely amenable groups related to operator algebras and
  ergodic theory.
\newblock {\em J. Inst. Math. Jussieu}, 6(2):279--315, 2007.

\bibitem[Hjo08]{hjorth-nontreeability}
Greg Hjorth.
\newblock Non-treeability for product group actions.
\newblock {\em Israel J. Math.}, 163:383--409, 2008.

\bibitem[Hou14]{houdayer-2014}
Cyril Houdayer.
\newblock A class of {$\textrm{II}_1$} factors with an exotic abelian maximal
  amenable subalgebra.
\newblock {\em Trans. Amer. Math. Soc.}, 366(7):3693--3707, 2014.

\bibitem[JKL02]{MR1900547}
S.~Jackson, A.~S. Kechris, and A.~Louveau.
\newblock Countable {B}orel equivalence relations.
\newblock {\em J. Math. Log.}, 2(1):1--80, 2002.

\bibitem[Kai04]{kaimanovich-boundary-amenability}
Vadim~A. Kaimanovich.
\newblock Boundary amenability of hyperbolic spaces.
\newblock In {\em Discrete geometric analysis}, volume 347 of {\em Contemp.
  Math.}, pages 83--111. Amer. Math. Soc., Providence, RI, 2004.

\bibitem[Kec10]{Kechris-global-aspects}
Alexander~S. Kechris.
\newblock {\em Global aspects of ergodic group actions}, volume 160 of {\em
  Mathematical Surveys and Monographs}.
\newblock American Mathematical Society, Providence, RI, 2010.

\bibitem[KS96]{kapovich-short-1996}
Ilya Kapovich and Hamish Short.
\newblock Greenberg's theorem for quasiconvex subgroups of word hyperbolic
  groups.
\newblock {\em Canad. J. Math.}, 48(6):1224--1244, 1996.

\bibitem[KST99]{KST99}
A.~S. Kechris, S.~Solecki, and S.~Todorcevic.
\newblock Borel chromatic numbers.
\newblock {\em Adv. Math.}, 141(1):1--44, 1999.

\bibitem[LN11]{lyons-nazarov}
Russell Lyons and Fedor Nazarov.
\newblock Perfect matchings as {IID} factors on non-amenable groups.
\newblock {\em European J. Combin.}, 32(7):1115--1125, 2011.

\bibitem[Ol{\cprime}91]{olshankii-book}
A.~Yu. Ol{\cprime}shanski{\u\i}.
\newblock {\em Geometry of defining relations in groups}, volume~70 of {\em
  Mathematics and its Applications (Soviet Series)}.
\newblock Kluwer Academic Publishers Group, Dordrecht, 1991.
\newblock Translated from the 1989 Russian original by Yu. A. Bakhturin.

\bibitem[OW80]{OW80}
Donald~S. Ornstein and Benjamin Weiss.
\newblock Ergodic theory of amenable group actions. {I}. {T}he {R}ohlin lemma.
\newblock {\em Bull. Amer. Math. Soc. (N.S.)}, 2(1):161--164, 1980.

\bibitem[Oza04]{ozawa-solid}
Narutaka Ozawa.
\newblock Solid von {N}eumann algebras.
\newblock {\em Acta Math.}, 192(1):111--117, 2004.

\bibitem[PP00]{pemantle-peres-2000}
Robin Pemantle and Yuval Peres.
\newblock Nonamenable products are not treeable.
\newblock {\em Israel J. Math.}, 118:147--155, 2000.

\bibitem[PT11]{peterson-thom-group-cocycles}
Jesse Peterson and Andreas Thom.
\newblock Group cocycles and the ring of affiliated operators.
\newblock {\em Invent. Math.}, 185(3):561--592, 2011.

\bibitem[She06]{shen-2006}
Junhao Shen.
\newblock Maximal injective subalgebras of tensor products of free group
  factors.
\newblock {\em J. Funct. Anal.}, 240(2):334--348, 2006.

\bibitem[V{\"a}i05]{vaisala-hyperbolic}
Jussi V{\"a}is{\"a}l{\"a}.
\newblock Gromov hyperbolic spaces.
\newblock {\em Expo. Math.}, 23(3):187--231, 2005.

\bibitem[Zim84]{zimmer-1984}
Robert~J. Zimmer.
\newblock {\em Ergodic theory and semisimple groups}, volume~81 of {\em
  Monographs in Mathematics}.
\newblock Birkh\"auser Verlag, Basel, 1984.

\end{thebibliography}
\bibliographystyle{alpha}

\end{document}